\documentclass[a4paper, 12pt, reqno]{amsart}
\usepackage{amsmath,amsthm,amssymb,latexsym,mathrsfs,stmaryrd,color,mathtools}
\usepackage[all]{xy}
\usepackage{url}
\usepackage{geometry}
\usepackage{enumerate}
\usepackage[utf8]{inputenc}
\usepackage[T1]{fontenc}

\usepackage{relsize}
\usepackage[bbgreekl]{mathbbol}
\usepackage{amsfonts}

\DeclareSymbolFontAlphabet{\mathbb}{AMSb}
\DeclareSymbolFontAlphabet{\mathbbl}{bbold}

\usepackage{enumitem}
\usepackage[colorlinks=true, hyperindex, linkcolor=magenta, pagebackref=false, citecolor=cyan, pdfpagelabels]{hyperref}
\usepackage[capitalize]{cleveref}
\setlist[enumerate]{itemsep=2pt,parsep=2pt,before={\parskip=2pt}}

\newcommand{\cosimp}[3]{\xymatrix@1{#1 \ar@<.4ex>[r] \ar@<-.4ex>[r] & {\ }#2 \ar@<0.8ex>[r] \ar[r] \ar@<-.8ex>[r] & {\ } #3 \ar@<1.2ex>[r] \ar@<.4ex>[r] \ar@<-.4ex>[r] \ar@<-1.2ex>[r] & \cdots }}

\newcommand{\adjunction}[4]{\xymatrix@1{#1{\ } \ar@<0.3ex>[r]^{ {\scriptstyle #2}} & {\ } #3 \ar@<0.3ex>[l]^{ {\scriptstyle #4}}}}

\numberwithin{equation}{section}

\DeclareMathOperator{\Shv}{Shv}
\DeclareMathOperator{\et}{\acute{e}t}

\DeclareMathOperator{\Zar}{Zar}
\DeclareMathOperator{\Hom}{Hom}

\DeclareMathOperator{\Spf}{Spf}
\DeclareMathOperator{\Spa}{Spa}
\DeclareMathOperator{\Spec}{Spec}

\DeclareMathOperator{\ad}{ad}

\DeclareMathOperator{\proet}{pro\acute{e}t}

\DeclareMathOperator{\cont}{cont}
\DeclareMathOperator{\Supp}{Supp}
\DeclareMathOperator{\Cone}{Cone}

\DeclareMathOperator{\qcoh}{qcoh}

\DeclareMathOperator{\Simpl}{Simpl}

\newtheorem{theorem}{Theorem}[section]
\newtheorem*{theorem*}{Theorem}
\newtheorem*{definition*}{Definition}
\newtheorem{proposition}[theorem]{Proposition}
\newtheorem{lemma}[theorem]{Lemma}
\newtheorem{corollary}[theorem]{Corollary}

\theoremstyle{definition}
\newtheorem{definition}[theorem]{Definition}

\newtheorem{remark}[theorem]{Remark}

\newtheorem{construction}[theorem]{Construction}

\crefname{assumption}{assumption}{assumptions}
\crefname{construction}{construction}{constructions}

\newcommand{\baseof}{\overset{\circ}}

\title[$p$-adic Cartier isomorphism for semistable formal schemes]{A $p$-adic Cartier isomorphism between the $A_{\inf}$-cohomology and de Rham-Witt complexes for semistable formal schemes}
\author{Kensuke Aoki}
\address{Department of Mathematics, Faculty of Science, Kyoto University}
\email{artmide13@gmail.com}

\begin{document}

\begin{abstract} 
\v{C}esnavi\v{c}ius-Koshikawa constructed the $A_{\inf}$-cohomology theory for semistable formal schemes over the ring of integers of $\mathbb{C}_p$. 
We prove the $p$-adic Cartier isomorphism between the $A_{\inf}$-cohomology and de Rham-Witt complexes for semistable formal schemes, extending the result of Bhatt-Morrow-Scholze in the smooth case. 
\end{abstract}

\maketitle

\tableofcontents

\section{Introduction}

We fix a prime number $p$. Let $C$ be the complete algebraic closure of $W(k)[ \tfrac{1}{p}]$ where $k$ is an algebraically closed field of characteristic $p > 0$ and let $\mathcal{O}_C$ be the ring of integral elements of $C$. 

Bhatt-Morrow-Scholze (\cite{BMS18}) have constructed the $A_{\inf}$-cohomology theory for smooth $p$-adic formal schemes over $\mathcal{O}_C$ which unifies various $p$-adic cohomology theories including \'etale, de Rham, and crystalline cohomology. It is defined on Fontaine's ring $A_{\inf} \coloneqq A_{\inf} (\mathcal{O}_C)$ and we may consider variants on the reduction for the de Rham specialization map $\theta \colon A_{\inf} \to \mathcal{O}_C$ and Hodge-Tate specialization map $\tilde{\theta} = \theta \varphi^{-1} \colon A_{\inf} \to \mathcal{O}_C$. The prismatic cohomology in \cite{BS19} is a more general concept and $A_{\inf}$-cohomology is seen as the prismatic cohomology on the base prism $(A_{\inf} , (\xi) \coloneqq \ker \theta)$ up to the Frobenius automorphism. 

There is another direction of generalization. \v{C}esnavi\v{c}ius and Koshikawa have constructed the variant of the $A_{\inf}$-cohomology theory for semistable log formal schemes. For a morphism $A \to B$ of $p$-adically complete rings, we call it $p$-adically \'etale if the maps $A/p^m \to B/p^m$ is \'etale for any $m \geq 1$. We let $R$ be a semistable $p$-adically complete algebra over $\mathcal{O}_C$, i.e., $R$ admits a $p$-adically \'etale morphism over 
\[
R^{\Box} \coloneqq \mathcal{O}_C\{ t_0,\ldots,t_r,t_{r+1}^{\pm 1},\ldots,t_d^{\pm 1}\}/(t_0\cdots t_r - p^q). 
\]
We let $\mathfrak{X}$ be a $p$-adic formal scheme that admits a Zariski covering which consist of formal spectra of semistable $p$-adic formal algebras and let $\mathfrak{X}_C^{\ad}$ be the adic generic fiber of $\mathfrak{X}$. We also let $\nu \colon (\mathfrak{X}_C^{\ad})_{\proet} \to \mathfrak{X}_{\Zar}$ be the natural morphism from the pro-\'etale site of $\mathfrak{X}_C^{\ad}$ to the Zariski site of $\mathfrak{X}$
(see \cite[Section 5.1]{BMS18} for details). Let $\widehat{\mathcal{O}}_{\Spf(R)_C^{\ad}}^+$, $\widehat{\mathcal{O}}_{\mathfrak{X}_C^{\ad}}^+$ be the completed integral structure sheaves on $(\Spf(R)_C^{\ad})_{\proet}$, $(\mathfrak{X}_C^{\ad})_{\proet}$, respectively (see \cite[Definition 5.4]{BMS18}). We define complexes
\begin{align*}
\widetilde{\Omega}_R &\coloneqq 
L\eta_{\zeta_p - 1} R\Gamma_{\proet} (\Spf(R)_C^{\ad}, \widehat{\mathcal{O}}_{\Spf(R)_C^{\ad}}^+), \\ 
\widetilde{\Omega}_{\mathfrak{X}} &\coloneqq 
L\eta_{\zeta_p - 1} R\nu_{\ast} \widehat{\mathcal{O}}_{\mathfrak{X}_C^{\ad}}^+,
\end{align*}
where $L\eta$ is the d\'ecalage functor, which is used in \cite[Section 6]{BMS18}. We also define complexes
\begin{align*}
\widetilde{W_n \Omega}_R &\coloneqq L\eta_{[\zeta_{p^n}] - 1} R\Gamma_{\proet} (\Spf(R)_C^{\ad}, W_n (\widehat{\mathcal{O}}_{\Spf(R)_C^{\ad}}^+)), \\ 
\widetilde{W_n \Omega}_{\mathfrak{X}} &\coloneqq L\eta_{[\zeta_{p^n}] - 1} R\nu_{\ast} W_n (\widehat{\mathcal{O}}_{\mathfrak{X}_C^{\ad}}^+).
\end{align*}
 
The aim of this paper is to show the following $p$-adic Cartier isomorphisms for the semistable formal schemes, which is the main theorem of this paper. We use the log de Rham-Witt complex $W_n \Omega^i$ of Matsuue \cite{Mat15} for the formulation of the theorem. We use the construction and \'etale base change of the log de Rham-Witt complexes. We mainly treat the local case, which is essential for overall arguments. The following comparison maps can be globally defined by the naturality of those of affine local claims. For a $p$-torsion free algebra $A$ on a topos, let $Q_A$ be the log structure $(A[1/p])^\ast \cap A \to A$ of $A$.

\begin{theorem}[Theorem \ref{O-logdRWC}]
\label{intro-O-logdrRWC}
There is a natural isomorphism
\[
H^i (\widetilde{W_n \Omega}_{\mathfrak{X}}) \cong 
\varprojlim_m W_n \Omega_{(\mathfrak{X}/p^m, Q_{\mathcal{O}_{\mathfrak{X}}})/(\mathcal{O}_C /p^m, Q_{\mathcal{O}_C})}^i \{ -i\} \eqqcolon W_n \Omega_{(\mathfrak{X}, Q_{\mathcal{O}_{\mathfrak{X}}})/(\mathcal{O_C}, Q_{\mathcal{O}_C})}^{i, \cont} \{ -i\}.
\]
Moreover, for $n = 1$, the map is compatible with the isomorphism in Theorem \ref{intro-CKisom}. 
\end{theorem}
We remark that the case of special fibers of the theorem is proved by Z.\ Yao (\cite{Yao18}) in the setting of the extension of the construction of the de Rham-Witt complexes of Bhatt-Lurie-Mathew (\cite{BLM21}) to the semistable case. In the special case for $n = 1$, \v{C}esnavi\v{c}ius-Koshikawa have shown the existence of the Hodge-Tate comparison map for $\widetilde{\Omega}_{\mathfrak{X}}$. 
\begin{theorem}[\v{C}esnavi\v{c}ius-Koshikawa]
\label{intro-CKisom}
There is a natural isomorphism 
\[
H^i (\widetilde{\Omega}_{\mathfrak{X}}) \cong 
\varprojlim_m \Omega_{(\mathfrak{X}/p^m, Q_{\mathcal{O}_{\mathfrak{X}}}) / (\mathcal{O}_C / p^m,
Q_{\mathcal{O}_C})}^i \{-i\} \eqqcolon
\Omega_{(\mathfrak{X}, Q_{\mathcal{O}_{\mathfrak{X}}}) / (\mathcal{O}_C,
Q_{\mathcal{O}_C})}^{i, \cont} \{-i\}.
\]
\end{theorem}
The theorem is shown from the smooth case of \cite{BMS18} and the Grothendieck existence theorem generalized by Fujiwara-Kato in a non-noetherian setting in \cite{CK19}. However, the isomorphism is not explicitly constructed in the proof. We also give another proof of the theorem by constructing the map explicitly using the log cotangent complexes in the article. We begin to explain the explicit proof of Theorem \ref{intro-CKisom}. The log cotangent complexes are defined in two ways: Olsson's complex is defined for morphisms of fine log schemes and Gabber's complex is for ones of prelog rings on a topos. We consider the log cotangent complex in the sense of Gabber for the morphism $(\mathcal{O}_C, Q_{\mathcal{O}_C}) \to (\widehat{\mathcal{O}}_{\mathfrak{X}_C^{\ad}}^+, Q_{\widehat{\mathcal{O}}_{\mathfrak{X}_C^{\ad}}^+})$ on the pro-\'etale topos $\Shv((\mathfrak{X}_C^{\ad})_{\proet})$. Let $\widehat{\mathbb{L}} \coloneqq R\varprojlim_n \mathbb{L}/p^n$ be the derived $p$-adic completion of the log cotangent complex $\mathbb{L}$. Let $R_{\infty}$ be a naturally constructed pro-\'etale perfectoid cover of $R$. The essential part of our proof is the following vanishing.  
\begin{proposition}[Proposition \ref{Rinfty-LG}]
We have the vanishing 
\[
\widehat{\mathbb{L}}_{(R_{\infty}, Q_{R_{\infty}})/(\mathcal{O}_C, Q_{\mathcal{O}_C})}^G \cong 0. 
\]
\end{proposition}
After the first draft of this paper was written, the author learned from T.\ Koshikawa this proposition is well-known to the experts (see \cite[Corollary 7.11]{Bha12}, \cite[Corollary 2.29]{KY22}). 

For our proof, we construct a direct system of integral morphisms of fine log rings whose colimit is the morphism in the proposition. Then, the comparison theorem of the two log cotangent complexes (\cite[Corollary 8.34]{Ols05}) enables us to use the definition of Olsson for computing the complex. The fiber product of the log semistable $p$-adically complete algebras on the algebraic stack $\mathcal{L}og_{(\mathcal{O}, Q_{\mathcal{O}})}$, where $\mathcal{O}$ is a discrete valuation ring, is an algebraic space which is realized as a quotient of a smooth scheme for an \'etale equivalence relation. Then the \'etale invariance of the cotangent complex assures that we may compute complex for the \'etale extension and the colimit of them turns out to be $0$. Therefore, the proposition follows by faithfully flat descent.

We come back to Theorem \ref{intro-O-logdrRWC}. In the smooth case, the theorem is proved in \cite[Section 11]{BMS18}. For the proof of the theorem, we may assume $R$ is the ring $R^{\Box}$ by the assumption of $R$ and the \'etale invariance of the log de Rham-Witt complexes (\cite[Proposition 3.7]{Mat15}, Proposition \ref{et-bc-ldRW}). We give a set of explicit generators of the log de Rham-Witt complex for a ring $R_+^{\Box}$, whose Zariski localization is $R^{\Box}$, as in \cite[Proposition 2.17]{LZ03}. Then we have a set of generators for $R^{\Box}$ by localization and we show the map of the theorem in the case of special fiber is an isomorphism. So the method of \cite[Section 11]{BMS18} works in the remaining argument. We remark that we may have another proof of log-crystalline comparison theorem \cite[Corollary 5.43]{CK19} of $A_{\inf}$-cohomology by the log-crystalline comparison of log de Rham-Witt complexes \cite[Theorem 7.9]{Mat15} and the proof of \cite[Theorem 14.1]{BMS18}.

\subsection{Notation} 
Throughout the paper, we fix a prime number $p$. 
Without any assumptions, rings are always commutative. 
For a ring $A$, let $W(A)$ (resp.\ $W_n (A), (n \geq 1)$) denote the ring of Witt vectors of $A$ (resp.\ the $n$-th truncation of the ring of Witt vectors of $A$). 
We denote by $D(A)$ the derived category of the abelian category of $A$-modules. We let $K \in D(A)$. For $n \in \mathbb{Z}$, we denote by $\tau_{\geq n} K, \tau_{\leq n} K$ the canonical truncations of $K$. In the case $A = \mathbb{Z}$, the \emph{derived} $p$-\emph{adic completion of} $K$ is the derived limit $R\varprojlim_n (K \otimes_{\mathbb{Z}}^L \mathbb{Z}/p^n)$. We often use the commutativity of $R\varprojlim$ and $Rf_{\ast}$ where $f$ is a morphism of ringed topoi (\cite[Tag 0A07]{SP}), the projection formula (\cite[Tag 0944]{SP}).
Let $\mathcal{A}$ be an abelian category, $M$ be an object in $A$ and $f_i \colon M \to M$ for $i = 1, \ldots, d$ be commuting endomorphisms.  The Koszul complex $K_M (f_1, \ldots, f_d)$ is defined as 
\begin{equation}
M \xrightarrow{(f_1, \ldots, f_d)} \bigoplus_{1 \leq i \leq d} M \to \bigoplus_{1 \leq i_1 < i_2 \leq d} M \to \cdots \to \bigoplus_{1 \leq i_1 < \ldots < i_k \leq d} M \to \cdots
\end{equation}
where $M_I = M_{\{i_1, \ldots, i_k\}} \to M_{\{j_1, \ldots, j_{k+1}\}} = M_J$ is given by $(-1)^{m-1} f_{j_m}$ if $I \subset J, J \backslash I = \{ j_m\}$, otherwise $0$. If $\mathcal{A}$ is the category of $A$-modules for a fixed ring $A$, we have the presentation 
\begin{equation}
K_M (f_1, \ldots, f_d) = M \otimes_{A[f_1, \ldots, f_d]} \bigotimes_{i=1}^d (\mathbb{Z}[f_1, \ldots, f_d] \xrightarrow{f_i} \mathbb{Z} [f_1, \ldots, f_d]).
\end{equation}
We use the notation of log schemes in \cite{Kat88}. For a log scheme $X$, we denote by $\baseof{X}$ the underlying scheme of $X$ and by $\mathcal{M}_X \to \mathcal{O}_X$ the log structure of $X$. We also set $\overline{\mathcal{M}}_X \coloneqq \mathcal{M}_X / \mathcal{O}_X^\ast$. 

\subsection*{Acknowledgements}
The author would like to thank Osamu Fujino, Tetsushi Ito and Teruhisa Koshikawa for advice on preliminary versions of this paper.

\section{Logarithmic cotangent complexes}
\label{sec;LogCotComp}

\subsection{Logarithmic derivation}

Let $\mathcal{R}$ be a topos. 

\begin{definition}\label{def:PrelogRingedStronTopos}
A \emph{prelog ring in $\mathcal{R}$} or a \emph{prelog ringed structure} is a triple $(A, M_1, \alpha)$ where $A$ is a commutative ring with unity in $\mathcal{R}$, $M_1$ is a commutative monoid with unity in $\mathcal{R}$, and $\alpha \colon M_1 \to A$ is a morphism of monoids. A prelog ringed structure $(A, M_1, \alpha)$ is a \emph{log ring} if the induced morphism $\alpha^{-1} (A^\ast) \to A$ is an isomorphism.
\end{definition}

Let $\theta \colon (A, M_1) \to (B, M_2)$ be a morphism of prelog rings. 

\begin{definition}
Let $\mathcal{R}$ be a punctual topos $\{ \ast \}$ and $H$ be an $B$-module. \emph{A log derivation on $(B, M_2)/(A, M_1)$} is a pair
\[
(D \colon B \to H, \ \delta \colon M_2 \to H)
\]
consisting of an $A$-linear derivation $D$ and a monoid homomorphism $\delta$ which satisfies the following conditions

\begin{enumerate}
\item $D(\beta (m)) = \beta (m) \delta (m)$ for any $m \in M_2$.
\item $\delta (\theta (n)) = 0$ for any $n \in M_1$. 
\end{enumerate}

We have the universal log derivation, which we denote by 
\[
(d \colon B \to \Omega_{(B, M_2)/(A, M_1)}^1, \ \delta \colon M_2 \to \Omega_{(B, M_2)/(A, M_1)}^1).
\]
In the case of general $\mathcal{R}$, the universal log derivation can be defined as the sheafification of the presheaf of those for $\mathcal{R} = \{ \ast \}$. We call them the \emph{module of logarithmic differentials}. 
\end{definition}

\subsection{Logarithmic cotangent complex of Olsson}

Let $f \colon X \to Y$ be a morphism of log schemes. Let $\Omega_{X/Y}^1$ denote the sheaf of logarithmic differentials of $f$. If $f$ locally has a chart, the notion of the module of logarithmic differentials naturally sheafified as in \cite[1.7]{Kat88} by \cite[Lemma IV.1.1.10]{Ogus}. 

For a fine log scheme $X$, we denote by $\mathcal{L}og_X$ the fppf-stack over $\overset{\circ}{X}$ whose objects are fine log schemes over $X$ and whose morphisms are strict morphism of $X$-log schemes \cite[Introduction]{Ols03}.
It is an algebraic stack locally of finite presentation over $\baseof{X}$;
see \cite[Theorem 1.1]{Ols03}.
For a morphism of fine log schemes $f \colon T \to X$, we denote by $\mathcal{L}_f \colon \overset{\circ}{T} \to \mathcal{L}og_X$ the associated $1$-morphism. To introduce logarithmic cotangent complexes of Olsson, we explain an equivalence of categories of projective systems in the derived categories of the lisse-\'etale topos and the \'etale topos on $X$. For a ringed topos $T$, we denote by $D'(T)$ the category of projective systems 
\[
K = (\cdots \to K^{\geq -n-1} \to K^{\geq -n} \to \cdots \to K^{\geq 0}), 
\]
in $D^{+} (T)$ in which the maps
\[
K^{\geq -n} \to \tau_{\geq -n} K^{\geq -n}, \quad 
\tau_{\geq -n} K^{\geq -n-1} \to \tau_{\geq -n} K^{\geq -n}
\]
are isomorphisms \cite[(1.8.5)]{Ols05}. Let $D'_{\qcoh} (T)$ be the full subcategory of $D'(T)$ consisting of complexes with quasi-coherent cohomology sheaves. For an algebraic stack $\mathcal{X}$, let $\mathcal{X}_{\textnormal{lis-\'et}}$ denote the lisse-\'etale topos of $\mathcal{X}$ \cite{LMB}, \cite{Ols05b}.
In \cite[17.3]{LMB}, for a morphism $\mathcal{X} \to \mathcal{Y}$ of algebraic stacks, the cotangent complex $\mathbb{L}_{\mathcal{X}/\mathcal{Y}}$ is defined to be an object in $D'_{\qcoh} (\mathcal{X}_{\textnormal{lis-\'et}})$.
As in \cite[Remark 3.3]{Ols05}, for a scheme $X$, the restriction functor 
\[
D'_{\qcoh} (X_{\textnormal{lis-\'et}}) \to D'_{\qcoh} (X_{\et})
\]
is an equivalence of categories, which is compatible with pullbacks. We now introduce Olsson's log cotangent complex.

\begin{definition}[see {\cite[Definition 3.2]{Ols05}}]
For a morphism $f \colon X \to Y$ of fine log schemes, the \emph{log cotangent complex of Olsson} $\mathbb{L}_{X/Y}$ of $f$ is defined to be the restriction of $\mathbb{L}_{\overset{\circ}{X}/\mathcal{L}og_Y}$ to the \'etale site of $\baseof{X}$. 
\end{definition}

\subsection{Logarithmic cotangent complex of Gabber}

Let $\mathcal{R}$ be a topos with enough points. 
We will only deal with cases where the assumption is satisfied (cf.\ \cite[p.2, (2)]{Sch13C}). 
For a category $\mathcal{C}$, let $\Simpl(\mathcal{C})$ denote the category of simplicial objects in $\mathcal{C}$.

For a fixed prelog ring $(A, M_1, \alpha)$, let $\mathcal{C}$ denote the category of morphisms $(A, M_1, \alpha) \to (B, M_2, \beta)$ of prelog ringed structures in $\mathcal{R}$, and $\mathcal{D}$ be the category $\mathcal{R} \times \mathcal{R}$ of pairs of objects of $\mathcal{R}$. 
There is an adjoint pair $(T, U)$ between $\mathcal{C}$ and $\mathcal{D}$ such that 
\[
U \colon \mathcal{C} \to \mathcal{D}, \ (M_2 \to B) \mapsto (M_2, B)
\]
and
\[
T \colon \mathcal{D} \to \mathcal{C}, \ (X, Y) \mapsto (\alpha \colon M_1 \oplus \mathbb{N}^X \to A\{ X \sqcup Y\})
\]
where for a set $S$, we write $A\{ S \}$ for the free algebra on $S$.

For any object $(B, M_2, \beta)$, its canonical free resolution as in \cite[I.1.5]{Ill71} using the functors $(T, U)$ gives the resulting simplicial object $P_{(A, M_1, \alpha)} (B, M_2, \beta) \in \Simpl(\mathcal{C})$
satisfying $P_{(A, M_1, \alpha)}^{n} (B, M_2, \beta) = (TU)^{n} (B, M_2, \beta)$.
Let $\overset{\circ}{P}_{(A, M_1, \alpha)} (B, M_2, \beta)$ be the underlying simplicial ring of $P_{(A, M_1, \alpha)} (B, M_2, \beta) \in \Simpl(\mathcal{C})$. Gabber's log cotangent complexes are defined as follows.
\begin{definition}[see {\cite[Definition 8.5]{Ols05}}]
For a morphism of prelog rings $(A, M_1, \alpha) \to (B, M_2, \beta)$,
the \emph{log cotangent complex of Gabber} is defined to be the chain complex of $B$-modules associated to the simplicial $B$-module 
\[
\mathbb{L}_{(B, M_2, \beta)/(A, M_1, \alpha)}^G
\coloneqq
\Omega_{P_{(A, M_1, \alpha)} (B, M_2, \beta)/(A, M_1, \alpha)}^1 \otimes_{\overset{\circ}{P}_{(A, M_1, \alpha)} (B, M_2, \beta)} B.
\]
For morphism of prelog ringed topoi with enough points $f \colon X \to Y$, we define $\mathbb{L}_{X/Y}^G$ to be the complex of $\mathcal{O}_X$-modules
\[
\mathbb{L}_{X/Y}^G \coloneqq \mathbb{L}_{(\mathcal{O}_X, M_X)/(f^{-1} \mathcal{O}_Y, f^{-1} M_Y)}^G
\]
on the underlying scheme of $X$. 
\end{definition}

\begin{lemma}
\label{comm-LG-fcolim}
Let $((A_i, M_i, \alpha_i) \to (B_i, M'_i, \beta_i))$ be a filtered system of prelog rings. Then we have 
\[
\varinjlim_i \mathbb{L}_{((A_i, M_i, \alpha_i) \to (B_i, M'_i, \beta_i))}^G \cong \mathbb{L}_{(\varinjlim_i ((A_i, M_i, \alpha_i) \to (B_i, M'_i, \beta_i)))}^G.
\]
\end{lemma}

\begin{proof}
The functors $T, U$ above and $((A, M_1, \alpha) \to (B, M_2, \beta)) \mapsto \Omega_{(B, M_2, \beta)/(A, M_1, \alpha)}^1$ commute with filtered colimits. Then the assertion follows. 
\end{proof}

\begin{lemma}
\label{comm-LG-Hi}
Let $\mathrm{Shv}(\mathcal{C})$ be a topos with enough points with site of definition $\mathcal{C}$ and $(A, M_1, \alpha) \to (B, M_2, \beta)$ be a homomorphism of sheaves of prelog rings on $\mathrm{Shv}(\mathcal{C})$. Then $H^i (\mathbb{L}_{(B, M_2, \beta)/(A, M_1, \alpha)}^G)$ is the sheaf associated to the presheaf
\[ U \mapsto H^i (\mathbb{L}_{(B(U), M_2 (U), \beta(U))/(A(U), M_1 (U), \alpha(U))}^G) \]
for $i \in \mathbb{Z}$. 
\end{lemma}

\begin{proof}
Let $\mathcal{C}'$ be the site which is endowed with the chaotic topology on the underlying category of $\mathcal{C}$, and let $\mathbb{L}^{\mathrm{pre}, G}, \mathbb{L}^G$ be the Gabber's log cotangent complex functors on $\mathcal{C}', \mathcal{C}$, respectively. If $\mathcal{C} = \mathcal{C}'$, the log cotangent complexes $\mathbb{L}^G$ commute with the functors taking sections by the construction of $T, U$, i.e., we have 
\[
\mathbb{L}_{(B, M_2, \beta)/(A, M_1, \alpha)}^{\mathrm{pre}, G} (U) = \mathbb{L}_{(B(U), M_2 (U), \beta(U))/(A(U), M_1 (U), \alpha(U))}^{\mathrm{pre}, G}
\]
on $\mathrm{Shv}(\mathcal{C}')$. Let $(f_{\ast}, f^{-1}) \colon \mathrm{Shv}(\mathcal{C}) \to \mathrm{Shv}(\mathcal{C}')$ be the natural morphism of topoi. It follows that $f^{-1}$ also commutes with $\mathbb{L}^G$, i.e., there is a natural quasi-isomorphism
\[
\mathbb{L}_{(B, M_2, \beta)/(A, M_1, \alpha)}^{G} = f^{-1} \mathbb{L}_{(B, M_2, \beta)/(A, M_1, \alpha)}^{\mathrm{pre}, G}
\]
since $f^{-1}$ and the functors $T$, $U$ commutes. As $f^{-1}$ is the sheafification functor, we have 
\[
H^i (\mathbb{L}_{(B, M_2, \beta)/(A, M_1, \alpha)}^G) = 
H^i (f^{-1} \mathbb{L}_{(B, M_2, \beta)/(A, M_1, \alpha)}^{\mathrm{pre}, G}) = 
f^{-1} H^i (\mathbb{L}_{(B, M_2, \beta)/(A, M_1, \alpha)}^{\mathrm{pre}, G})
\]
and the assertion follows.
\end{proof}

By the construction of the log cotangent complex of Gabber, the simplicial module $\mathbb{L}_{(B, M_2, \beta)/(A, M_1, \alpha)}^G$ is the sheafification of the association \[
U \mapsto \mathbb{L}_{(B(U), M_2 (U), \beta(U))/(A(U), M_1 (U), \alpha(U))}^G
\]
and we have the natural map
\[
\mathbb{L}_{(B(U), M_2 (U), \beta(U))/(A(U), M_1 (U), \alpha(U))}^G \to \mathbb{L}_{(B, M_2, \beta)/(A, M_1, \alpha)}^G (U).
\]
For a prelog ring $(A, M_1, \alpha)$, we let $\Spec(A, M_1, \alpha)$ denote the log scheme with the underlying scheme $\Spec A$ and the log structure associated with $M_1 \to A = \Gamma(\Spec A, \mathcal{O}_{\Spec A})$. 

\begin{lemma}
\label{LG-ring-asch}
Let $(A, M_1, \alpha) \to (B, M_2, \beta)$ be a morphism of fine prelog rings. We let $f \colon \Spec(B, M_2, \beta) \to \Spec(A, M_1, \alpha)$ be the corresponding morphism of fine log schemes. Then the canonical map 
\[
(\mathbb{L}_{(B, M_2, \beta)/(A, M_1, \alpha)}^G)^{\verb|~|} \to \mathbb{L}_{\Spec (B, M_2, \beta)/\Spec (A, M_1, \alpha)}^G
\]
is a quasi-isomorphism in $D(\Spec B)$, where $\verb|~|$ is the functor associating (complexes of) $B$-modules to (complexes of) quasi-coherent $\mathcal{O}_{\Spec B}$-modules as in \cite[Tag 01I7]{SP}.
\end{lemma}

\begin{proof}
We put $(X, M_X) \coloneqq \Spec (B, M_2, \beta)$ and $(Y, M_Y) \coloneqq \Spec (A, M_1, \alpha)$ respectively. We have the natural map
\[
\mathbb{L}_{(B, M_2, \beta)/(A, M_1, \alpha)}^G \to \mathbb{L}_{(\mathcal{O}_X (X), M_X (X)) / (f^{-1} \mathcal{O}_Y (X), f^{-1} M_Y (X))}^G.
\]
We also have the natural map
\[
\mathbb{L}_{(\mathcal{O}_X (X), M_X (X)) / (f^{-1} \mathcal{O}_Y (X), f^{-1} M_Y (X))}^G \to \mathbb{L}_{(X, M_X)/(Y, M_Y)}^G (X)
\]
by the argument before the lemma. The composition of the two natural maps 
\[
\mathbb{L}_{(B, M_2, \beta)/(A, M_1, \alpha)}^G \to \mathbb{L}_{(X, M_X)/(Y, M_Y)}^G (X)
\]
gives the morphism of complexes of $\mathcal{O}_{\Spec(B)}$-modules
\[
(\mathbb{L}_{(B, M_2, \beta)/(A, M_1, \alpha)}^G)^{\verb|~|} \to \mathbb{L}_{(X, M_X)/(Y, M_Y)}
\]
by \cite[Tag 01I7]{SP}. Then it suffices to show that the map is a quasi-isomorphism at each stalk. We take a point $\mathfrak{p} \in \Spec (B)$. We have to show the induced map at $\mathfrak{p}$ : 
\[
(\mathbb{L}_{(B, M_2, \beta)/(A, M_1, \alpha)}^G)_{\mathfrak{p}}^{\verb|~|} \to (\mathbb{L}_{(X, M_X)/(Y, M_Y)}^G)_{\mathfrak{p}}
\]
is a quasi-isomorphism. 
Let $\beta_{\mathfrak{p}}$ (resp.\ $\alpha_{\mathfrak{q}}$) denotes the map $M_2 \to B_{\mathfrak{p}}$ (resp.\ $M_1 \to A_{\mathfrak{q}}$) where $\mathfrak{q}$ is the prime ideal which is the inverse image of $\mathfrak{p}$ under the map $A \to B$. 
The right term $(\mathbb{L}_{(X, M_X)/(Y, M_Y)}^G)_{\mathfrak{p}}$ is naturally isomorphic to $\mathbb{L}_{(B_{\mathfrak{p}}, (M_X)_{\mathfrak{p}})/(A_{\mathfrak{q}}, (M_Y)_{\mathfrak{q}})}^G$
since taking stalks commutes with $\mathbb{L}^G$, which is seen by the argument of the commutativity of $f^{-1}$ and $\mathbb{L}^G$ in Lemma \ref{comm-LG-Hi}. Moreover, it is quasi-isomorphic to $\mathbb{L}_{(B_{\mathfrak{p}}, M_2, \beta_{\mathfrak{p}})/(A_{\mathfrak{q}}, M_1, \alpha_{\mathfrak{q}})}^G$ by the commutativity of $\mathbb{L}^G$ and taking associated log ringed structure up to the natural quasi-isomorphisms (\cite[Theorem 8.20]{Ols05}) in the case that $\mathcal{R}$ is the punctual topos $\{ \ast \}$.

We have to show that the left term $(\mathbb{L}_{(B, M_2, \beta)/(A, M_1, \alpha)}^G)_{\mathfrak{p}}^{\verb|~|}$ is naturally quasi-isomorphic to $(\mathbb{L}_{(B_{\mathfrak{p}}, M_2, \beta_{\mathfrak{p}})/(A_{\mathfrak{q}}, M_1, \alpha_{\mathfrak{q}})}^G)$, so it suffices to show that there is a natural quasi-isomorphism 
\[
\mathbb{L}_{(B_{\mathfrak{p}}, M_2, \beta_{\mathfrak{p}})/(A_{\mathfrak{q}}, M_1, \alpha_{\mathfrak{q}})}^G \cong \mathbb{L}_{(B, M_2, \beta)/(A, M_1, \alpha)}^G \otimes_B B_{\mathfrak{p}}.
\]
It is seen by the following two distinguished triangles
\begin{align*}
&\mathbb{L}_{(B, M_2, \beta)/(A, M_1, \alpha)}^G \otimes_B^L B_{\mathfrak{p}} \to 
\mathbb{L}_{(B_{\mathfrak{p}}, M_2, \beta_{\mathfrak{p}})/(A, M_1, \alpha)}^G \to 
\mathbb{L}_{(B_{\mathfrak{p}}, M_2, \beta_{\mathfrak{p}})/(B, M_2, \beta)}^G \to, \\ 
&\mathbb{L}_{(A_{\mathfrak{q}}, M_1, \alpha_{\mathfrak{q}})/(A, M_1, \alpha)}^G \otimes_{A_{\mathfrak{q}}}^L B_{\mathfrak{p}} \to 
\mathbb{L}_{(B_{\mathfrak{p}}, M_2, \beta_{\mathfrak{p}})/(A, M_1, \alpha)}^G \to 
\mathbb{L}_{(B_{\mathfrak{p}}, M_2, \beta_{\mathfrak{p}})/(A_{\mathfrak{q}}, M_1, \alpha_{\mathfrak{q}})}^G \to.
\end{align*}
\end{proof}

Let $f \colon X \to Y$ be a morphism of fine log schemes. The natural comparison morphism
\[
c_{X/Y} \colon \tau_{\geq \ast} \mathbb{L}_{X/Y}^G \to \mathbb{L}_{X/Y}
\]
between the two log cotangent complexes is constructed in \cite[8.31]{Ols05}. 

\begin{construction}[{\cite[8.31]{Ols05}}]
Since $\mathcal{L}og_Y$ is an algebraic stack over $Y$, there exists a smooth surjection $u \colon U \to \mathcal{L}og_Y$ from a scheme $U$. Let $U_{\ast}$ denote the $0$-coskeleton of $u$. We denote by $(U_{\ast}, M_{U_{\ast}})$ the associated simplicial log scheme for $U_{\ast} \to \mathcal{L}og_Y$, and let $(V_{\ast}, M_{V_{\ast}})$ be the log scheme 
\[
(X \times_{\mathcal{L}og_Y} U_{\ast}, \pi^{-1} M_X) \cong (X \times_{\mathcal{L}og_Y} U_{\ast}, j^\ast M_{U_{\ast}}),
\]
where $\pi \colon (V_{\ast}, M_{V_{\ast}}) \to X$ and $j \colon V_{\ast} \to U_{\ast}$ are the pullback morphisms. Then there is a commutative diagram
\[\xymatrix{
(V_{\ast}, M_{V_{\ast}})\ar[r]^-{j} \ar[d]^\pi & (U_{\ast}, M_{U_{\ast}})\ar[d]\\
X\ar[r] & Y,
}\]
which induces a morphism
\[
\pi^\ast \mathbb{L}_{X/Y}^G \to 
\mathbb{L}_{(V_{\ast}, M_{V_{\ast}})/(U_{\ast}, M_{U_{\ast}})}^G.
\]
Here $\mathbb{L}_{(V_{\ast}, M_{V_{\ast}})/(U_{\ast}, M_{U_{\ast}})}^G$ is isomorphic to $\mathbb{L}_{V_{\ast}/U_{\ast}}$ because $j$ is strict.
So we obtain a morphism $c_{X/Y} \colon \tau_{\geq \ast} \mathbb{L}_{X/Y}^G \to \mathbb{L}_{X/Y}$ in $D'_{\mathrm{qcoh}} (X)$, which is independent of the choice of a smooth covering $u \colon U \to \mathcal{L}og_Y$. 
The morphism $c_{X/Y} \colon \tau_{\geq \ast} \mathbb{L}_{X/Y}^G \to \mathbb{L}_{X/Y}$ is functorial in $X \to Y$. 
\end{construction}

The morphism is an isomorphism in special cases. 

\begin{proposition}[{\cite[Corollary 8.34]{Ols05}}]
\label{comp-LG-L}
Let $f \colon X \to Y$ be a morphism of fine log schemes. If at least one of following conditions are satisfied, the morphism $c_{X/Y} \colon \tau_{\geq \ast} \mathbb{L}_{X/Y}^G \to \mathbb{L}_{X/Y}$ is a quasi-isomorphism.
\begin{enumerate}
\item $Y$ is log flat over $\Spec \mathbb{Z}$ with the trivial log structure.
\item $f$ is integral.
\end{enumerate}
\end{proposition}

\section{Integral \texorpdfstring{$p$}{p}-adic cohomology of semistable formal schemes}
\label{sec;AinfCoh}

Let $S$ be a ring which is $\pi$-adically complete for some element $\pi \in S$ dividing $p$, and $\varphi \colon S/p \to S/p$ be the absolute Frobenius. We let
\[
S^{\flat} \coloneqq \varprojlim_{\varphi} (S/p)
\]
be the tilt of $S$. Then, there is the following natural isomorphisms of multiplicative monoids,
\[
\varprojlim_{x \mapsto x^p} S \cong \varprojlim_{\varphi} (S/p) = S^\flat.
\]
The period ring $A_{\inf} (S)$ is defined by $A_{\inf} (S) \coloneqq W(S^\flat)$ and $\varphi \colon A_{\inf} (S) \to A_{\inf} (S)$ also denotes the Witt vector Frobenius. We equip the ring $A_{\inf} (S)$ with the product topology $W(S^\flat) \cong \prod_{n=1}^\infty S^{\flat}$. Then the topology in $A_{\inf} (S)$ is complete and agree with the $(p, [x])$-adic topology for any nonzero nonunit $x = (\ldots,x^{(1)}, x^{(0)}) \in S^\flat$, where $x^{(i)}$ is an element of $S$ and $(x^{(i+1)})^p = x^{(i)}$. The \emph{de Rham specialization map} 
\[
\theta \colon A_{\inf} (S) \to S
\]
of $A_{\inf} (S)$ is defined to be the unique extension of the assignment $[x] \mapsto x^{(0)}$. The \emph{Hodge-Tate specialization map} is the composition
\[
\theta \varphi^{-1} \colon A_{\inf} (S) \to S.
\]
By the canonical isomorphism $A_{\inf} (S) \cong \varprojlim_F W_r (S)$ of \cite[Lemma 3.2]{BMS18}, we define the homomorphism 
\[
\tilde{\theta}_n \colon A_{\inf} (S) \to W_n (S)
\]
for each $n \geq 1$ to be the composition of the isomorphism above and the canonical projection $\varprojlim_F W_r (S) \to W_n (S)$. We also define the map 
\[
\theta_n \coloneqq \tilde{\theta}_n \varphi^n \colon A_{\inf} (S) \to W_n (S)
\]
for each $n \geq 1$. We mainly consider the case where $S$ is a perfectoid ring.

\begin{definition}[{\cite[Definition 3.5]{BMS18}}]
A ring $S$ is \emph{perfectoid} if and only if the following conditions are satisfied.
\begin{enumerate}
\item It is $\pi$-adically complete for some element $\pi \in S$ such that $\pi^p$ divides $p$.
\item The Frobenius map $\varphi \colon S/p \to S/p$ is surjective.
\item The kernel of $\theta \colon A_{\inf} (S) \to S$ is principal.
\end{enumerate}
\end{definition}

In the following, $S$ is assumed to be a perfectoid ring which contains a compatible system $(\zeta_{p^n})_{n \geq 1}$ of $p$-power roots of unity. We fix a compatible system $\epsilon = (\ldots, \zeta_{p^2}, \zeta_p, 1) \in S^\flat$ of $p$-power roots of unity in $S$. We put $\epsilon^{1/p^n} \coloneqq (\ldots,\zeta_{p^{n+1}}, \zeta_{p^n})$ and we define the elements $\mu$, $\xi_n$, $\tilde{\xi}_n$, $\xi$, $\tilde{\xi}$ of $A_{\inf} (S)$ by
\begin{align*}
\mu &\coloneqq [\epsilon] - 1, \\ 
\xi_n &\coloneqq \mu/\varphi^{-n} (\mu) = ([\epsilon] - 1)/([\epsilon^{1/p^n}] - 1) = \xi \varphi^{-1} (\xi) \cdots \varphi^{-(n-1)} (\xi), \\ 
\tilde{\xi}_n &\coloneqq \varphi^n (\mu) / \mu = \varphi^n (\xi_n)= \varphi (\xi) \varphi^2 (\xi) \cdots \varphi^n (\xi), \\ 
\xi &\coloneqq \xi_1, \\ 
\tilde{\xi} &\coloneqq \tilde{\xi_1}.
\end{align*}
Then the specialization map $\theta$ (resp.\ $\theta \varphi^{-1}$) has the kernel generated by the element $\xi$ (resp.\ $\tilde{\xi}$) by \cite[Example 3.16]{BMS18}. Moreover, for each $n \geq 1$, the map $\theta_n$ has the kernel generated by $\xi_n$ and the map $\tilde{\theta}_n$ has the kernel generated by $\tilde{\xi}_n$ by \cite[Lemma 3.12]{BMS18}. The ideal $(\mu) \subset A_{\inf} (S)$ does not depend on the choice of $\epsilon$ and the topology of $A_{\inf} (S)$ is $(p, \mu)$-adic.  

Let $K$ be a perfectoid field (in the sense of \cite[Definition 1.2]{Sch12}) which contains a compatible system of $p$-power roots of unity and $p^q$ for any $q \in \mathbb{Q}$. Let $\mathcal{O}$ be the ring of integers of $K$ for the associated rank-$1$-valuation, and $\mathfrak{m}$ be the maximal ideal of $\mathcal{O}$. We fix a system of compatible $p^n$-power roots $(\zeta_{p^n})_n$ of unity in $\mathcal{O}$. We also fix a system of compatible $p^n$-power roots $p^{\flat} \coloneqq (\ldots,p^{1/p},p) \in \mathcal{O^{\flat}}$ of $p$ in $\mathcal{O}$. 
We fix a non-negative rational number $q \in \mathbb{Q}_{\geq 0}$ and $p^q \in \mathcal{O}$. We denote by $A_{\inf}$ the period ring $A_{\inf} (\mathcal{O})$.

Let $\mathfrak{X}$ be a $p$-adic formal scheme over $\mathcal{O}$ that may be covered by open affines $\mathfrak{U}$ which admit $p$-adically \'etale $\mathcal{O}$-morphisms, which are called \emph{\'etale coordinate map} 
\begin{equation}
\label{EtCoordinateMap}
\Box \colon \mathfrak{U} = \Spf(R) \to \Spf(R^{\Box}) \ \mathrm{with} \ 
R^{\Box} \coloneqq \mathcal{O}\{ t_0,\ldots,t_r,t_{r+1}^{\pm 1},\ldots,t_d^{\pm 1}\}/(t_0\cdots t_r - p^q),
\end{equation}
for some $d \geq 0$ and $0 \leq r \leq d$. For example, the case where $\mathfrak{X}$ is the base change of a strictly semistable formal scheme over a discrete valuation subring $\mathcal{O}' \subset \mathcal{O}$ \cite[Section 1.5]{CK19}. The formal scheme $\mathfrak{X}$ is equipped with the log structure given by the subsheaf associated to the subpresheaf $Q_{\mathcal{O}_{\mathfrak{X}}} \coloneqq \mathcal{O}_{\mathfrak{X}_{\et}} \cap (\mathcal{O}_{\mathfrak{X}_{\et}} [\tfrac{1}{p}]) \subset \mathcal{O}_{\mathfrak{X}_{\et}}$.
Let $C$ be the complete algebraic closure of $W(k)[ \tfrac{1}{p}]$ where $k$ is an algebraically closed field of characteristic $p > 0$. 
In \cite{CK19}, $\mathcal{O}$ is assumed to be the ring $\mathcal{O}_C$ of integers of $C$. This assumption is mainly used for applying the results of \cite{Bei13b} to the log-crystalline (in addition, \'etale, log-de Rham) comparison theorem of $A_{\inf}$-cohomology. By the discussion of \cite[Section 1.5]{CK19}, the coordinate maps above admit models over some discrete valuation subring $\mathcal{O}' \subset \mathcal{O}_C$, in other words, there is an \'etale $\mathcal{O}'$-morphism
\begin{equation}
\label{OModel-CM}
U \to \Spec(\mathcal{O}'[t_0,\ldots t_r,t_{r+1}^{\pm 1},\ldots t_d^{\pm 1}]/(t_0,\ldots t_r - p^q)),
\end{equation}
whose $p$-adic completion of the $\mathcal{O}_C$-base change is the \'etale coordinate map (\ref{EtCoordinateMap}) in which $\mathcal{O} = \mathcal{O}_C$. In the following, we note some results about the charts of the formal schemes $\mathfrak{X}$ discussed in \cite[Section 1.6]{CK19}. we fix a valuation subring $\mathcal{O}' \subset \mathcal{O}$ and $\mathcal{O}$-model (\ref{OModel-CM}) of the coordinate morphism. 
Let $Q_U$ (resp.\ $Q_{\mathfrak{U}}$) be the log structure associated to $\mathcal{O}_{U_{\et}} \cap (\mathcal{O}_{U_{\et}} [\tfrac{1}{p}])^\ast$ (resp.\ $\mathcal{O}_{\mathfrak{U}_{\et}} \cap (\mathcal{O}_{\mathfrak{U}_{\et}} [\tfrac{1}{p}])^\ast$), and $V$ be the monoid of $p$-adic valuations of $\mathcal{O}' \backslash \{ 0\}$. Let $V \to (\mathcal{O}' \backslash \{ 0\})$ be the map that is constructed by a fixed system $(p^v)_{v \in V}$ of elements in $\mathcal{O}'$.

\begin{lemma}[{\cite[Claim 1.6.1]{CK19}}]
\label{chart-ssalg}
For a valuation subring $\mathcal{O}' \subset \overline{W(k)}$, the log scheme $(U, Q_U)$ has the chart
\[
\mathbb{N}^{r+1} \sqcup_{\mathbb{N}} V \to \Gamma(U, \mathcal{O}_U)
\]
given by $(a_i)_{0 \leq i \leq r} \mapsto \prod_{0 \leq i \leq r} t_i^{a_i}$ on $\mathbb{N}^{r+1}$, the diagonal $\mathbb{N} \to \mathbb{N}^{r+1}$ and $\mathbb{N} \xrightarrow{1 \mapsto q} V$ on $\mathbb{N}$, and the structure map $V \to (\mathcal{O}' \backslash \{ 0\}) \to \Gamma(U, \mathcal{O}_U)$ on $V$.
\end{lemma}

\begin{lemma}[{\cite[Claim 1.6.3]{CK19}}]
\label{chart-ssfalg}
For a valuation subring $\mathcal{O}' \subset \overline{W(k)}$, the formal $p$-adic completion morphism 
\[
j \colon (\mathfrak{U}_{\et}, Q_{\mathfrak{U}}) \to (U_{\et}, Q_U)
\]
of log ringed \'etale sites is strict.
\end{lemma}

Let $R$ be a $p$-adically algebra which admit an \'etale coordinate map $\Box \colon R^{\Box} \to R$. For each $m \geq 0$, we consider the $R^{\Box}$-algebras
\begin{align*}
R_m^{\Box} &\coloneqq \mathcal{O}\{ t_0^{1/p^m},\ldots, t_r^{1/p^m}, t_{r+1}^{\pm 1/p^m},\ldots, t_d^{\pm 1/p^m}\} /(t_0^{1/p^m}\cdots t_r^{1/p^m} - p^{q/p^m}), \\ 
R_{\infty}^{\Box} &\coloneqq (\varinjlim_m R_m^{\Box})^\wedge, \\ 
R_m &\coloneqq R \otimes_{R^{\Box}} R_m^{\Box}, \\ 
R_{\infty} &\coloneqq (\varinjlim_m R_m)^\wedge \cong (R \otimes_{R^{\Box}} R_{\infty}^{\circ})^\wedge.
\end{align*}
The rings $R_{\infty}^{\Box}, R_{\infty}$ are perfectoid, and the rings $R_m^{\Box}, R_m$ are $p$-adically complete by \cite[Lemma 7.1.6, (ii)]{GR03}.
The rings $R_m^{\Box}, R_m, R_{\infty}^{\Box}, R_{\infty}$ have the continuous and $R^{\Box}$-linear actions of the profinite group 
\[
\Delta \coloneqq \{ (\epsilon_0,\ldots ,\epsilon_d) \in (\varprojlim_{m \geq 0} \mu_{p^m}(\mathcal{O}))^{\oplus(d+1)} \mid \epsilon_0 \cdots \epsilon_r = 1\} \cong \mathbb{Z}_p^{\oplus d}
\]
defined by 
\[
(\epsilon_0, \ldots, \epsilon_d) \colon t_j^{1/p^m} \mapsto \epsilon_j^{(m)} t_j^{1/p^m} 
\]
for $(\epsilon_0, \ldots, \epsilon_d) \in \Delta$, where $\epsilon_j^{(m)} \in \mu_{p^m} (\mathcal{O})$ is the $m$-th component of $\epsilon_j$. The profinite group $\Delta$ is 
topologically generated by the elements 
\[
\delta_i \coloneqq (\epsilon^{-1}, 1, \ldots, \epsilon, \ldots, 1, \ldots, 1) \qquad (i = 1, \ldots, r),
\]
where the $0$-th and $i$-th entries are nonidentity and the elements
\[
\delta_i \coloneqq (1, \ldots, 1, \epsilon, 1, \ldots, 1) \qquad (i = r+1, \ldots, d),
\]
where the $i$-th entry is nonidentity. Then we have the explicit ($p$-adically completed) direct sum decompositions of $R_m^{\Box}$ and $R_m^{\Box} [\tfrac{1}{p}]$. 
\begin{lemma}
\label{RmBox-cdsdecomp}
$R_m^{\Box}$ and $R_m^{\Box} [\tfrac{1}{p}]$ have the following $p$-adically completed direct sum decompositions. For a function
\[ k \colon \{0,\cdots,d\} = [0, d] \to \mathbb{Z}_{\geq 0} [1/p], \]
let $\Supp k \coloneqq [0, d] \backslash (k^{-1} (\{0\}))$, $k_i \coloneqq k(i)$ and
\begin{align*}
R_m^{\Box} &\cong \widehat{\bigoplus_{\substack{k \colon [0, d] \to \tfrac{1}{p^m} \mathbb{Z} \\ [0, r] \not\subset \Supp k}}} \mathcal{O} \cdot t_0^{k_0} \cdots t_d^{k_d} \\ 
R_m^{\Box} [\tfrac{1}{p}] &\cong \bigoplus_{k_1, \ldots, k_d \in \{ 0, 1/p^m, \ldots, {p^m - 1}/{p^m} \} } R^{\Box} [\tfrac{1}{p}] \cdot t_1^{k_1} \cdots t_d^{k_d}.
\end{align*}
Furthermore, $R_{\infty}^{\Box}$ and $R_{\infty}^{\Box} [\tfrac{1}{p}]$ have similar ones so that the canonical decompositions
\[
R_{\infty}^{\Box} \cong R^{\Box} \oplus M_{\infty}^{\Box},
\quad 
R_{\infty}^{\Box} [\tfrac{1}{p}] \cong R^{\Box} [\tfrac{1}{p}] \oplus M_{\infty}^{'\Box}
\]
exist. 
\end{lemma}

In particular, $R_m^{\Box} [\tfrac{1}{p}]$ is generated by the $p^m$-th roots of $t_1, \ldots, t_d \in R^{\Box} [\tfrac{1}{p}]$ over $R^{\Box} [\tfrac{1}{p}]$, and $R_m^{\Box} [\tfrac{1}{p}]$ is finite \'etale over $R^{\Box} [\tfrac{1}{p}]$. Moreover, 
\[
(\Spf(R^{\Box}))_{K, \infty}^{\ad} \coloneqq \varprojlim_m \Spa(R_m^{\Box} [\tfrac{1}{p}], R_m^{\Box}) \cong \Spa(R_{\infty}^{\Box} [\tfrac{1}{p}], R_{\infty}^{\Box})
\]
is an affinoid perfectoid pro-(finite \'etale) cover of $(\Spa(R^{\Box} [\tfrac{1}{p}], R^{\Box}))_K^{\ad}$. In the same way, 
\[
(\Spf(R))_{K, \infty}^{\ad} \coloneqq \varprojlim_m \Spa(R_m [\tfrac{1}{p}], R_m) \cong \Spa(R_{\infty} [\tfrac{1}{p}], R_{\infty})
\]
is an affinoid perfectoid pro-(finite \'etale) cover of $(\Spa(R [\tfrac{1}{p}], R))_K^{\ad}$.

We construct the canonical lifts $A(R^{\Box}), A_{\inf} (R_{\infty}^{\Box})$ of $R^{\Box}, R_{\infty}^{\Box}$ over $A_{\inf}$ and the unique $(p, \mu)$-adically \'etale algebras $A(R)$ over $A(R^{\Box})$. We can explicitly write down $A_{\inf} (R_{\infty}^{\Box})$ as follows.
\begin{align*}
&\quad A_{\inf} (R_{\infty}^{\Box}) \\
&\cong 
\left( \varinjlim_m A_{\inf} [X_0^{1/p^m}, \ldots, X_r^{1/p^m}, X_{r+1}^{\pm 1/p^m}, \ldots, X_d^{\pm 1/p^m}]/(\prod_{i=0}^r X_i^{1/p^m} - [(p^{\flat})^{q/p^m}]) \right)^\wedge.
\end{align*}
Let $A(R^{\Box})$ be the direct summand of the elements that is topologically generated by the monomials which have integral exponents
\[
A(R^{\Box}) \cong A_{\inf} \{ X_0, \ldots, X_r, X_{r+1}^{\pm 1}, \ldots, X_d^{\pm 1} \} /(X_0 \cdots X_r - [(p^{\flat})^q]) \subset A_{\inf} (R_{\infty}^{\Box}).
\]
We define the action of $(\epsilon_0, \ldots, \epsilon_d) \in \Delta$ on $A_{\inf} (R_{\infty}^{\Box})$ by 
\[
(\epsilon_0, \ldots, \epsilon_d) \colon X_j^{1/p^m} \mapsto [\epsilon_j^{1/p^m}] X_j^{1/p^m} 
\]
where $\epsilon_j^{1/p^m}$ is defined to be the element $(\ldots,\epsilon_j^{(m+1)},\epsilon_j^{(m)}) \in R^{\flat}$ and $[\epsilon_j^{1/p^m}] \in A_{\inf}$.
We use the surjection
\[
\theta \colon A(R^{\Box}) \twoheadrightarrow R^{\Box}
\]
defined by $\theta \colon A_{\inf} \twoheadrightarrow \mathcal{O}$ and $X_i \mapsto t_i$. The map uniquely lifts the \'etale $R^{\Box} /p$-algebra $R/p$ to the $(p, \mu)$-adically complete, $p$-adically \'etale $A(R^{\Box})$-algebra $A(R)$. We have the identification 
\[
A_{\inf} (R_{\infty}) \cong A_{\inf} (R_{\infty}^{\Box}) \widehat{\otimes}_{A(R^{\Box})} A(R),
\]
where the completion is $(p, \mu)$-adic, and the $(p, \mu)$-adically complete direct sum decompositions as in Lemma \ref{RmBox-cdsdecomp}. So the following canonical decompositions exist. 

\begin{lemma}
\label{AinfRinfty-cdsdecomp}
We have the canonical direct sum decompositions
\[
A_{\inf} (R_{\infty}^{\Box}) \cong A(R^{\Box}) \oplus N_{\infty}^{\Box},
\quad 
A_{\inf} (R_{\infty}) \cong A(R) \oplus N_{\infty}.
\]
\end{lemma}

We follow the construction of \cite{Sch13}, \cite{Sch13C} and \cite{BMS18}. We let $\mathfrak{X}_K^{\ad}$ denote the adic generic fiber of $\mathfrak{X}$. We have sheaves of rings on the pro-\'etale site $(\mathfrak{X}_K^{\ad})_{\proet}$ defined in \cite[Section 5.1]{BMS18},
which are 
\begin{align*}
\widehat{\mathcal{O}}_{\mathfrak{X}_K^{\ad}}^+ &\coloneqq \varprojlim_n (\mathcal{O}_{\mathfrak{X}_K^{\ad}, \proet} / p^n) \\ 
\widehat{\mathcal{O}}_{\mathfrak{X}_K^{\ad}}^{+, \flat} &\coloneqq \varprojlim_{x \mapsto x^p} (\mathcal{O}_{\mathfrak{X}_K^{\ad}, \proet} / p) \\ 
\mathbb{A}_{\inf, \mathfrak{X}_K^{\ad}} &\coloneqq (A_{\inf} (\widehat{\mathcal{O}}_{\mathfrak{X}_K^{\ad}}^+))^{\wedge} = 
(W(\widehat{\mathcal{O}}_{\mathfrak{X}_K^{\ad}}^{+, \flat}))^{\wedge}
\end{align*}
where $(-)^{\wedge}$ denotes the derived $p$-adic completion. We remark that the definition of $\mathbb{A}_{\inf, \mathfrak{X}_K^{\ad}}$ in \cite[2.2]{CK19} erroneously lacks the derived $p$-adic completion of $W(\widehat{\mathcal{O}}_{\mathfrak{X}_K^{\ad}}^{+, \flat})$. If the adic space $\mathfrak{X}_K^{\ad}$ is obvious from the context, we omit it and we denote by $\widehat{\mathcal{O}}^{+}$ the sheaf of rings $\widehat{\mathcal{O}}_{\mathfrak{X}_K^{\ad}}^{+}$. The category of pro-(finite \'etale) covers of connected locally noetherian schemes or connected locally noetherian adic space $X$ is canonically equivalent to the category of finite sets with continuous $\pi_1 (X, \bar{x})$-action for a fixed geometric point $\bar{x} \in X$ \cite[Proposition 3.5]{Sch13}. Then the translation of the pro-(finite \'etale) cohomology theory into the continuous cohomology theory exists as in \cite[Proposition 3.7, Corollary 6.6]{Sch13}. Thus the \v{C}ech complex of the sheaf $\widehat{\mathcal{O}}_{K^{\ad}}^+$ with respect to the pro-(finite \'etale) affinoid perfectoid cover 
\[
(\Spf(R))_{K, \infty}^{\ad} \to (\Spf(R))_K^{\ad}
\]
is identified with the continuous cochain complex $R\Gamma_{\cont} (\Delta, R_{\infty})$. 

\begin{definition}
\emph{The edge map} is the map from the continuous cohomology to the pro-\'etale cohomology 
\[
e \colon R\Gamma_{\cont} (\Delta, R_{\infty}) \to R\Gamma_{\proet} ((\Spf(R))_K^{\ad}, \widehat{\mathcal{O}}^+)
\]
obtained by the affinoid perfectoid pro-\'etale cover above and using \cite[Tag 01GY]{SP}. 
\end{definition}

A map $e \colon D \to D'$ of complexes of $\mathcal{O}$ (resp. $W_n (\mathcal{O})$, $A_{\inf}$ -modules is called an \emph{almost quasi-isomorphism} if each cohomology of the cone $\Cone(e)$ of $e$ is annihilated by $\mathfrak{m} \subset \mathcal{O}$ (resp. $W_n (\mathfrak{m}) \subset W_n (\mathcal{O})$, $W(\mathfrak{m^{\flat}}) \subset A_{\inf} (\mathcal{O})$). By the almost purity theorem \cite[Lemma 4.10 (v)]{Sch13} and Cartan-Leray spectral sequence, the edge map $e$ turns out to be an almost quasi-isomorphism. The following lemma follows by \cite[Lemma 6.4]{BMS18}. 

\begin{lemma}
\label{Letae-O}
For the edge map $e$ above, we have the almost quasi-isomorphism
\[
L\eta_{\zeta_p - 1} R\Gamma_{\cont} (\Delta, R_{\infty}) \to L\eta_{\zeta_p - 1} R\Gamma_{\proet} ((\Spf(R))_K^{\ad}, \widehat{\mathcal{O}}^+).
\]
\end{lemma}

In the same way, we have the following natural almost quasi-isomorphisms by \cite[Lemma 5.6, Theorem 5.7]{BMS18}.

\begin{lemma}
\label{Letae-Ainf}
Let $\mathfrak{X} = \Spf(R)$. For $n \geq 1$, let $e_{\inf}, e_n$ be the natural almost quasi-isomorphisms
\begin{align*}
e_n \colon R\Gamma_{\cont} (\Delta, W_n (R_{\infty})) &\to R\Gamma_{\proet} ((\Spf(R))_K^{\ad}, W_n (\widehat{\mathcal{O}}^+)), \\  
e_{\inf} \colon R\Gamma_{\cont} (\Delta, A_{\inf} (R_{\infty})) &\to R\Gamma_{\proet} ((\Spf(R))_K^{\ad}, \mathbb{A}_{\inf, \mathfrak{X}_K^{\ad}})
\end{align*}
constructed by the Cartan-Leray spectral sequence. Then we have the natural maps 
\begin{align*}
L\eta_{\mu} e_n \colon L\eta_{\mu} R\Gamma_{\cont} (\Delta, W_n (R_{\infty})) &\to L\eta_{\mu} R\Gamma_{\proet} ((\Spf(R))_K^{\ad}, W_n (\widehat{\mathcal{O}}^+)), \\ 
L\eta_{\mu} e_{\inf} \colon L\eta_{\mu} R\Gamma_{\cont} (\Delta, A_{\inf} (R_{\infty})) &\to L\eta_{\mu} R\Gamma_{\proet} ((\Spf(R))_K^{\ad}, \mathbb{A}_{\inf, \mathfrak{X}_K^{\ad}}),
\end{align*}
whose kernels and cokernels in each degree are annihilated by $W (\mathfrak{m}^{\flat})$ 
\end{lemma}

We notice that $W(\mathfrak{m}^{\flat})^2 \neq W(\mathfrak{m}^{\flat})$, so we should take care of the meanings of ``almost (quasi-)isomorphism''. This argument needs the derived $p$-adic completeness of $\mathbb{A}_{\inf, \mathfrak{X}_K^{\ad}}$ in its definition.

We define complexes $\widetilde{\Omega}_R, \widetilde{W_n \Omega}_R, A\Omega_R$ and so on. We remark that all of these complexes can be defined as the complexes of sheaves on the \'etale site $\mathfrak{X}_{\et}$ of a semistable $p$-adic formal scheme $\mathfrak{X}$ as in \cite{CK19}. However, we define those as $R$, $W_n (R)$, $W(R^{\flat})$-modules respectively. Nevertheless, most of the following argument may work as in \cite{CK19}. 

\begin{definition}
Let $\mathfrak{X}$ be the $p$-adic formal spectrum $\Spf (R)$ of $R$. We define the following complexes 
\begin{align*}
\widetilde{\Omega}_R^{\Box} &\coloneqq L\eta_{\zeta_p - 1} R\Gamma_{\cont} (\Delta, R_{\infty}), \\ 
\widetilde{\Omega}_R &\coloneqq L\eta_{\zeta_p - 1} R\Gamma_{\proet} (\mathfrak{X}_K^{\ad}, \widehat{\mathcal{O}}_X^+), \\ 
\widetilde{W_n \Omega}_R^{\Box} &\coloneqq L\eta_{[\zeta_{p^n}]-1} R\Gamma_{\cont} (\Delta, W_n (R_{\infty})), \\ 
\widetilde{W_n \Omega}_R &\coloneqq L\eta_{[\zeta_{p^n}] - 1} R\Gamma_{\proet} (\mathfrak{X}_K^{\ad}, W_n (\widehat{\mathcal{O}}_{\mathfrak{X}_K^{\ad}}^+)), \\ 
A\Omega_R^{\Box} &\coloneqq L\eta_{\mu} R\Gamma_{\cont} (\Delta, A_{\inf} (R_{\infty})), \\ 
A\Omega_R &\coloneqq L\eta_{\mu} R\Gamma_{\proet} (\mathfrak{X}_K^{\ad}, \mathbb{A}_{\inf, \mathfrak{X}_K^{\ad}}),
\end{align*}
where $L\eta_{\zeta_p -1}$ (resp.\ $L\eta_{[\zeta_{p^n}] - 1}$, $L\eta_{\mu}$) is taken in the derived category $D(\mathcal{O})$ (resp.\ $D(W_n (\mathcal{O}))$, $D(A_{\inf})$).
\end{definition}

The framing $\Box$ of $R$ induces the actions of $\Delta$ to the complexes $\widetilde{\Omega}_R^{\Box}$, $\widetilde{W_n \Omega}_R^{\Box}$, $A\Omega_R^{\Box}$ in the definition. 
We have following lemmas, which are proved in \cite[Section 8, Section 9]{BMS18} for the smooth case.  

\begin{proposition}[cf.\ {\cite[Corollary 8.13]{BMS18}, \cite[Proposition 4.4, Proposition 4.8]{CK19}}]
\label{galg-tildeLambda}
There is an isomorphism
\[
R^d \cong H^1 (\widetilde{\Omega}_R), 
\]
depending on the choice of framing. The exterior powers of the map induce isomorphisms
\[
\bigwedge^i H^1 (\widetilde{\Omega}_R) \cong H^i (\widetilde{\Omega}_R).
\]
\end{proposition}

Now we define log $q$-de Rham complexes. We show that $A\Omega_R^{\Box}$ is naturally quasi-isomorphic to the log $q$-de Rham complex in Lemma \ref{ALambda-WnLambda}. 

\begin{definition}
\label{qlogdRcomp}
For $q = [\epsilon]$, we define the \emph{log $q$-de Rham complex} for $A(R)^{\Box} / A_{\inf}$ by 
\[
q\textrm{-}\Omega_{A(R^{\Box}) / A_{\inf}} \coloneqq 
K_{A(R^{\Box})} \left( \dfrac{\delta_1 - 1}{[\epsilon] - 1},\ldots,\dfrac{\delta_d - 1}{[\epsilon] - 1} \right).
\]
\end{definition}

There is an isomorphism
\[
q\textrm{-}\Omega_{A(R^{\Box}) / A_{\inf}}
\cong \bigotimes_{i=1}^d (A(R^{\Box}) \xrightarrow{\tfrac{\partial}{\partial t_i}} A(R^{\Box}) d\log t_i),
\]
where $\tfrac{\partial}{\partial t_i}$ is the induced map from $\tfrac{\delta_i - 1}{[\epsilon] - 1}$.

The $q$-de Rham complexes similarly exist over $W_n (\mathcal{O})$. Let $\Delta'$ denote the free abelian subgroup of $\Delta$ generated by $\delta_i, \ 1 \leq i \leq d$. 

\begin{lemma}
\label{ALambda-WnLambda}
The natural map in $D(W_n (\mathcal{O}))$
\[
A\Omega_R^{\Box} \otimes_{A_{\inf}, \tilde{\theta}_r}^L W_n (\mathcal{O}) \to \widetilde{W_n \Omega}_R^{\Box} 
\]
is a quasi-isomorphism. Moreover, we have the following natural quasi-isomorphisms : 
\begin{align*}
A\Omega_R^{\Box} &\to q\textrm{-}\Omega_{A(R) / A_{\inf}} \\ 
\widetilde{W_n \Omega}_R^{\Box} &\to q\textrm{-}\Omega_{(A(R)/ \tilde{\xi}_n) / W_n (\mathcal{O})}.
\end{align*}
\end{lemma}

\begin{proof}
We have a decomposition of cohomology modules 
\[
H_{\cont}^i (\Delta, A_{\inf} (R_{\infty})) \cong 
H_{\cont}^i (\Delta, A(R)) \oplus H_{\cont}^i (\Delta, N_{\infty}) 
\]
for $i \in \mathbb{Z}$ by Lemma \ref{AinfRinfty-cdsdecomp}. The element $\mu$ kills the direct summand $H_{\cont}^i (\Delta, N_{\infty})$ by \cite[Proposition 3.25]{CK19}. So we have $L\eta_{\mu} R\Gamma_{\cont} (\Delta, N_{\infty}) = 0$ by \cite[Lemma 6.4]{BMS18}. Since the equalities
\[
R\Gamma_{\cont} (\Delta, A(R)) = 
R\Gamma (\Delta', A(R)) = 
K_{A(R)} (\delta_1 - 1,\ldots, \delta_d - 1) 
\]
follow by \cite[Lemma 7.3, Lemma 7.5]{BMS18}, we obtain quasi-isomorphisms 
\begin{align*}
A\Omega_R^{\Box} &= L\eta_{\mu} R\Gamma_{\cont} (\Delta, A_{\inf} (R_{\infty})) \\ 
&\to L\eta_{\mu} R\Gamma_{\cont} (\Delta, A(R)) \\ 
&\to K_{A(R)} \left( \dfrac{\delta_1 - 1}{[\epsilon] - 1},\ldots, \dfrac{\delta_d - 1}{[\epsilon] - 1} \right) \\ 
&= q\textrm{-}\Omega_{A(R) / A_{\inf}}
\end{align*}
by \cite[Corollary 6.5, Lemma 7.9]{BMS18}. 

Similarly, we shall give a quasi-isomorphism 
\[
\widetilde{W_n \Omega}_R^{\Box} \to q\textrm{-}\Omega_{(A(R)/ \tilde{\xi}_n) / W_n (\mathcal{O})}.
\]
Arguing as before, it is enough to prove
\[
L\eta_{[\zeta_{p^n}] - 1} (K_{N_{\infty}/ \tilde{\xi}_n} (\delta_1 - 1,\ldots, \delta_d - 1)) = 0.
\]
Since $L\eta_{\mu} (K_{N_{\infty}} (\delta_1 - 1,\ldots, \delta_d - 1)) = 0$ as above, it is enough to prove 
\[
L\eta_{[\zeta_{p^n}] - 1} (K_{N_{\infty}/ \tilde{\xi}_n} (\delta_1 - 1,\ldots, \delta_d - 1)) \cong 
L\eta_{\mu} (K_{N_{\infty}} (\delta_1 - 1,\ldots, \delta_d - 1)) \otimes_{A_{\inf}}^L A_{\inf}/ \tilde{\xi}_n.
\]
This follows from \cite[5.16]{Bha18} since for $i \in \mathbb{Z}$, we have the following natural injection
\[
H^i (K_{N_{\infty}} (\delta_1 - 1,\ldots, \delta_d - 1) \otimes_{A_{\inf}}^L A_{\inf} / \mu) \to 
H_{\cont}^i (\Delta, A_{\inf} (R_{\infty}) / \mu),
\]
in which $H_{\cont}^i (\Delta, A_{\inf} (R_{\infty}) / \mu)$ has no nonzero $p^n$-torsion by \cite[Proposition 3.19]{CK19}, and $p^n \equiv \tilde{\xi}_n$ modulo $\mu$ by \cite[Proposition 3.17]{BMS18}. 

Finally, the first assertion of the lemma follows by the quasi-isomorphism 
\[
q\textrm{-}\Omega_{A(R) / A_{\inf}} \otimes_{A_{\inf}, \tilde{\theta}_r}^L W_n (\mathcal{O}) \to q\textrm{-}\Omega_{(A(R)/ \tilde{\xi}_n) / W_n (\mathcal{O})}.
\]
\end{proof}

\begin{lemma}[cf.\ {\cite[Lemma 9.9]{BMS18}}]
\label{WnLambda-et-basechange}
Let $R$ be as above, and let $R \to R'$ be a $p$-adically \'etale map of $p$-adically complete algebras. We equip $R'$ with the induced framing $\Box'$ from that of $R$. Then the natural map 
\[
\widetilde{W_n \Omega}_R^{\Box} \widehat{\otimes}_{W_n (R)}^L W_n (R') \to 
\widetilde{W_n \Omega}_{R'}^{\Box'}
\]
is a quasi-isomorphism. 
\end{lemma}

We can naturally define \'etale sheafified version $\widetilde{\Omega}_{\mathfrak{X}}, \widetilde{W_n \Omega}_{\mathfrak{X}}, A\Omega_{\mathfrak{X}}$ on $\mathfrak{X}_{\et}$ for a $p$-adic formal scheme $\mathfrak{X}$ over $\mathcal{O}$ as in \cite[Section 4]{CK19}.  

\begin{proposition}
\label{comp-WrLambda}
Let $\mathfrak{X}$ be the $p$-adic formal spectrum $\Spf (R)$ of $R$. The natural maps in $D(W_n (\mathcal{O}))$
\[
\widetilde{W_n \Omega}_R^{\Box} \to \widetilde{W_n \Omega}_R \to 
R\Gamma (\mathfrak{X}_{\et}, \widetilde{W_n \Omega}_{\mathfrak{X}})
\]
are quasi-isomorphisms. 
\end{proposition}

\begin{proof}
We prove that the first map is a quasi-isomorphism at first. We follow the proof of \cite[Proposition 3.18]{CK19}. The cohomology group $H_{\cont}^i (\Delta, W_n (R_{\infty}) / ([\zeta_{p^n}] - 1))$ has no nonzero $W_n (\mathfrak{m})$-torsion by the equation $\tilde{\theta}_n ((\mu, \tilde{\xi}_n)) = \tilde{\theta}_n ((\mu, p^n))$, which follows from \cite[Proposition 3.17]{BMS18}, of ideals in $W_n (R_{\infty})$ and \cite[Proposition 3.19]{CK19}. We put 
\[
B = L\eta_{[\zeta_{p^n}]-1} R\Gamma_{\cont} (\Delta, W_n (R_{\infty})), \qquad B' = L\eta_{[\zeta_{p^n}] - 1} R\Gamma_{\proet} (\mathfrak{X}_C^{\ad}, W_n (\widehat{\mathcal{O}}_{\mathfrak{X}_K^{\ad}}^+)),
\]
and let $e \colon B \to B'$ be the natural morphism constructed by the Cartan-Leray spectral sequence. The ideal $W_n (\mathfrak{m})^2 = W_n (\mathfrak{m})$, which is seen by \cite[Corollary 10.2]{BMS18}, kills the cohomology of 
\[
\Cone(e) \otimes_{W_n (R_{\infty})}^L W_n (R_{\infty}) / ([\zeta_{p^n}] - 1).
\]
Then we have the following exact sequence
\begin{align*}
0 &\to H^i (B \otimes_{W_n (R_{\infty})}^L W_n (R_{\infty}) / ([\zeta_{p^n}] - 1)) \to 
H^i (B' \otimes_{W_n (R_{\infty})}^L W_n (R_{\infty}) / ([\zeta_{p^n}] - 1)) \\ &\to 
H^i (\Cone(e) \otimes_{W_n (R_{\infty})}^L W_n (R_{\infty}) / ([\zeta_{p^n}] - 1)) \to 0.
\end{align*}
For $K \in D(W_n (R_{\infty}))$, we have the natural quasi-isomorphism
\[
L\eta_{[\zeta_{p^n}] - 1} K \otimes_{W_n (R_{\infty})}^L W_n (R_{\infty}) / ([\zeta_{p^n}] - 1) \cong H^{\ast} (K \otimes_{W_n (R_{\infty})}^L W_n (R_{\infty}) / ([\zeta_{p^n}] - 1))
\]
by \cite[Proposition 6.12]{BMS18}. The differential of the right hand side is induced by the Bockstein differential. We have $L\eta_{[\zeta_{p^n}] - 1} (\Cone(e)) = 0$ by the almost purity theorem \cite[Theorem 6.5 (ii)]{Sch13}, \cite[Lemma 5.6, Theorem 5.7]{BMS18}. So the map 
\[
L\eta_{[\zeta_{p^n}] - 1} (e) \otimes_{W_n (R_{\infty})}^L W_n (R_{\infty}) / ([\zeta_{p^n}] - 1)
\]
is a quasi-isomorphism. So it follows that 
\[
\Cone(L\eta_{[\zeta_{p^n}] - 1)} (e)) \otimes_{W_n (R_{\infty})}^L W_n (R_{\infty}) / ([\zeta_{p^n}] - 1) \cong 0
\]
and so $[\zeta_{p^n}] - 1$ acts isomorphically on each $H^i (\Cone(L\eta_{[\zeta_{p^n}] - 1} (e)))$. Therefore, it suffices to show that each $H^i (\Cone(L\eta_{[\zeta_{p^n}] - 1} (e))) \otimes_{A_{\inf}}^L A_{\inf} [1/\mu]$ vanishes. This is seen by the equation 
\[
H^i (L\eta_{[\zeta_{p^n}] - 1} B) \left[ \dfrac{1}{[\zeta_{p^n}] - 1} \right] \cong H^i (B) \left[ \dfrac{1}{[\zeta_{p^n}] - 1} \right].
\]
For the second map, we can similarly show that it is a quasi-isomorphism as \cite[Corollary 9.11]{BMS18}. We have the following natural map
\[
\widetilde{W_n \Omega}_R \otimes_{W_n (R)} W_n (\mathcal{O}_{\mathfrak{X}}) \to \widetilde{W_n \Omega}_{\mathfrak{X}}.
\]
We shall show the map is a quasi-isomorphism on each stalk at each geometric point $\overline{x} \to \mathfrak{X}$. For the left term, we have $R\Gamma (\mathfrak{U}, \widetilde{W_n \Omega}_R \otimes_{W_n (R)} W_n (\mathcal{O}_{\mathfrak{X}})) = \widetilde{W_n \Omega}_{R'}$, where $\mathfrak{U} = \Spf R'$ is \'etale over $\mathfrak{X}$. So it follows that 
\begin{equation}
\label{comp-WrLambda-leftstalk}
\varinjlim_{\mathfrak{U}} R\Gamma (\mathfrak{U}, \widetilde{W_n \Omega}_R \otimes_{W_n (R)} W_n (\mathcal{O}_{\mathfrak{X}})) = 
\varinjlim_{\mathfrak{U}} L\eta_{[\zeta_{p^n}] - 1} R\Gamma_{\proet} (\mathfrak{U}_K^{\ad}, W_n (\widehat{\mathcal{O}}_{\mathfrak{X}_K^{\ad}})),
\end{equation}
where $\mathfrak{U}$ runs over affine \'etale neighborhood of $\overline{x}$. For the right term, we have 
\[ 
(\widetilde{W_n \Omega}_{\mathfrak{X}})_{\overline{x}} = 
L\eta_{[\zeta_{p^n}] - 1} \varinjlim_{\mathfrak{U}}R\Gamma_{\proet} (\mathfrak{U}_K^{\ad}, W_n (\widehat{\mathcal{O}}_{\mathfrak{X}_K^{\ad}}))
\]
by \cite[Lemma 6.14]{BMS18}. 
The assertion follows from the commutativity of $L\eta$ and filtered colimits. 
\end{proof}

\begin{proposition}[cf.\ {\cite[Theorem 3.20]{CK19}}]
\label{comp-ALambda}
Let $\mathfrak{X}$ be the $p$-adic formal spectrum $\Spf (R)$ of $R$. The natural maps in $D(A_{\inf})$ 
\[
A\Omega_R^{\Box} \to A\Omega_R \to R\Gamma (\mathfrak{X}, A\Omega_{\mathfrak{X}})
\]
are quasi-isomorphisms. 
\end{proposition}

We define the BKF twist $\mathcal{O} \{ 1\}$ of $\mathcal{O}$.
Here ``BKF'' stands for ``Breuil-Kisin-Fargues.'' 

\begin{definition}
We set 
\[
\mathcal{O} \{ 1\} \coloneqq T_p (\Omega_{\mathcal{O}/\mathbb{Z}_p}^1) = \varprojlim_n (\Omega_{\mathcal{O}/\mathbb{Z}_p}^1 [p^n]).
\]
Here $\mathcal{O} \{ 1\}$ is called the \emph{BKF-twist} of $\mathcal{O}$. 
\end{definition}

BKF-twist can be written in terms of the $p$-complete cotangent complex or the subquotient of $A_{\inf}$. 

\begin{lemma}
\label{BKF-fund}
There is a natural isomorphisms 
\[
\mathcal{O} \{ 1\} = \varprojlim_n (\Omega_{\mathcal{O}/\mathbb{Z}_p}^1 [p^n]) \cong 
(\mathbb{L}_{\mathcal{O}/\mathbb{Z}_p} [-1])^{\wedge} \cong 
(\mathbb{L}_{\mathcal{O}/A_{\inf}} [-1])^{\wedge} \cong 
\tilde{\xi} A_{\inf} / \tilde{\xi}^2 A_{\inf},
\]
where $(-)^\wedge$ denotes the derived $p$-adic completion. 
\end{lemma}

\begin{proof}
The first isomorphism is seen by \cite[Theorem 6.5.12]{GR03} and the definitions of the derived completion and $R\varprojlim$. For the second one, \cite[Lemma 3.14]{BMS18} and derived Nakayama lemma shows that $\mathbb{L}_{A_{\inf}/\mathbb{Z}_p}$ vanishes. So the second isomorphism follows by the triangle of $p$-completed derived base change of the $p$-complete cotangent complexes. The third equivalence follows by \cite[Corollaire III.3.2.7]{Ill71}. 
\end{proof}

We note some properties of $\mathcal{O} \{ 1\}$. The canonical map $d\log \colon \mu_{p^\infty} \to \Omega_{\mathcal{O}/\mathbb{Z}_p}^1$ induces the embedding of the Tate twist $\mathcal{O} (1) = \mathbb{Z}_p (1) \otimes_{\mathbb{Z}_p} \mathcal{O}$ into $\mathcal{O} \{ 1\}$ as 
\[
d\log \colon \mathcal{O} (1) \hookrightarrow \mathcal{O} \{ 1\}, \ 
(\zeta_{p^i})_i \otimes a \mapsto a \cdot (d\log \zeta_{p^i})_i,
\]
where $a \in \mathcal{O}$. For any topological generator $(\zeta_{p^i})_i \in \mathbb{Z}_p (1)$,
the BKF-twist $\mathcal{O} \{ 1\}$ is generated by the element 
\[
\left( \dfrac{1}{\zeta_p - 1} d\log \zeta_{p^i} \right)_i \in T_p (\Omega_{\mathcal{O}/\mathbb{Z}_p}).
\]
We set $\mathcal{O} \{ -1\}$ the free $\mathcal{O}$-module $\mathrm{Hom}_{\mathcal{O}} (\mathcal{O} \{ 1\}, \mathcal{O})$ of rank $1$. Then for $n \in \mathbb{Z}$, we set $\mathcal{O} \{ n\}$ the free $\mathcal{O}$-module $(\mathcal{O} \{ \mathrm{sgn}(n)\})^{\otimes |n|}$ of rank $1$. For any $\mathcal{O}$-module $M$ and $n \in \mathbb{Z}$, we write $M\{ n\} \coloneqq M \otimes_{\mathcal{O}} \mathcal{O} \{ n\}$.

Let $(R, M_R) \to (S, M_S)$ be a morphism of $p$-torsion free rings with prelog structures.

\begin{definition}
The \emph{$p$-completed logarithmic cotangent complex} $\widehat{\mathbb{L}}_{(S, M_S)/(R, M_R)}^G$ is the derived limit of derived mod $p^n$ reductions of $\mathbb{L}_{(S, M_S)/(R, M_R)}$. In other words, 
\[
\widehat{\mathbb{L}}_{(S, M_S)/(R, M_R)}^G \coloneqq R\varprojlim_n \mathbb{L}_{(S/p^n, M_S)/(R/p^n, M_R)}^G.
\]
\end{definition}

For a $p$-torsion free ring $R$, $Q_R$ of $R$ denotes the (pre-)log structure $R \cap (R [\tfrac{1}{p}])^\ast \to R$. Let $\overline{W(k)}$ be as in Lemma \ref{chart-ssalg} and \ref{chart-ssfalg}. Fix a $\mathbb{Z}_p$-algebra $S$ with the prelog structure $\alpha_S \colon M_S \to S$ and an element $\pi \in M_S$. We mainly treat the case of special fibers of semistable formal schemes, but we treat more general case of $R_S^{\Box} / S$ and $R_{S, +}^{\Box} / S$ in which  
\begin{align*}
R_S^{\Box} &\coloneqq S[t_0^{\pm 1},\ldots,t_r^{\pm 1},t_{r+1},\ldots,t_d]/(t_0\cdots t_r - \alpha_S (\pi)), \\ 
R_{S,+}^{\Box} &\coloneqq S[t_0,\ldots,t_r,t_{r+1},\ldots,t_d]/(t_0\cdots t_r - \alpha_S (\pi))
\end{align*}
with the log structure of the chart $\mathbb{N}^{r+1} \sqcup_{\mathbb{N}} M_S$, where $\pi$ denotes its image in $S$ of $\alpha_S$. The algebras are equipped with log structures associated to $\mathbb{N}^{r+1} \sqcup_{\mathbb{N}} (M_S \backslash \{ 0\}) \to S$ given by $((a_i)_{0 \leq i \leq r}, c) \mapsto c\prod_{0 \leq i \leq r} t_i^{a_i}$ and they are denoted by $Q_S^{\Box}, Q_{S, +}^{\Box}$. The condition $S = \overline{W(k)}$ assures that the base changes for the \'etale map $R_{\overline{W(k)}}^{\Box} \to R$ takes the log structure $Q_{\overline{W(k)}}^{\Box} = Q_{R_{\overline{W(k)}}^{\Box}}$ to $Q_R$ by Lemma \ref{chart-ssalg}.

\begin{proposition}
\label{Rinfty-LG}
Let $K = C$. Let $(R_{\infty}, Q_{R_{\infty}})$ be the prelog ring previously described. Then we have
\begin{enumerate}
\item $\widehat{\mathbb{L}}_{(R_{\infty}, Q_{R_{\infty}})/(\mathcal{O}_C, Q_{\mathcal{O}_C})}^G \cong 0$, 
\item $\widehat{\mathbb{L}}_{(\mathcal{O}_C, Q_{\mathcal{O}_C})/(\mathcal{O}_C, 0)}^G \cong 0$. 
\end{enumerate}
\end{proposition}

\begin{proof}
First, we prove the first vanishing. Let $\mathcal{O}' \coloneqq \overline{W(k)}$, whose $p$-adic completion is $\mathcal{O}_C$. We take a model $R_{\mathcal{O}_1}^{\Box} \to R_1$ of $R^{\Box} \to R$ over $\mathcal{O}_1 \subset \mathcal{O}'$. We may assume $\mathcal{O}_1$ is a finite extension of $W(k)$. Then we take an increasing system of finite extensions of $\mathcal{O}_1$
\[
\mathcal{O}_1 \subset \mathcal{O}_2 \subset \cdots \subset \mathcal{O}_m \subset \cdots
\]
such that $p^{q/p^m} \in \mathcal{O}_m$ and the monoids of $p$-adic valuation of $\varinjlim_m \mathcal{O}_m$ and $\mathcal{O}'$ agree. Then we can assume $\mathcal{O}' = \varinjlim_m \mathcal{O}_m$. For $m \geq 1$, we put 
\[
R_m^{\Box} \coloneqq \mathcal{O}_m [t_0,\ldots,t_r,t_{r+1}^{\pm 1},\ldots,t_d^{\pm 1}]/(t_0 \cdots t_r - p^{q/p^m}).
\]
We also put
\begin{align*}
\varphi_m \colon R_{\mathcal{O}_1}^{\Box} = R_1^{\Box} &= \mathcal{O}_1 [t_0,\ldots,t_r,t_{r+1}^{\pm 1},\ldots,t_d^{\pm 1}]/(t_0 \cdots t_r - p^q) \\ 
&\to \mathcal{O}_m [t_0,\ldots,t_r,t_{r+1}^{\pm 1},\ldots,t_d^{\pm 1}]/(t_0 \cdots t_r - p^{q/p^m}) = R_m^{\Box}
\end{align*}
be the morphism defined by $t_i \mapsto t_i^{p^m}$ for $0 \leq i \leq d$, and $R_m$ be the base change of $R_1$ along $\varphi_m$. Then we obtain an inductive system $(R_m)$ whose maps are defined to be the $p$-powers of the coordinates. The $p$-adic completion of $R_{\infty, \mathcal{O}'} \coloneqq \varinjlim_m R_m$ is isomorphic to the algebra $R_{\infty}$. We equip $R_m$ with the log structure associated to the chart $M_m \coloneqq \mathbb{N}^{r+1} \sqcup_{\mathbb{N}} V_m$ where $V_m$ is the monoid of $p$-adic values of $\mathcal{O}_m$. Taking the associated log structures do not change $\mathbb{L}^G$ up to the natural quasi-isomorphisms by \cite[Theorem 8.20]{Ols05} and we may consider $\mathbb{L}^G$ as that of the associated log rings. Then by Proposition \ref{comp-LG-L}, Lemma \ref{comm-LG-fcolim} and Lemma \ref{chart-ssalg}, we have quasi-isomorphisms of the complexes of $R_{\infty, \mathcal{O}'}$-modules 
\begin{align*}
\widehat{\mathbb{L}}_{(R_{\infty}, Q_{R_{\infty}})/(\mathcal{O}_C, Q_{\mathcal{O}_C})}^G
&\cong \widehat{\mathbb{L}}_{(R_{\infty, \mathcal{O}'}, Q_{R_{\infty, \mathcal{O}'}})/(\mathcal{O}', Q_{\mathcal{O}'})}^G \\ 
&\cong (\varinjlim_m \mathbb{L}_{(R_m, M_m)/(\mathcal{O}_m, V_m)}^G)^\wedge \\ 
&\cong (\varinjlim_m R\Gamma (\Spec(R_m, M_m), \mathbb{L}_{\Spec(R_m, M_m)/\Spec(\mathcal{O}_m, V_m)}))^\wedge,
\end{align*}
where $(-)^\wedge$ denotes the derived $p$-adic completion and $\varinjlim$ is taken in the category of complexes. Then it suffices to show
\[ \varinjlim_m (R\Gamma (\Spec(R_m, M_m), \mathbb{L}_{\Spec(R_m, M_m)/\Spec(\mathcal{O}_m, V_m)}) \otimes_{\mathbb{Z}}^L \mathbb{Z}/p^n) \cong 0 \]
for $n \geq 1$. We may assume $R_m = R_m^{\Box}$ by \'etale base change. Let 
\begin{align*}
X_m &\coloneqq \Spec (\mathcal{O}_m [t_0,\ldots,t_r,t_{r+1}^{\pm 1},\ldots,t_d^{\pm 1}]/(t_0 \cdots t_r - p^{q/p^m})), \\ 
Y_m &\coloneqq \Spec (\mathcal{O}_m [u_0,\ldots,u_r,u_{r+1}^{\pm 1},\ldots,u_d^{\pm 1}]/(u_0 \cdots u_r - p^{q/p^m})).
\end{align*}
We construct the fiber product of the diagram 

\begin{equation}
\label{RmBox2}
\xymatrix{
& X_m \ar[d]^{\mathcal{M}_{X_m}} \\
Y_m \ar[r]^-{\mathcal{M}_{Y_m}} & \mathcal{L}og_{(\Spec (\mathcal{O}_m), V_m)}
}
\end{equation}
where $\mathcal{M}_{X_m}, \mathcal{M}_{Y_m}$ are respectively the log structures on $X_m, Y_m$ corresponding to the log structure $M_m$ of the algebra $R_m^{\Box}$ along the argument of \cite[Example 3.11]{Ols03}. Let $S_n$ be the symmetric group of degree $n$ and $\mathrm{Aut}(M_m) \cong S_{r+1}$ be the natural isomorphisms of groups with respect to the permutation of $t_0,\ldots,t_r$. We define schemes 
\begin{align*}
I_{m, \sigma} &\coloneqq \Spec R_{m, \sigma} \\
&= \Spec (\mathcal{O}_m [t_0,v_0^{\pm 1},\ldots,t_r,v_r^{\pm 1},t_{r+1}^{\pm 1},u_{r+1}^{\pm 1},\ldots,t_d^{\pm 1},u_d^{\pm 1}]/(t_0 \cdots t_r - p^{q/p^m}, v_0 \cdots v_r - 1))
\end{align*}
for each $\sigma \in S_{r+1}$, morphisms $p_{X_m} \colon I_{m, \sigma} \to X_m$ to be the canonical projection, and morphisms $p_{Y_m} \colon I_{m, \sigma} \to Y_m$ as 
\[
\Gamma(Y_m, \mathcal{O}_{Y_m}) \to \Gamma(I_{m, \sigma}, \mathcal{O}_{I_{m, \sigma}}) \qquad u_{\sigma(i)} \mapsto (v_i)^{-1} t_i 
\]
for $0 \leq i \leq r$, $u_i \mapsto u_i$ for $r+1 \leq i \leq d$. Let $\tilde{I}_m \coloneqq \coprod_{\sigma \in S_{r+1}} I_{m, \sigma}$. Any morphism $g \colon Z \to \tilde{I}_m$ of schemes is locally corresponds to a pair $(\iota, \sigma)$ where $\iota \colon g^\ast p_{X_m}^\ast \mathcal{M}_{X_m} \to g^\ast p_{Y_m}^\ast \mathcal{M}_{Y_m}$ is an isomorphism and $\sigma \in S_{r+1}$ is a permutation. Then the fiber product $I_m \coloneqq X_m \times_L Y_m$, where $L = \mathcal{L}og_{(\Spec (\mathcal{O}_m), V_m)}$, of the diagram \ref{RmBox2} is the quotient of $\tilde{I}_m$ by the equivalence relation defined by the subfunctor $\Gamma_m \subset \tilde{I}_m \times_{X_m \times Y_m} \tilde{I}_m$ locally consisting of pairs $\{(\iota, \sigma), (\iota, \sigma')\}$. The commutativity of the diagrams for both cases $\tau = \sigma$ and $\tau = \sigma'$
\[
\xymatrix{
M_m \ar[r]^-{e_i \mapsto e_{\tau(i)}} \ar[d]^{\beta_{X_m}} & M_m \ar[d]^{\beta_{Y_m}} \\ 
\overline{g^\ast p_{X_m}^\ast \mathcal{M}}_{X_m} \ar[r]^-{\iota} & \overline{g^\ast p_{Y_m}^\ast \mathcal{M}}_{Y_m},
}
\]
where $e_i$ is the $i$-th element of the canonical basis of $\mathbb{N}^{r+1}$, $\beta_{X_m}, \beta_{Y_m}$ are the charts induced by those of $\mathcal{M}_{X_m}, \mathcal{M}_{Y_m}$ and $g : Z \to \tilde{I}_m$ is a morphism of schemes, is an open condition by \cite[Lemma 3.5 (ii), (iii)]{Ols03}. We let 
\[
\Gamma_{m, \sigma, \sigma'} \coloneqq \Gamma_m \times_{(\tilde{I}_m \times_{X_m \times Y_m} \tilde{I}_m)} (I_{m, \sigma} \times_{X_m \times Y_m} I_{m, \sigma'}).
\] 
Then the two projections $\Gamma_{m, \sigma, \sigma'} \rightrightarrows \tilde{I}_m$ are open immersions, so the projections $\Gamma_m \rightrightarrows \tilde{I}_m$ are \'etale. Therefore, $\Gamma_m$ is an \'etale equivalence relation so that $I_m = \tilde{I}_m / \Gamma$ is an algebraic space and the canonical projection $\pi_m \colon I_m \to \tilde{I}_m$ is an \'etale morphism by \cite[Tag 02WW]{SP}. By Lemma \ref{LG-ring-asch}, the cotangent complex $\mathbb{L}_{I_{m, \sigma} / X_m}$ is naturally quasi-isomorphic to the complex of quasi-coherent sheaves associated with the free module 
\[
\mathbb{L}_{R_{m, \sigma} / R_m^{\Box}} = 
R_{m, \sigma} \langle d\log u_1,\ldots,d\log u_r, d\log u_{r+1},\ldots,d\log u_d \rangle [0]\]
concentrated in degree $0$ by \cite[III.3.2.7]{Ill71}. Then the morphism
\[ \mathbb{L}_{R_{m, \sigma} / R_m^{\Box}} \otimes_{\mathbb{Z}}^L \mathbb{Z}/p^n \to \mathbb{L}_{R_{m+1, \sigma} / R_{m+1}^{\Box}} \otimes_{\mathbb{Z}}^L \mathbb{Z}/p^n \]
is represented by the diagram 
\[
\xymatrix{
\Omega_{R_{m, \sigma} / R_m^{\Box}}^1 \ar[r]^-{p^n} \ar[d] & \Omega_{R_{m, \sigma} / R_m^{\Box}}^1 \ar[d] \\ 
\Omega_{R_{m+1, \sigma} / R_{m+1}^{\Box}}^1 \ar[r]^-{p^n} & \Omega_{R_{m+1, \sigma} / R_{m+1}^{\Box}}^1,
}
\]
where the vertical maps are defined as $d\log t_i \mapsto d\log t_i^p = pd\log t_i$ for $1 \leq i \leq d$. Any element $x \in \Omega_{R_{m, \sigma} / R_m^{\Box}}^1$ in the diagram maps to $0$ via the composition of $n$ vertical maps. Hence the colimit of the system of vertical maps in the diagram is $0$. Then the complex $\varinjlim (\mathbb{L}_{R_{m, \sigma} / R_m^{\Box}} \otimes_{\mathbb{Z}}^L \mathbb{Z}/p^n)$ is also $0$. The map $\tilde{I}_m \to Y_m$ is a faithfully flat morphism by \cite[Corollary 3.15]{Ols03}. 
We have the following natural quasi-isomorphisms of sheaves of modules on $I_m$ 
\[
L(p_{Y_m} \pi_m)^{\ast} (\mathbb{L}_{(R_m, M_m)/(\mathcal{O}_m, V_m)}^G)^{\verb|~|} \cong 
\mathbb{L}_{I_m / X_m} \cong 
L\pi_m^{\ast} (\mathbb{L}_{R_{m, \sigma} / R_m^{\Box}})^{\verb|~|}
\]
by the \'etale invariance of cotangent complexes, Lemma \ref{LG-ring-asch} and Lemma \ref{comp-LG-L}. Since $\pi_m \colon I_m \to \tilde{I}_m$ is \'etale and the natural map 
\[
\beta \colon Lp_{Y_m}^{\ast} (\mathbb{L}_{(R_m, M_m)/(\mathcal{O}_m, V_m)}^G)^{\verb|~|} \to (\mathbb{L}_{R_{m, \sigma} / R_m^{\Box}})^{\verb|~|}
\]
induces the above map, $\beta$ is a quasi-isomorphism. It follows that the natural map 
\[
\mathbb{L}_{(R_m^{\Box}, Q_{R_m^{\Box}}) / (\mathcal{O}_m, Q_{\mathcal{O}_m})}^G \otimes_{R_m^{\Box}}^L R_{\infty, \sigma} \to 
\mathbb{L}_{R_{m, \sigma} / R_m^{\Box}}
\]
is a quasi-isomorphism. 

Let $R_{\infty, \sigma} \coloneqq \varinjlim R_{m, \sigma}$. Then we have 
\begin{align*}
&(\varinjlim_m (\mathbb{L}_{(R_m^{\Box}, Q_{R_m^{\Box}}) / (\mathcal{O}_m, Q_{\mathcal{O}_m})}^G \otimes_{\mathbb{Z}}^L \mathbb{Z}/p^n)) \otimes_{R_{\infty}^{\Box}}^L R_{\sigma}^{\Box}) \\ 
&\cong \varinjlim_m (\mathbb{L}_{(R_m^{\Box}, Q_{R_m^{\Box}}) / (\mathcal{O}_m, Q_{\mathcal{O}_m})}^G \otimes_{\mathbb{Z}}^L \mathbb{Z}/p^n) \otimes_{R_m^{\Box}}^L R_{\sigma}^{\Box}) \\ 
&\cong \varinjlim_m (\mathbb{L}_{(R_m^{\Box}, Q_{R_m^{\Box}}) / (\mathcal{O}_m, Q_{\mathcal{O}_m})}^G \otimes_{R_m^{\Box}}^L R_{\infty, \sigma} \otimes_{\mathbb{Z}}^L \mathbb{Z}/p^n) \\ 
&\cong \varinjlim_m (\mathbb{L}_{R_{m, \sigma} / R_m^{\Box}} \otimes_{R_{m, \sigma}}^L R_{\infty, \sigma} \otimes_{\mathbb{Z}}^L \mathbb{Z}/p^n) \\ 
&\cong \varinjlim_m ((\mathbb{L}_{R_{m, \sigma} / R_m^{\Box}} \otimes_{\mathbb{Z}}^L \mathbb{Z}/p^n) \otimes_{R_{m, \sigma}}^L R_{\infty, \sigma}) \\
&= 0.
\end{align*}
Since the morphism $R_{\infty}^{\Box} \to R_{\infty, \sigma}$ is faithfully flat, the claim
\[ \varinjlim_m (\mathbb{L}_{(R_m^{\Box}, Q_{R_m^{\Box}}) / (\mathcal{O}_m, Q_{\mathcal{O}_m})}^G \otimes_{\mathbb{Z}}^L \mathbb{Z}/p^n) \cong 0 \]
follows by faithfully flat descent. 

For the second vanishing, we take an increasing system of prelog rings
\[
(\mathcal{O}_1, \mathbb{N}) \subset (\mathcal{O}_2, \mathbb{N}) \subset \cdots \subset (\mathcal{O}_m, \mathbb{N}) \subset \cdots
\]
such that the prelog structures are defined by the uniformizers $\pi_m \in \mathcal{O}_m$ and their colimit is isomorphic to $(\mathcal{O}, Q_{\mathcal{O}})$. As in the proof of the first assertion, it is enough to show that the colimit of the directed system of the complexes 
\[
\mathbb{L}_{A_1 / \mathcal{O}_1} 
\to \cdots 
\mathbb{L}_{A_m / \mathcal{O}_m} 
\to 
\mathbb{L}_{A_{m+1} / \mathcal{O}_{m+1}} 
\to \cdots 
\]
for $n \geq 1$ where 
\[
A_m \coloneqq \mathcal{O}_m [v^{\pm 1}] / (\pi_m v - \pi_m) \cong \mathcal{O}_m / \pi_m
\]
vanishes. Then each arrow of the system is the natural maps 
\[
(\pi_m / \pi_m^2) [-1] \to (\pi_{m+1} / \pi_{m+1}^2) [-1].
\]
Then the assertion follows. 
\end{proof}

\begin{remark}
We only deal with $\mathbb{L}^G$ of the smooth morphisms in Proposition \ref{Rinfty-LG} (1), and so we can avoid using log cotangent complex of Olsson for showing the vanishing by $\mathbb{L}^G = \Omega^1$. It is pointed out by T.\ Koshikawa. However we present more general method in the proof of the proposition.
\end{remark}

Assume $K = C$ and let $\mathfrak{X} = \Spf R$ and $X = \Spa(R[\tfrac{1}{p}], R)$. We give another proof of the semistable Hodge-Tate comparison theorem for $\tilde{\Omega}_{\mathfrak{X}}$ of \v{C}esnavi\v{c}ius-Koshikawa \cite[Theorem 4.11]{CK19}. Consider the transitivity triangle
\begin{equation}
\label{log-triangle}
\widehat{\mathbb{L}}_{(\mathcal{O}_C, Q_{\mathcal{O}_C})/(\mathbb{Z}_p, 0)}^G [-1] \otimes_{\mathcal{O}_C} \widehat{\mathcal{O}}_X^+ \to 
\widehat{\mathbb{L}}_{(\widehat{\mathcal{O}}_X^+, Q_{\widehat{\mathcal{O}}_X^+}) / (\mathbb{Z}_p, 0)}^G [-1] \to 
\widehat{\mathbb{L}}_{(\widehat{\mathcal{O}}_X^+, Q_{\widehat{\mathcal{O}}_X^+}) /(\mathcal{O}_C, Q_{\mathcal{O}_C})}^G [-1]
\end{equation}
of $p$-completed cotangent complexes on $X_{\proet}$, where $\mathbb{L}^G$ is calculated over the pro-\'etale site $X_{\proet}$. The third term of the triangle is $0$ by Lemma \ref{chart-ssalg}, Lemma \ref{chart-ssfalg}, Lemma \ref{comm-LG-Hi} and Proposition \ref{Rinfty-LG}. We have $\widehat{\mathbb{L}}_{(\mathcal{O}_C, Q_{\mathcal{O}_C})/(\mathcal{O}_C, 0)} \cong 0$ by Proposition \ref{Rinfty-LG} and so that $\widehat{\mathbb{L}}_{(\mathcal{O}, Q_{\mathcal{O}})/(\mathbb{Z}_p, 0)} \cong \mathcal{O} \{1\}$ by Lemma \ref{BKF-fund} and the canonical distinguished triangles. So we obtain a map 

\begin{equation}
\label{CKcomp-map}
\begin{aligned}
\widehat{\mathbb{L}}_{(\mathfrak{X}, Q_{\mathfrak{X}})/(\mathbb{Z}_p, 0)}^G [-1]
&\to R\Gamma_{\proet} (X, \widehat{\mathbb{L}}_{(\widehat{\mathcal{O}}_X^+, Q_{\widehat{\mathcal{O}}_X^+}) / (\mathbb{Z}_p, 0)}^G [-1]) \\ 
&\cong R\Gamma_{\proet} (X, \widehat{\mathbb{L}}_{(\mathcal{O}_C, Q_{\mathcal{O}_C})/(\mathbb{Z}_p, 0)}^G [-1] \otimes_{\mathcal{O}_C} \widehat{\mathcal{O}}_X^+) \\ 
&\cong R\Gamma_{\proet} (X, \widehat{\mathcal{O}}_X^+) \{1\}.
\end{aligned}
\end{equation}
We set 
\[
\Omega_{(\mathfrak{X}, Q_{\mathfrak{X}})/({\mathcal{O}_C}, Q_{\mathcal{O}})}^i \coloneqq \bigwedge^i \varprojlim_n \Omega_{(\mathfrak{X}/p^n, Q_{\mathfrak{X}})/({\mathcal{O}_C /p^n, Q_{\mathcal{O}_C})}}^1,
\]
where $Q_{\mathcal{O}_C} \coloneqq Q_{\mathcal{O}_{\Spf \mathcal{O}_C}}$. 

\begin{theorem}
\label{CKisom}
Let $K = C$. Then there is a canonical isomorphism
\[
H^1 (\widetilde{\Omega}_R) \cong \Omega_{(\mathfrak{X}, Q_{\mathfrak{X}})/(\mathcal{O}_C, Q_{\mathcal{O}_C})}^1 \{-1\}.
\]
Moreover, the above map is naturally extends to an isomorphism
\[
H^i (\widetilde{\Omega}_R) \cong \Omega_{(\mathfrak{X}, Q_{\mathfrak{X}})/(\mathcal{O}_C, Q_{\mathcal{O}_C})}^i \{-i\}.
\]
\end{theorem}

\begin{proof}
The last isomorphism follows by the first one and Lemma \ref{galg-tildeLambda}. For the first one, the proof is the same as in \cite[Section 8.2]{BMS18}. We show the morphism (\ref{CKcomp-map}) induces an isomorphism in degree $1$. The cohomology groups of $\widehat{\mathbb{L}}_{(R, Q_R)/(\mathbb{Z}_p, 0)}^G [-1]\{-1\}$ are given by $R$ in degree $0$ and $\Omega_{(R, Q_R)/(\mathcal{O}_C, Q_{\mathcal{O}_C})}^1 \{-1\}$ in degree $1$ by the triangle same as (\ref{log-triangle}). The \'etale invariance of functors $\widetilde{\Omega}$ (Lemma \ref{WnLambda-et-basechange}) and strict \'etale invariance of functors $\widehat{\mathbb{L}}^G$, which is seen by \cite[Lemma 8.22]{Ols05} and the \'etale invariance of cotangent complexes (\cite[Tag 08R2, Tag 08R3]{SP}), assure that we may assume $R = R^{\Box}$. By \cite[Lemma 8.16]{BMS18}, it suffices to show that

\begin{enumerate}

\item The induced morphism in degree $0$
\[
R^{\Box} = H^0 (\widehat{\mathbb{L}}_{(R^{\Box}, Q_{R^{\Box}})/(\mathbb{Z}_p, 0)}^G [-1]\{-1\}) \to 
H^0_{\cont} (\Delta, R_{\infty}^{\Box})
\]
is an isomorphism. 

\item The induced morphism in degree $1$
\begin{multline*}
\Omega_{(R^{\Box}, Q_{R^{\Box}})/(\mathcal{O}_C, Q_{\mathcal{O}_C})}^{1, \cont} \coloneqq 
\varprojlim_j \Omega_{(R^{\Box}/p^j, Q_{R^{\Box}})/(\mathcal{O}_C/p^j, Q_{\mathcal{O}_C})}^1 = 
H^1 (\widehat{\mathbb{L}}_{(R^{\Box}, Q_{R^{\Box}})/(\mathbb{Z}_p, 0)}^G [-1]\{-1\}) \\
\to H_{\cont}^1 (\Delta, R_{\infty}^{\Box})
\end{multline*}
is an isomorphism onto $(\zeta_p - 1)H_{\cont}^1 (\Delta, R_{\infty}^{\Box})$. 

\end{enumerate}

The first claim follows from the construction of (\ref{CKcomp-map}). For the second claim, we have a commutative diagram of natural maps 
\[
\xymatrix{
\widehat{\mathbb{L}}_{(R^{\Box}, Q_{R^{\Box}})/(\mathbb{Z}_p, 0)}^G \ar[r] \ar[d] & R\Gamma_{\cont} (\Delta, R_{\infty}^{\Box}) [1]\{1\} \ar[d]^{=} \\ 
R\Gamma_{\cont} (\Delta, \widehat{\mathbb{L}}_{(R_{\infty}^{\Box}, Q_{R_{\infty}^{\Box}})/(\mathbb{Z}_p, 0)}^G) \ar[r]^-{\cong} & R\Gamma_{\cont} (\Delta, R_{\infty}^{\Box}) [1]\{1\}
}
\]
by Proposition \ref{Rinfty-LG}. The $R$-module $(\zeta_p - 1) H_{\cont}^1 (\Delta, R_{\infty}^{\Box})$ is a free $R$-module of rank $d$ by \cite[Proposition 3.8]{CK19}. We describe a canonical basis of $(\zeta_p - 1) H_{\cont}^1 (\Delta, R_{\infty}^{\Box} \{1\})$. The map
\[ H_{\cont}^1 (\Delta, \mathcal{O}\{1\}) \otimes_{\mathcal{O}_C} R^{\Box} \to H_{\cont}^1 (\Delta, R_{\infty}^{\Box} \{1\}) \] induces an equality
\[
(\zeta_p - 1) H_{\cont}^1 (\Delta, \mathcal{O}_C \{1\}) \otimes_{\mathcal{O}_C} R^{\Box} = 
(\zeta_p - 1) H_{\cont}^1 (\Delta, \mathcal{O}_C \{1\} \otimes_{\mathcal{O}_C} R_{\infty}^{\Box})
\]
again by \cite[Proposition 3.8]{CK19}. We shall give some elements of
\[
(\zeta_p - 1) H_{\cont}^1 (\Delta, \mathcal{O}_C \{1\}) \cong 
(\zeta_p - 1) H_{\cont}^0 (\Delta, \widehat{\mathbb{L}}_{(\mathcal{O}_C, Q_{\mathcal{O}_C})/(\mathbb{Z}_p, 0)}^G).
\]
and show that they form a basis of $(\zeta_p - 1) H_{\cont}^1 (\Delta, \mathcal{O}_C \{1\}) \otimes_{\mathcal{O}_C} R^{\Box}$. 

By the natural quasi-isomorphism
\[
R\Gamma_{\cont} (\Delta, \widehat{\mathbb{L}}_{(\mathcal{O}_C, Q_{\mathcal{O}_C})/(\mathbb{Z}_p, 0)}^G) \to 
R\varprojlim_n R\Gamma_{\cont} (\Delta, \mathbb{L}_{(\mathcal{O}_C, Q_{\mathcal{O}_C})/(\mathbb{Z}_p, 0)}^G \otimes_{\mathbb{Z}}^L \mathbb{Z}/p^n),
\]
and the vanishing 
\[
R^1 \varprojlim H_{\cont}^{-1}(\Delta, \mathbb{L}_{(\mathcal{O}_C, Q_{\mathcal{O}_C})/(\mathbb{Z}_p, 0)}^G \otimes_{\mathbb{Z}}^L \mathbb{Z}/p^n) = 0,
\]
we shall take compatible systems of elements in each  
\[
H^0_{\cont} (\Delta, \mathbb{L}_{(\mathcal{O}_C, Q_{\mathcal{O}_C})/(\mathbb{Z}_p, 0)}^G \otimes_{\mathbb{Z}}^L \mathbb{Z}/p^n)
\]
for $n \geq 1$, and identify them with elements of $H^0 (\Delta, \widehat{\mathbb{L}}_{(\mathcal{O}_C, Q_{\mathcal{O}_C})/(\mathbb{Z}_p, 0)}^G)$. Here 
\[
R\Gamma_{\cont} (\Delta, \mathbb{L}_{(\mathcal{O}_C, Q_{\mathcal{O}_C})/(\mathbb{Z}_p, 0)}^G \otimes_{\mathbb{Z}}^L \mathbb{Z}/p^n)
\]
is represented by the total complex of the tensor product of $K_{\Omega_{\mathcal{O}_C / \mathbb{Z}_p}^1} (\delta_1 - 1,\ldots,\delta_d - 1)$ and the complex $\mathbb{Z} \overset{p^n}{\to} \mathbb{Z}$, where the right term is in degree $0$. The tensor product is written by
\[
\xymatrix{
\Omega_{\mathcal{O}_C / \mathbb{Z}_p}^1 \ar[rr]^-{\oplus_j (\delta_j - 1)} & & \bigoplus_j \Omega_{\mathcal{O}_C / \mathbb{Z}_p}^1 \ar[r] & \cdots \\ 
\Omega_{\mathcal{O}_C / \mathbb{Z}_p}^1 \ar[u]^{p^n} \ar[rr]^-{\oplus_j (\delta_j - 1)} & & \bigoplus_j \Omega_{\mathcal{O}_C / \mathbb{Z}_p}^1 \ar[u]^{p^n} \ar[r] & \cdots,
}
\]
where the top left term is in bidegree $(0, 0)$.
We remark that $\widehat{\mathbb{L}}_{(\mathcal{O}_C, Q_{\mathcal{O}_C}) / (\mathbb{Z}_p, 0)}^G$ and $\widehat{\mathbb{L}}_{(R_{\infty}^{\Box}, Q_{R_{\infty}^{\Box}}) / (\mathbb{Z}_p, 0)}^G$ are concentrated in degree $-1$ by the proof of Proposition \ref{Rinfty-LG}. We take the system \[
  ((0,\ldots,0,d\log \zeta_{p^n},0,\ldots,0))_n
  \in
  \varprojlim_n
  \ker \bigg(
  \bigoplus_j \Omega_{\mathcal{O}_C / \mathbb{Z}_p}^1 \overset{p^n}{\longrightarrow}
  \bigoplus_j \Omega_{\mathcal{O}_C / \mathbb{Z}_p}^1
  \bigg)
  \cong \bigoplus_j T_p (\Omega_{\mathcal{O}_C / \mathbb{Z}_p})
\]
where the $i$-th coordinate is nonzero, in the bottom right term in the diagram for $1 \leq  i \leq d$ and it is denoted by
\[
  d\log_i \in \bigoplus_j T_p (\Omega_{\mathcal{O}_C / \mathbb{Z}_p}).
\]
Here $d\log_i$ is sent to the element 
\[
f_i \in \Hom_{\cont} (\Delta, T_p (\Omega_{\mathcal{O}_C / \mathbb{Z}_p})) = H_{\cont}^1 (\Delta, \mathcal{O}_C \{1\}),
\]
where $f_i (\delta_i) = (d\log \zeta_{p^n})_n$ and $f_i (\delta_j) = 0$ for $i \neq j$, via the identification $\widehat{\mathbb{L}}_{\mathcal{O}_C / \mathbb{Z}_p}^G \cong \mathcal{O}_C \{1\}[1]$. Moreover, $\{f_i\}$ is a basis of the $\mathcal{O}_C$-module $(\zeta_p - 1) H_{\cont}^1 (\Delta, \mathcal{O}_C \{1\})$.

On the other hand,
\[ R\Gamma_{\cont} (\Delta, \mathbb{L}_{(R_{\infty}^{\Box}, Q_{R_{\infty}^{\Box}})/(\mathbb{Z}_p, 0)}^G \otimes_{\mathbb{Z}}^L \mathbb{Z}/p^n)\]
is represented by the total complex of the double complex
\[
\xymatrix{
\Omega_{R_{\infty}^{\Box} / \mathbb{Z}_p}^1 \ar[rr]^-{\oplus_j (\delta_j - 1)} & & \bigoplus_j \Omega_{R_{\infty}^{\Box} / \mathbb{Z}_p}^1
 \ar[r] & \cdots
\\ 
\Omega_{R_{\infty}^{\Box} / \mathbb{Z}_p}^1 \ar[u]^{p^n} \ar[rr]^-{\oplus_j (\delta_j - 1)} & & \bigoplus_j \Omega_{R_{\infty}^{\Box} / \mathbb{Z}_p}^1 \ar[u]^{p^n}
 \ar[r] & \cdots
}
\]
Here the top leftmost term
$\Omega_{R_{\infty}^{\Box} / \mathbb{Z}_p}^1$
is of degree $(0,0)$, and $d\log t_i \in \Omega_{R_{\infty}^{\Box} / \mathbb{Z}_p}^1$ of degree $(0,0)$ of the double complex.
So it is enough to show that $d\log t_i$ is equivalent to the element $d\log_i$ of degree $(1,-1)$.
The element $d\log t_i = p^n d\log (t_i^{1/p^n})$ of degree $(0,0)$ is equivalent to the element 
\[
((\delta_1 - 1)d\log (t_i^{1/p^n}),\ldots,(\delta_d - 1)d\log (t_i^{1/p^n}))
\]
of degree $(1, -1)$. For the $j$-th coordinate such that $i \neq j$, the claim follows by $(\delta_j - 1)d\log (t_i^{1/p^n}) = 0$. For the $i$-th coordinate, it follows from the following calculation
\begin{align*}
(\delta_i - 1)d\log (t_i^{1/p^n}) &= 
\delta_i d\log (t_i^{1/p^n}) - d\log (t_i^{1/p^n}) \\
&= d\log (\zeta_{p^n} t_i^{1/p^n}) - d\log (t_i^{1/p^n}) \\
&= d\log (\zeta_{p^n}).
\end{align*}
\end{proof}

\begin{remark}
In the smooth case, the isomorphism in Theorem \ref{CKisom} is the same as in \cite{BMS18} by construction.
In the general case of semistable formal schemes, it is the same as in \cite{CK19}, in which it is constructed by
the formal GAGA and Grothendieck existence theorems
(extended to non-noetherian settings by K.\ Fujiwara and F.\ Kato)
from the smooth case in \cite{BMS18}. 
\end{remark}

\section{Logarithmic \texorpdfstring{$p$}{p}-adic Cartier isomorphisms for \texorpdfstring{$\widetilde{W_n \Omega_R}$}{}}
\label{sec;LdRWCisom}

In this section, we prove a semistable analogue of the $p$-adic Cartier isomorphism of \cite[Theorem 9.2, (iii)]{BMS18}. 
As in the previous section, let $R$ be a $p$-adically complete $\mathcal{O}$-algebra which admits a fixed \'etale coordinate map to $R^{\Box}$, and let its variants ($R_{\infty}$, $A(R)$, and so on.) be the ones defined as in the previous section with fixed \'etale coordinate maps $\Box$. In the following arguments, we often use the notation in Lemma \ref{chart-ssalg}. In Section \ref{LdRWC-LdRWc}, we recall the definition of the log de Rham-Witt complexes in \cite{Mat15}. We show the main isomorphisms of special fibers using the method of Illusie-Raynaud \cite{Ill79}, \cite{IR83} in Section \ref{LdRWC-Cisom}. We prove the main theorem in Section \ref{LdRWC-padicCisom}. 

\begin{remark}
While there is a result of prismatic variant \cite{Mol20} of the smooth case and it may be formulated on the prismatic cohomology of a perfect base $(A, I, R^{\flat} \backslash \{ 0\})$ as in \cite[Example 3.4, (iii)]{Kos20}, we only deal with $A_{\inf}$-cohomology of \cite{BMS18}, \cite{CK19}. 
\end{remark}

\subsection{Log de Rham-Witt complexes}
\label{LdRWC-LdRWc}

In this section, the differential graded algebras are assumed to be commutative. For a prelog ring $(A, M_1)$ and an integer $n \geq 1$, we denote by $W_n (A, M_1)$ the Witt prelog ring $(W_n (A), M_1)$ where the structure morphism $M_1 \to W_n (A)$ is given by $m \mapsto [\alpha (m)]$. We recall the construction of the log de Rham-Witt complex of Matsuue \cite{Mat15}. 

\begin{definition}[{\cite[Section 3]{Mat15}, \cite[Definition 1.1.1, Definition 1.2.4]{HS18}}]
Let $\theta \colon (A, M_1, \alpha) \to (B, M_2, \beta)$ be a morphism of prelog rings. Assume that $B$ is equipped with a divided power structure $(I, \{ \gamma_n\})$. 

\begin{enumerate}

\item For a $B$-module $M$, an $A$-linear derivation $D \colon B \to M$ is called \emph{a pd-derivation} if we have the equality $D(\gamma_i (b)) = \gamma_{i-1} (b) D(b)$ for any $i \geq 1$ and $b \in I$. We have the universal pd-derivation, which we denote by $d \colon B \to \breve{\Omega}_{B/A}^1$. It can be constructed as a quotient module of the module of differentials $\Omega_{B/A}^1$. 

\item For a $B$-module $M$, a log derivation $(d \colon B \to H, \delta \colon M \to H)$ is called \emph{a log pd-derivation} if $d$ is a pd-derivation. we have the universal log pd-derivation, which is denoted by 
\[
(d \colon S \to \breve{\Omega}_{(B, M_2)/(A, M_1)}^1, \delta \colon M \to \breve{\Omega}_{(B, M_2)/(A, M_1)}^1).
\]
It can be constructed as a quotient module of the module of log differentials $\Omega_{(B, M_2)/(A, M_1)}^1$. 

\item  \emph{A log differential graded $(B, M_2)/(A, M_1)$-algebra} is a triple $(E^\ast, d, \partial)$, where $(E^\ast, d)$ is a differential graded $B/A$-algebra, $\partial \colon M_2 \to E^1$ is a morphism of monoids, and $(d, \partial)$ is a log derivation such that $d \partial = 0$. 

\item \emph{A log $F$-$V$-procomplex} for $(B, M_2)/(A, M_1)$ consists of the following data: $(\mathcal{W}_n^\ast, R, F, V)$ 

\begin{enumerate}
\item A log differential graded $W_n (B, M_2) / W_n (A, M_1)$ algebra $\mathcal{W}_n^\ast = \oplus_{i \geq 0} \mathcal{W}_n^i$ for each integer $n \geq 1$. 
\item Morphisms $R \colon \mathcal{W}_{n+1}^\ast \to R_* \mathcal{W}_n^\ast$ of differential graded $W_{n+1} (A)$-algebras for $n \geq 1$. 
\item Morphisms $F \colon \mathcal{W}_{n+1}^\ast \to F_{\ast} \mathcal{W}_n^\ast$ of graded $W_{n+1} (A)$-algebras for $n \geq 1$. 
\item Morphisms $V \colon F_{\ast} \mathcal{W}_n^\ast \to \mathcal{W}_{n+1}^\ast$ of graded abelian groups for $n \geq 1$. 
\end{enumerate}
such that the following conditions are satisfied.
\begin{enumerate}
\item For each $n \geq 1$, $\lambda_n$ commutes with $F$, $V$ (and $R$). Here we denote $\lambda_r \colon W_n (B) \to \mathcal{W}_r^0$ the base map of the differential graded $W_n (B, M_2) / W_n (A, M_1)$-algebra $\mathcal{W}_n$ for $r \geq 1$. 
\item $R$ commutes with both $F$ and $V$. 
\item $FV = p$.
\item $FdV = d$.
\item $V(F(x)y) = xV(y)$ for $x \in \mathcal{W}_{n+1}^\ast$, $y \in \mathcal{W}_{n}^\ast$ and $n \geq 1$. 
\item $Fd\lambda_{n+1} ([b]) = \lambda_n ([b])^{p-1} d\lambda_n ([b])$ for $b \in B$ and $n \geq 1$. 
\item $F\partial_{n+1} (m) = \partial_n (m)$ for $n \geq 1$ and $m \in M_2$. 
\end{enumerate}
We remark that the last condition, which is not contained in the non-log case in \cite{LZ03}, is newly added in \cite{Mat15}.

\end{enumerate}

\end{definition}

There is an initial object in the category of log $F$-$V$-procomplexes. It can be constructed as a quotient of the dg algebra of universal pd-derivation $\breve{\Omega}_{W_n (B, M_2) / W_n (A, M_1)}^\ast$, hence as a quotient of the dg algebra $\Omega_{W_n (B, M_2) / W_n (A, M_1)}^\ast$. 

\begin{theorem}[{\cite[Proposition 3.5]{Mat15}}]
The category of log $F$-$V$-procomplexes for $(B, M_2)/(A, M_1)$ has the initial object. We denote it by 
\[
\{((W_n \Omega_{(B, M_2)/(A, M_1)}^\ast, d, d\log)_n, R, F, V)\}
\]
and call it the \emph{$n$-truncated log de Rham-Witt complex for $(B, M_2)/(A, M_1)$}. 
\end{theorem}

We put 
\[
W\Omega_{(B, M_2)/(A, M_1)}^l \coloneqq \varprojlim_n W_n \Omega_{(B, M_2)/(A, M_1)}^l
\]
for $l \geq 0$ and call it the \emph{total log de Rham-Witt complex}.

\begin{definition}[{cf.\ \cite[Section 4]{Mat15}}]
\label{dRW-Box+i}

\begin{enumerate}
\item Fix integers $1 \leq r \leq d$. For a function $k \colon \{0,\cdots,d\} = [0, d] \to \mathbb{Z}_{\geq 0} [1/p] \sqcup \{ p^{-\infty}\}$, let $\Supp k \coloneqq [0, d] \backslash (k^{-1} (\{0\}))$ and $k_i \coloneqq k(i)$. \emph{A weight} is a function $k \colon [0, d] \to \mathbb{Z}_{\geq 0} [1/p] \sqcup \{ p^{-\infty}\}$ such that the following properties are satisfied.  

\begin{enumerate}
\item $k_i \neq p^{-\infty}$ for every $i$ with $r < i \leq d$. 
\item $[0, r] \not\subset \Supp k^+$ where $k^+$ is the associated weight without poles defined to be 
\[
(k^+)_i \coloneqq 
\begin{cases}
0 & (k_i = p^{-\infty}), \\ 
k_i & (k_i \neq p^{-\infty}).
\end{cases}
\]
\end{enumerate}

\item A weight $k$ is \emph{integral} if $k([0, d]) \subset \mathbb{Z}_{\geq 0} \sqcup \{ p^{-\infty}\}$. \emph{A weight without poles} is a weight $k$ with $k^{-1} (\{ p^{-\infty} \}) = \emptyset$. For each weight $k$, we fix a total order $\preceq_{k}$ on $\Supp k^+ = \{ i_1,\ldots,i_s\}$ such that the function $(\mathrm{ord}_p i_s)_s$ is weakly $\preceq_{k}$-increasing with respect to $s$ and assume that the sequence $(i_t)_t$ is $\preceq_{k}$-increasing. We set $\tau (k_{i_l}) \coloneqq -\mathrm{ord}_p k_{i_l}$ and $\mu (k_{i_l}) \coloneqq \mathrm{max} (0, \tau (k_{i_l}))$. We call subsets $\{i_t,\ldots,i_{t+m}\} \subset \{ i_1,\ldots,i_s\}$ for $0 \leq t \leq t+m \leq s$ \emph{intervals of $\Supp k^+$}. 
\item If $I = \{i_t,\ldots,i_{t+m}\}$ is an interval of $\Supp k$, let $k_I$ denote the restriction of $k$ to $I$. We set $\tau (I) = \tau (k_I) \coloneqq \tau (k_{i_t})$ and $\mu (I) = \mu (k_I) \coloneqq \mu (k_{i_t})$

\item We say $(I_{-\infty}, I_0,\ldots,I_l)$ is \emph{a partition of $\Supp k$} if the following conditions are satisfied. 

\begin{enumerate}
\item $I_{-\infty} = k^{-1} (\{p^{-\infty}\})$. 
\item There is a decomposition into intervals $\Supp k = I_{-\infty} \sqcup I_0 \sqcup \cdots \sqcup I_l$ such that the elements of $I_j$ are $\preceq_{k}$-smaller than those of $I_{j+1}$ for $j = 0, \ldots, l-1$. 
\item $I_1,\ldots,I_l$ are not empty ($I_0$ may be an empty set). 
\end{enumerate}

\end{enumerate}
\end{definition}

\subsection{Log $p$-basic elements}
In this section, we generalize the result of \cite{LZ03} and \cite{Mat15}, calculating generators of log de Rham-Witt complexes of prelog rings that is related to the formal schemes of \v{C}esnavi\v{c}ius-Koshikawa type. For a $\mathbb{Z}_p$-algebra $S$ with a prelog structure $\alpha_S \colon M_S \to S$ and an element $\pi \in M_S$, we define the rings $R_S^{\Box} / S$ and $R_{S, +}^{\Box} / S$ as 
\begin{align*}
R_S^{\Box} &\coloneqq S[t_0^{\pm 1},\ldots,t_r^{\pm 1},t_{r+1},\ldots,t_d]/(t_0\cdots t_r - \alpha_S ({\pi})), \\ 
R_{S,+}^{\Box} &\coloneqq S[t_0,\ldots,t_r,t_{r+1},\ldots,t_d]/(t_0\cdots t_r - \alpha_S ({\pi}))
\end{align*}
with the log structure of the chart $N_S \coloneqq \mathbb{N}^{r+1} \sqcup_{\mathbb{N}} M_S$. For a weight without poles $k$ and a partition $\mathcal{P} = (I_{-\infty} = \emptyset, I_0, \ldots, I_l)$ of $\Supp k$, we define $t^{k(I_s)} \coloneqq \prod_{i \in I_s} t_i^{k_i}$ for $s \in [0, l]$. We shall define some essential elements in log de Rham-Witt complexes. 

\begin{definition}[{cf.\ \cite[Section 4]{Mat15}}]
\label{dRW-Box+ii}
For an integral weight $k$ without poles and an interval $k$ in $\Supp(k)$, We define symbols $p^{\tau(k_I)} dt^{k_I}$ as elements \[
(t^{p^{\tau(k_I)} k_I})^{(p^{-\tau(k_I)} - 1)} dt^{p^{\tau(k_I)} k_I} = 
(t^{k_I - p^{\tau(k_I)} k_I}) dt^{p^{\tau(k_I)} k_I}
\]
in $\Omega_{R_{S,+}^{\Box} / S}^1$. For an interval $I$ in $\Supp k$, We also define symbols 
\[
F^{-\tau(I)} (d[t]^{p^{\tau(I)} k_I}) \coloneqq [t]^{k_I - p^{\tau(I) k_I} d[t]^{p^{\tau(I)} k_I}}.
\]
For a weight $k$ and a partition $\mathcal{P} = (I_{-\infty}, I_0,\ldots,I_l)$ of $\Supp(k)$, we put $\mathcal{P}^+$ the partition $(\emptyset, I_0,\ldots,I_l)$ of $\Supp(k^+)$.
\begin{enumerate}
\item If $k$ is an integral weight, let $(k, \mathcal{P})$ be a pair of the weight $k$ and a partition $\mathcal{P} = (I_{-\infty}, I_0,\ldots,I_l)$ of $\Supp k$. We associate the \emph{$p$-basic element}
\[
e(k^+, \mathcal{P}^+) \coloneqq t^{k_{I_0}} (p^{\tau(k_{I_1})} dt^{k_{I_1}}) \cdots (p^{\tau(k_{I_l})} dt^{k_{I_l}})
\]
of $\Omega_{R_{S,+}^{\Box} / S}^l$ to the pair $k^+, \mathcal{P}^+$ and the \emph{log $p$-basic element}
\[
\epsilon (k, \mathcal{P}) \coloneqq \left( \prod_{i \in I_{-\infty}} d\log t_i \right) \cdot e(k^+, \mathcal{P}^+)
\]
in $\Omega_{(R_{S,+}^{\Box}, N_S) / (S, M_S)}^{|I_{-\infty}| + l}$ to the pair $(k, \mathcal{P})$. 

\item Let $\rho_1 \in [0, l]$ denote the greatest integer such that $-\tau (k_{I_{\rho_1}}) = \mathrm{ord}_p k_{I_{\rho_1}} < 0$. For a triple $(\xi, k, \mathcal{P})$ which consists of a weight without poles $k$, an element $\xi = V^{\mu(k)} \eta \in V^{\mu(k)} W(R_{S, +}^{\Box})$ and a partition $\mathcal{P}$ of $\Supp k$, we associate a \emph{basic Witt differential} $e(\xi, k^+, \mathcal{P}^+)$ in $W \Omega_{R_{S,+}^{\Box} / S}^l$ as follows. 

\begin{enumerate}

\item If $I_0 \neq \emptyset$, 
\begin{align*}
\epsilon' \coloneqq &V^{\mu(I_0)} (\eta [t]^{p^{\mu(I_0)} k_{I_0}}) (d V^{\mu(I_1)} [t]^{p^{\mu(I_1)} k_{I_1}}) \cdots (d V^{\mu(I_{\rho_1})} [t]^{p^{\mu(I_{\rho_1})} k_{I_1}}) \\ 
\cdot \ &(F^{-\tau(I_{\rho_1 +1})} d[t]^{p^{\tau(I_{\rho_1 +1})} k_{I_{\rho_1 +1}}}) \cdots (F^{-\tau(I_{\rho_2})} d[t]^{p^{\tau(I_{\rho_2})} k_{I_{\rho_2}}}). 
\end{align*}
\item If $I_0 = \emptyset$ and $k$ is not integral, 
\begin{align*}
\epsilon' \coloneqq &(d V^{\mu(I_1)} (\eta [t]^{p^{\mu(I_1)} k_{I_1}})) (d V^{\mu(I_2)} [t]^{p^{\mu(I_2)} k_{I_1}}) \cdots (d V^{\mu(I_{\rho_1})} [t]^{p^{\mu(I_{\rho_1})} k_{I_1}}) \\ 
\cdot \ &(F^{-\tau(I_{\rho_1 +1})} d[t]^{p^{\tau(I_{\rho_1 +1})} k_{I_{\rho_1 +1}}}) \cdots (F^{-\tau(I_{\rho_2})} d[t]^{p^{\tau(I_{\rho_2})} k_{I_{\rho_2}}}). 
\end{align*}
\item If $I_0 = \emptyset$ and $k$ is integral, 
\[
\epsilon' \coloneqq \eta (F^{-\tau(I_1)} d[t]^{p^{\tau(I_1)} k_{I_1}}) \cdots (F^{-\tau(I_{\rho_2})} d[t]^{p^{\tau(I_{\rho_2})} k_{I_{\rho_2}}}). 
\]

\end{enumerate}

\item For a weight $k$, let $(\xi, k, \mathcal{P})$ be a triple such that $\xi \in V^{\mu (k^+)} W(\mathcal{O}_C)$ and $\mathcal{P}$ be a partition of $\Supp k$. We associate the \emph{log basic Witt differential}
\[
\epsilon (\xi, k, \mathcal{P}) \coloneqq \left( \prod_{i \in I_{-\infty}} d\log [t_i] \right) \cdot e(\xi, k^+, \mathcal{P}^+)
\]
in $W \Omega_{(R_{S,+}^{\Box}, N_S) / (S, M_S)}^{|I_{-\infty}| + l}$. 
\end{enumerate}

\end{definition}

We give an explicit generators of the total log de Rham-Witt complex $W\Omega_{(R_{S,+}^{\Box}, N_S) / (S, M_S)}$. The smooth case is proved by Langer-Zink \cite[Proposition 2.3]{LZ03}. Moreover, the SNCD case \cite[Proposition 4.3]{Mat15} and the semistable case where $\pi = 0$ \cite[Proposition 4.5]{Mat15} are proved by Matsuue.

\begin{proposition}
\label{dRW-RBox+-Basis}
Any element in $W\Omega_{(R_{S,+}^{\Box}, N_S) / (S, M_S)}$ has a unique expression as a convergent sum 
\[
\sum_{(k, \mathcal{P})} \epsilon (\xi_{k, \mathcal{P}}, k, \mathcal{P})
\]
of log basic Witt differentials, where $k$ runs over weights and $\mathcal{P}$ over all partitions of $\Supp k$. Here a \emph{convergent sum} means for any $m \geq 1$, we have $\xi_{k, \mathcal{P}} \in V^m W(R_{S, +}^{\Box})$ for all but finitely many weights $k$. 
\end{proposition}

First, we prove the degree $0$ part of the Proposition \ref{dRW-RBox+-Basis}.

\begin{lemma}
\label{convsum-W}
Any element $\xi$ of $W(R_{S, +}^{\Box})$ may be uniquely written as a convergent sum
\[
\xi = \sum_k V^{\mu(k)} (\eta_k [t]^{p^{\mu(k)} k})
\]
in the sense of Proposition \ref{dRW-RBox+-Basis} i.e., for any integer $m \geq 1$, it suffices that $V^{\mu(k)} \eta_k \in V^m W(S)$ for almost all $k$. Here $k$ runs over all weights without poles. 
\end{lemma}

\begin{proof}
For any element $\xi \in W(R_{S, +}^{\Box})$, we may choose a unique element $a_k \in S$ for each integral weight without poles $k$ which satisfies 
\[
\xi - \sum_{\substack{\mathrm{integral} \ k \ \mathrm{without \ poles}}} [a_k][t]^k \in V W(R_{S, +}^{\Box}).
\]
By the induction, we can express
\[
\xi = \sum_{\substack{m \geq 0, \\ \mathrm{integral} \ k \ \mathrm{without \ poles}}} V^m ([a_{k,m}][t]^k).
\]
For the greatest nonnegative integer $\rho_{k, m}$ such that $p^{-\rho_{k, m}} k$ is integral and $\rho_{k, m} \leq m$, We have 
\[
V^m ([a_{k, m}] [t]^k) = V^{m-\rho_{k, m}} (V^{\rho_{k, m}} [a_{k, m}] \cdot [t]^{p^{-\rho_{k, m}} \cdot k}).
\]
For each weight without poles $k$, let 
\[
\eta_k \coloneqq \sum_{i \geq \mu(k)} V^{\rho_{p^i k, i}} [a_{p^i k, i}] = \sum_{i \geq \mu(k)} V^{i - \mu(k)} [a_{p^i k, i}] 
\]
be the convergent sum. Then we have the unique expression of the proposition. 
\end{proof}

Next, we consider the case of the log de Rham complex $\Omega_{(R_{S,+}^{\Box}, N_S) / (S, M_S)}^\ast$. 

\begin{lemma}[{cf. \cite[Proposition 2.1]{LZ03}, \cite[Lemma 4.1, Section 4.2]{Mat15}}]
\label{Lambda-basis}
The log $p$-basic elements in Definition \ref{dRW-Box+ii} form a basis of the log de Rham complex $\Omega_{(R_{S,+}^{\Box}, N_S) / (S, M_S)}^\ast$ as an $S$-module. 
\end{lemma}

\begin{proof}
For an integer $j \in [r+1, d]$, let $d\log t_j$ denote a symbol which satisfies $t_j d\log t_j = dt_j$. For an integer $l \geq 1$, $\Omega_{(R_{S,+}^{\Box}, N_S) / (S, M_S)}^l$ has a basis

\begin{equation}
\label{Lambda+-basis}
d\log t_{i_1} \cdots d\log t_{i_m} \cdot \left( \prod_{i \in [0, d]} t_i^{k_i} \right) \cdot d\log t_{j_1} \cdots d\log t_{j_{l-m}}
\end{equation}
as an $S$-module, where $k$ runs over all integral weights without poles, $i_1,\ldots,i_m \in [0, r] \cap k^{-1} (\{0\})$ and $j_1,\ldots,j_{l-m} \in \Supp k^+$. Let $\Omega^l (k, I) \subset \Omega_{(R_{S,+}^{\Box}, N_S) / (S, M_S)}^l$ be the free $S$-submodule which has the basis of the form (\ref{Lambda+-basis}) for the integral weight without poles $k$ and the subset $I = \{i_1,\ldots,i_m \} \subset [0, r] \cap k^{-1} (\{0\})$. Then we have the decomposition 
\[
\Omega_{(R_{S,+}^{\Box}, N_S) / (S, M_S)}^l = \bigoplus_{(k, I)} \Omega^l (k, I)
\]
of free $S$-modules. The rank of the $S$-module $\Omega^l (k, I)$ is $\binom{s}{l-m}$ where $s \coloneqq |\Supp k|$. On the other hand, the number of log $p$-basic elements for pairs $(k, \mathcal{P})$ such that $I_{-\infty} = I$ is also $\binom{s}{l-m}$. Then it suffices to show that all of the elements of the form (\ref{Lambda+-basis}) can be written as $S$-linear combinations of the log $p$-basic elements. 

For $l = 1$, it follows from the proof of \cite[Lemma 2.2]{LZ03}. For the case of higher degree, it is enough to show that any product of log $p$-basic elements of degree $1$ is an $S$-linear combination of log $p$-basic elements. We use the induction on $l$ and assume the $l$-th degree case is true. We consider the elements of the form 
\begin{equation}
\label{basis-induction}
\epsilon (k, \mathcal{P}) \cdot \epsilon(h, \mathcal{P'} = (J_{-\infty}, J_0, J \coloneqq J_1))
\end{equation}
for all integral weights $k, h$ of degree $l, l'$. We show that they are $S$-linear combinations of log $p$-basic elements. It may be easily reduced to
the case $l' = 1$ and $J_{-\infty} = \emptyset$. We firstly reduce to the case $I_0 = J_0 = \emptyset$. To confirm this, it suffices to confirm the multiplication of any monomial $t_j^h$ in $R_{S, +}^{\Box}$ and any log $p$-basic element of degree $l+1$ is an $S$-linear combination of $p$-basic elements of degree $l+1$. To show this, we can inductively apply again \cite[Lemma 2.2]{LZ03} to $t_j^h p^{\tau(k_{I_i})} dt^{k_{I_i}}$ for $i = l,\ldots,1$. 

Moreover, we explain that we may also assume $[0, r] \not\subset \Supp k^+ \cup \Supp h^+$. To confirm this, the equation 
\[
p^{\tau(k_I)} dt^{k_I} = 
p^{\tau(k_I)} t_{\gamma}^{k_{\gamma}} (dt^{h_{J \backslash \{ \gamma \}}} + 
t^{h_{J}} (t_{\gamma}^{k_{\gamma}})^{-1} d\log t_{\gamma}^{h_{\gamma}}),
\]
where $\gamma \in [0, r]$ and $I = I_0,\ldots,I_l$ and $t_0 \ldots t_r = \pi$ reduces to the case of a multiplication of any monomial and any $l+1$ products of log $p$-basic elements of degree $1$, which comes down to the case $[0, r] \not\subset \Supp k^+ \cup \Supp h^+$ by induction on $l$ and the last argument. 

Then we assume that $I_0 = J_0 = \emptyset$. We define a morphism of $S$-modules
\[
\alpha \colon \Omega_{(R_{S,+}^{\Box}, N_S) / (S, M_S)}^l \to \Omega_{(R_{S,+}^{\Box}, N_S) / (S, M_S)}^l
\]
by
\[ (\prod t_s^{k_s}) (\prod d\log t_i) (\prod dt_j) \mapsto (\prod t_s^{p k_s}) (\prod t_j^{(p-1)}) (\prod d\log t_i) (\prod dt_j), \]
in which $k$ is an integral weight without poles. This map is well-defined, multiplicative as for elements associated to weights $k^{(1)},\ldots,k^{(u)}$ as long as $[0, r] \not\subset \cup_i \Supp k^{(i)}$ and satisfies
\[
d\alpha = p\alpha d.
\]
Let $\mu$ be the minimum number in $\tau(k_{I_1}),\ldots,\tau(k_{I_l}), \tau(h_J)$. Then the element (\ref{basis-induction}) with assumptions above is written as
\[
\alpha^{\mu} (p^{\mu + \tau(k_{I_1})} dt^{p^{-\mu} k_{I_1}} \cdots p^{\mu + \tau(k_{I_l})} dt^{p^{-\mu} k_{I_l}} \cdot p^{\mu + \tau(h_J)} dt^{p^{-\mu} h_J}).
\]
Since $\mathrm{min} (\mu + \tau(k_{I_1}),\ldots,\mu + \tau(k_{I_l}), \mu + \tau(h_J)) = 0$, the element in the bracket is a differential of $l$ multiplication of log $p$-basic elements of degree $1$, which is an $S$-linear combination of log $p$-basic elements of degree $l$ by induction on $l$. The claim follows from the fact that $\alpha$ and $d$ preserve the log $p$-basic elements. 
\end{proof}

Now we prove Proposition \ref{dRW-RBox+-Basis}. We use the morphisms 
\[
\omega_m \colon W_m \Omega_{(R_{S,+}^{\Box}, N_S) / (S, M_S)}^\ast \to \Omega_{(R_{S,+}^{\Box}, N_S) / (S, M_S), w_m}^\ast
\]
of log phantom components, where $\Omega_{(R_{S,+}^{\Box}, N_S) / (S, M_S), w_m}^\ast$ is the $W(R_{S, +}^{\Box})$-modules $\Omega_{(R_{S,+}^{\Box}, N_S) / (S, M_S), w_m}$ obtained by the restriction of scalars via the morphism of the $m$-th Witt polynomial $w_m \colon W(R_{S, +}^{\Box}) \to R_{S, +}^{\Box}$ defined in \cite[Section 3.8]{Mat15}.

\begin{proof}[Proof of Proposition \ref{dRW-RBox+-Basis}]
Any element in $W\Omega_{(R_{S,+}^{\Box}, N_S) / (S, M_S)}^\ast$ can be written as a convergent sum of the following form 

\begin{equation}
\label{simplebase-WLambda}
d\log [t]_{j_1} \cdots d\log [t]_{j_m} \cdot V^{\mu{k_0}} (\xi_0 
[t]^{p^{\mu(k_0)} k_0}) dV^{\mu{k_1}} (\xi_1 [t]^{p^{\mu(k_1)} k_1}) \cdots dV^{\mu{k_{s}}} (\xi_{s} [t]^{p^{\mu(k_{s})} k_{s}})
\end{equation}
by Lemma \ref{convsum-W}. We may assume that $[0, r] \not\subset \cup_{i=1}^{m} \Supp k$ by the argument of the proof of Lemma \ref{Lambda+-basis}. So we can see that the elements of the form (\ref{simplebase-WLambda}) can be written as sums of log $p$-basic Witt differentials by the proof of \cite[Proposition 4.3]{Mat15} because $F, V, d$ on log $p$-basic Witt differentials do not change the condition of $[0, r] \not\subset \Supp k$. The uniqueness of the expressions is seen by \cite[Proposition 4.2]{Mat15}, the formulas (1),(2),(3) in p.38 of \cite{Mat15} and Lemma \ref{Lambda-basis}. 
\end{proof}

Before the proof of the localized version of Lemma \ref{dRW-RBox+-Basis}, we give the \'etale invariance of the log de Rham-Witt complexes, which improves the result \cite[Proposition 3.7]{Mat15}, removing the assumption that a base ring $R$ is $F$-finite. 
\begin{proposition}
\label{et-bc-ldRW}
We fix an integer $n \geq 1$. Let $(R, P) \to (S, Q)$ be a morphism of prelog rings and $S \to S'$ be an \'etale morphism of rings. Then the natural map
\[
W_n \Omega_{(S', Q)/(R, P)}^{\ast} \to W_n (S') \otimes_{W_n (S)} W_n \Omega_{(S, Q)/(R, P)}^{\ast}
\]
is an isomorphism.
\end{proposition}

\begin{proof}
We can prove the proposition along the method of the proof of \cite[Proposition 1.7]{LZ03}. The assumption that the base ring is $F$-finite in \cite[Proposition 1.7]{LZ03} is only used to verify that $W_n (S) \to W_n (S')$ is \'etale, but the \'etaleness of the map is true for a general $S$ by \cite[Theorem 10.4]{BMS18}.
\end{proof}
Then we prove an $R_S^{\Box}$-analogue of Proposition \ref{dRW-RBox+-Basis}. In the smooth case, it is proved in \cite[Section 10.4]{BMS18}. We define \emph{Laurent weight} without poles $k \colon [0, d] \to \mathbb{Z}[1/p]$ as functions such that $k([r+1, d]) \subset \mathbb{Z}_{\geq 0} [1/p]$ and $[0, r] \not\subset \Supp k$. For each Laurent weight without poles $k$, we fix the total order $\preceq_k$ of $[0, d]$ such that $\mathrm{ord}_p k \colon [0, d] \to \mathbb{Z} \cup \{ \infty \}$ weakly preserves the order and define a partition $\mathcal{P} = (I_0,\ldots,I_l)$ of $k$ which satisfies $I_0 \sqcup \cdots \sqcup I_l = [0, d]$ and any element of $I_{j-1}$ is $\preceq_k$-smaller than that of $I_j$ and $I_1,\ldots,I_l \neq \emptyset$ as before. Let $\rho_1 \in [0, l]$ denote the greatest integer such that $-\tau (k_{I_{\rho_1}}) = \mathrm{ord}_p k_{I_{\rho_1}} < 0$, and $\rho_2 \in [0, l]$ denote the greatest integer such that $\mathrm{ord}_p k_{I_{\rho_2}} < \infty$. For a Laurent weight without poles $k$, $\xi = V^{\mu (k)} \eta \in V^{\mu (k)} W(R_S^{\Box})$ and a partition $\mathcal{P}$ of $[0, d]$, we define an element $\epsilon' = \epsilon'(\xi, k, \mathcal{P}) \in W\Omega_{(R_S^{\Box}, N_S) / (S, M_S)}^l$ as follows. 

\begin{enumerate}

\item If $I_0 \neq \emptyset$, 
\begin{align*}
\epsilon' \coloneqq &V^{\mu(I_0)} (\eta [t]^{p^{\mu(I_0)} k_{I_0}}) (d V^{\mu(I_1)} [t]^{p^{\mu(I_1)} k_{I_1}}) \cdots (d V^{\mu(I_{\rho_1})} [t]^{p^{\mu(I_{\rho_1})} k_{I_1}}) \\ 
\cdot \ &(F^{-\tau(I_{\rho_1 +1})} d[t]^{p^{\tau(I_{\rho_1 +1})} k_{I_{\rho_1 +1}}}) \cdots (F^{-\tau(I_{\rho_2})} d[t]^{p^{\tau(I_{\rho_2})} k_{I_{\rho_2}}}) \\ 
\cdot \ &(d\log \prod_{i \in I_{\rho_2 + 1}} [t_i]) \cdots (d\log \prod_{i \in I_l} [t_i]).
\end{align*}
\item If $I_0 = \emptyset$ and $k$ is not integral, 
\begin{align*}
\epsilon' \coloneqq &(d V^{\mu(I_1)} (\eta [t]^{p^{\mu(I_1)} k_{I_1}})) (d V^{\mu(I_2)} [t]^{p^{\mu(I_2)} k_{I_1}}) \cdots (d V^{\mu(I_{\rho_1})} [t]^{p^{\mu(I_{\rho_1})} k_{I_1}}) \\ 
\cdot \ &(F^{-\tau(I_{\rho_1 +1})} d[t]^{p^{\tau(I_{\rho_1 +1})} k_{I_{\rho_1 +1}}}) \cdots (F^{-\tau(I_{\rho_2})} d[t]^{p^{\tau(I_{\rho_2})} k_{I_{\rho_2}}}) \\ 
\cdot \ &(d\log \prod_{i \in I_{\rho_2 + 1}} [t_i]) \cdots (d\log \prod_{i \in I_l} [t_i]).
\end{align*}
\item If $I_0 = \emptyset$ and $k$ is integral, 
\[
\epsilon' \coloneqq \eta (F^{-\tau(I_1)} d[t]^{p^{\tau(I_1)} k_{I_1}}) \cdots (F^{-\tau(I_{\rho_2})} d[t]^{p^{\tau(I_{\rho_2})} k_{I_{\rho_2}}}) \cdot (d\log \prod_{i \in I_{\rho_2 + 1}} [t_i]) \cdots (d\log \prod_{i \in I_l} [t_i]).
\]

\end{enumerate}

\begin{proposition}
\label{dRW-RBox-Basis}
Any element in $W_n \Omega_{(R_S^{\Box}, N_S) / (S, M_S)}^l$ has a unique expression as a finite sum of elements of the form 
\[
\sum_{(k, \mathcal{P})} \epsilon'(\xi_{k, \mathcal{P}}, k, \mathcal{P}),
\]
where $k$ runs over Laurent weights without poles and $\mathcal{P}$ over all partitions $(I_0,\ldots,I_l)$ of $k$. Therefore, the map of $W_n (S)$-modules 
\[
e \colon \bigoplus_{\substack{k \colon \{1,\ldots,d\} \to p^{-n} \mathbb{Z}, \\ k([0, r]) \subset p^{-n}\mathbb{Z}_{\geq 0}, \\ [0, r] \not\subset \Supp k}} \bigoplus_{\substack{\mathrm{partition} \\ \mathcal{P} = (I_0,\ldots,I_n) \\ \mathrm{of} \ [0, d]}} 
V^{\mu(k)} W_{n - \mu(k)} (S) \to W_n \Omega_{(R_S^{\Box}, N_S) / (S, M_S)}^l,
\]
given by the sum of the maps 
\[
V^{\mu(k)} W_{n - \mu(k)} (S) \to W_n \Omega_{(R_S^{\Box}, N_S) / (S, M_S)}^l, \ V^{\mu(k)} (\xi) \mapsto \epsilon' (\xi, k, \mathcal{P}),
\]
is an isomorphism. 
\end{proposition}

\begin{proof}
The $W_n (R_S^{\Box})$-module $W_n \Omega_{(R_S^{\Box}, N_S) / (S, M_S)}^l$ turns out to be the localization of the $W_n (R_S^{\Box})$-module $W_n \Omega_{(R_{S, +}^{\Box}, N_S) / (S, M_S)}^l$ at the non zero-divisors $[t_{r+1}],\ldots,[t_d]$ by Proposition \ref{et-bc-ldRW} and the fact that $W_n$ commutes with localizations. So we have 
\[
W_n \Omega_{(R_S^{\Box}, N_S) / (S, M_S)}^l = \bigcup_{j \in \mathbb{Z}_{\geq 0}} [t_{r+1} \cdots t_d]^{-j} W_n \Omega_{(R_{S, +}^{\Box}, N_S) / (S, M_S)}^l.
\]
Then elements $[t_{r+1} \cdots t_d]^{-j} \epsilon (\xi_{k, \mathcal{P}}, k, \mathcal{P})$ can be written as $W_{n - \mu(k)} (S)$-linear combinations of the elements $\epsilon (\xi_{k', \mathcal{P}'}, k', \mathcal{P}') \prod_{i \in J} d\log[t_i]$ where $k' \colon [0, d] \to \mathbb{Z}[1/p]$ is the function such that $k'(i) = k(i)$ if $i \in [0, r]$ and $k'(i) = k(i) - j$ if $i \in [r+1, d]$, and $J \subset k'^{-1} \{0\} \cap [r+1, d]$ by the proof of Proposition \ref{dRW-RBox+-Basis} (cf. \cite[Proposition 2.11, Lemma 2.12]{LZ03}). The part of $\prod_{i \in J} d\log[t_i]$ is transferred to the form $(d\log \prod_{i \in I_{\rho_2 + 1}} [t_i]) \cdots (d\log \prod_{i \in I_l} [t_i])$ by $\mathbb{Z}$-linear combinations by Lemma \ref{partition-basis}. Conversely, for each $k$ we take a sufficiently large integer $N$. Then $[t_{r+1} \cdots t_d]^N \epsilon' (\xi_{k, \mathcal{P}}, k, \mathcal{P})$ can be written as $W(S)$-linear combinations of the elements $\epsilon (\xi_{k', \mathcal{P}'}, k', \mathcal{P}')$ where $k' \colon [0, d] \to \mathbb{Z}_{\geq 0} [1/p]$ is the function such that $k'(i) = k(i)$ if $i \in [0, r]$ and $k'(i) = k(i) + j$ if $i \in [r+1, d]$ by the above argument, so the assertion follows.
\end{proof}

The following lemma is used in the above proof. 

\begin{lemma}
\label{partition-basis}
Let $A$ be a ring, $M$ be a free $A$-module of rank $m$ which has a basis $\{e_1,\ldots,e_m\}$ and $\mathcal{P} = (I_0, I_1,\ldots,I_r)$ be a partition of $[1, m]$. Then the free $A$-module $\bigwedge^r M$ has a basis of the elements of the form 
\[
(\sum_{i \in I_1} e_i) \wedge \cdots \wedge (\sum_{i \in I_r} e_i),
\]
which just corresponds to the partition $\mathcal{P}$. 
\end{lemma}

\begin{proof}
It is seen by the coincidence of the number of the generators $\{e_{i_1} \wedge \cdots \wedge e_{i_r} \}$ and $\{ (\sum_{i \in I_1} e_i) \wedge \cdots \wedge (\sum_{i \in I_r} e_i) \}$, and each element in either of the generators is generated by those in the other one. 
\end{proof}

For $n \in \mathbb{Z}$, we define 
\[
\mathrm{Fil}^i W_n \Omega^i \coloneqq
\begin{cases}
0 & (i \geq n), \\ 
\mathrm{ker}(R^{n-i} : W_n \Omega^\ast \to W_i \Omega^\ast) & (0 \leq i < n), \\ 
W_n \Omega^\ast & (i < 0),
\end{cases}
\]
and we also define
\begin{align*}
\mathrm{gr}^i W_n \Omega^l &\coloneqq \mathrm{Fil}^i W_n \Omega^l / \mathrm{Fil}^{i+1} W_n \Omega^l, \\
ZW_n \Omega^l &\coloneqq \mathrm{ker}(d \colon W_n \Omega^l \to W_n \Omega^{l+1}), \\
BW_n \Omega^l &\coloneqq dW_n \Omega^{l-1},
\end{align*}
in which we omit subscriptions of $\Omega$. As for the total log de Rham-Witt complex $W\Omega^\ast$, the above notions are similarly defined. 

\begin{corollary}
\label{FilWLambda}
For $i \geq 0$ and $1 \leq m \leq n$, 
\begin{align*}
\mathrm{Fil}^{n-m} W_n \Omega_{(R_S^{\Box}, N_S) / (S, M_S)}^l
&= V^{n-m} W_m \Omega_{(R_S^{\Box}, N_S) / (S, M_S)}^l + dV^{n-m} W_m \Omega_{(R_S^{\Box}, N_S) / (S, M_S)}^{l-1}.
\end{align*}
\end{corollary}

\begin{proof}
The relation $\supset$ is clear. The kernel of $R^m \colon W_n \Lambda^l \to W_m \Lambda^l$ is generated by the elements $\epsilon'(\xi, k, \mathcal{P} = (I_0, I_1,\ldots,I_l))$ with $\xi \in V^{n-m} W_m (R_S^{\Box})$ by Proposition \ref{dRW-RBox-Basis}. Hence, it suffices to show each $\epsilon'$ in the kernel is in the right hand side of the claim of Proposition \ref{dRW-RBox-Basis}. If $I_0 \neq \emptyset$ or $k$ is integral, 
\[
\epsilon'(\xi = V^{n-m} \eta, k, \mathcal{P}) = V^{n-m} \epsilon'(\eta, p^{n-m} k, \mathcal{P}) \in V^{n-m} W\Omega_{(R_S^{\Box}, N_S) / (S, M_S)}^l.
\]
If $I_0 = \emptyset$ and $k$ is not integral, 
\[
\epsilon'(\xi, k, \mathcal{P} = (\emptyset, I_1,\ldots,I_l)) = 
d\epsilon'(\xi, k, (I_1,\ldots,I_l)). 
\]
By the first case, $d\epsilon'(\xi, k, (I_1,\ldots,I_l))$ here lies in $dV^{n-m} W\Omega_{(R_S^{\Box}, N_S) / (S, M_S)}^{l-1}$. 
\end{proof}

\subsection{Cartier isomorphisms for log de Rham-Witt complexes of the special fiber}
\label{LdRWC-Cisom}

In the following, let $k$ be a perfect field in the setting of the previous argument of Lemma \ref{chart-ssalg}. 

We prepare for the proof of logarithmic Cartier isomorphisms for special fibers. Without any description, we equip the algebras of special fibers with the log structure induced by those of mixed characteristic. Recall that the $p$-adically \'etale morphisms $R'^{\Box} \coloneqq R_{\mathcal{O}'}^{\Box} \to R$ of log $p$-adic formal algebras over $\mathcal{O}_C$ have models $R'^{\Box} \to R'$ over a discrete valuation subring $\mathcal{O}' \subset \overline{W(k)}$. We put $k' \coloneqq \mathcal{O}' / \mathfrak{m}'$, where $\mathfrak{m}'$ is the maximal ideal of $\mathcal{O}'$. The morphism $(k', Q_{\mathcal{O}'}) \to (R_{k'}^{\Box}, Q'^{\Box} \coloneqq Q_{R'^{\Box}})$ of log rings are log smooth of Cartier type as a morphism of log schemes by \cite[2.13]{HK94}, so that we have the Cartier isomorphism \cite[Theorem 4.12]{Kat88}, \cite[2.12]{HK94}
\[
C^{-1} \colon \Omega_{(R_{k'}^{\Box}, N_{\mathcal{O}'})^{(p)} / (k', M_{\mathcal{O}'})}^n \cong H^n (\Omega_{(R_{k'}^{\Box}, N_{\mathcal{O}'}) / (k', M_{\mathcal{O}'})}^\ast),
\]
where $n \geq 0$ and $(R_{k'}^{\Box}, N_{\mathcal{O}'})^{(p)}$ is the base change of $(R_{k'}^{\Box}, N_{\mathcal{O}'})$ for the absolute Frobenius $\mathcal{F}_{(k', M_{\mathcal{O}'})}$ of $(k', M_{\mathcal{O}'})$, characterized by 
\[
C^{-1} \colon a\prod_{i=1}^{n} d\log(h^\ast (b_i)) \mapsto g^\ast (a) \prod_{i=1}^{n} d\log(b_i),
\]
where $a \in (R_{k'}^{\Box})^{(p)}$, $b_i \in Q'^{\Box}$, $h$ is the base change of $\mathcal{F}_{k'}$ on $(R_{k'}^{\Box})$, and $g$ is the relative Frobenius $\mathcal{F}_{R_{k'}^{\Box} / k'}$ as sections of affine schemes. Similarly, we have Cartier isomorphisms over an extension $\mathcal{O}/\mathcal{O}'$ of valuation rings since the base change for the morphism $Q_{\mathcal{O}'} \to Q_{\mathcal{O}}$ of log structures on $\mathcal{O}$ preserves the module of log differentials. 
In this section, we fix a valuation ring $\mathcal{O}/\mathcal{O}'$ with residue field $k_{\mathcal{O}}$. 
For simplicity, we put $W_n \Omega^l \coloneqq W_n \Omega_{(R_{k_{\mathcal{O}}}^{\Box}, N_{\mathcal{O}}) / (k_{\mathcal{O}}, M_{\mathcal{O}})}^l$. If $n = 1$, we put $\Omega^l \coloneqq W_1 \Omega^l$.
We can naturally identify $\Omega_{(R_{k'}^{\Box}, N_{\mathcal{O}'}) / (k', M_{\mathcal{O}'})}^n$ with its Frobenius twist $\Omega_{(R_{k'}^{\Box}, N_{\mathcal{O}'})^{(p)} / (k', M_{\mathcal{O}'})}^n$ via the natural maps
\[
(R_{k'}^{\Box}, N_{\mathcal{O}'}) \to 
(R_{k'}^{\Box}, N_{\mathcal{O}'})^{(p)},
\]
which is the base change of the absolute Frobenius of $(k', M_{\mathcal{O}'})$, in the following argument. 

For $n \geq 1$, we define submodules $B_n \Omega^l \subset Z_n \Omega^l$ of $\Omega^l$ as follows :
\begin{align*}
&B_0 \Omega^l = 0, \qquad Z_0 \Omega^l = \Omega^l, \\ 
&B_1 \Omega^l = \mathrm{im}(d \colon \Omega^{l-1} \to \Omega^l), \qquad Z_1 \Omega^l = \mathrm{ker}(d \colon \Omega^l \to \Omega^{l+1}), \\ 
&C^{-1} (B_n \Omega^l) = B_{n+1} \Omega^l / B_1 \Omega^l, \\ 
&C^{-1} (Z_n \Omega^l) = Z_{n+1} \Omega^l / B_1 \Omega^l.
\end{align*}
For $m \geq 1$, repeatedly composing $C^{-1}$, we have the composite of Cartier inverses
\[
C^{-m} \colon \Omega^l \to Z_n \Omega^l / B_n \Omega^l.
\]
The map is an isomorphism. 

\begin{lemma}[cf.\ {\cite[Proposition I.3.3]{Ill79}}]
\label{Lambda-C-1}
For $n \geq 1$ and $l \geq 0$, the morphism $F \colon W_{n+1} \Omega^l \to W_n \Omega^l$ induces 
\[
F \colon W_n \Omega^l \to W_n \Omega^l / dV^{n-1} \Omega^{l-1}.
\]
Moreover, if $n=1$, the map is the composition of the Cartier isomorphism 
\[
C^{-1} \colon \Omega^l \to H^l (\Omega^{\ast})
\]
and natural injection $H^l (\Omega^\ast) \to \Omega^l / d\Omega^{l-1}$. We call the composite Cartier inverse $C^{-1} \colon \Omega^l \to \Omega^l / d\Omega^{l-1}$.
\end{lemma}

\begin{proof}
By Corollary \ref{FilWLambda} and the formula $FdV = d$, the assertion follows.
\end{proof}

Similarly, for $m \geq 1$, we can see that the map $F^m : W_{m+1} \Omega^l \to \Omega^l$ induces the map $C^{-m} \colon \Omega^l \to Z_m \Omega^l / B_m \Omega^l$. We calculate three kernels of the morphisms $p^i \colon W_n \Omega^l \to W_n \Omega^l$ (Lemma \ref{kerpi}),  $d \colon W_n \Omega^l \to W_n \Omega^{l+1}$ (Lemma \ref{kerd}), and $F^i \colon W_{n+i} \Omega^l \to W_n \Omega^l$ (Lemma \ref{kerFi}) of the log de Rham-Witt complexes using the method of \cite[Proposition I.3.4, Proposition I.3.11, Proposition I.3.21]{Ill79}. 

\begin{lemma}
\label{kerpi}
For $i \geq 0$, $n \geq 1$ and $l \geq 0$, we have 
\[
\ker(p^i \colon W_n \Omega^l \to W_n \Omega^l) = \mathrm{Fil}^{n-i} W_n \Omega^l.
\]
\end{lemma}

\begin{proof}
We have $p(\mathrm{Fil}^{n-1} W_n \Omega^l) = 0$ by Corollary \ref{FilWLambda} and the induced morphism $\tilde{p} \colon W_{n-1} \Omega^l \to W_n \Omega^l$ is injective by Proposition \ref{dRW-RBox-Basis}. The inclusion $\mathrm{Fil}^{n-i} W_n \Omega^l \subset \ker(p^i \colon W_n \Omega^l \to W_n \Omega^l)$ is obvious. We show the opposite inclusion by induction on $i$. For $i = 0$, it is obvious. If $p^{i+1} x = 0$ for $x \in W_n \Omega^l$, we have $px \in \mathrm{Fil}^{n-i} W_n \Omega^l$ by the assumption of the induction. Then $x \in \mathrm{Fil}^{n-i-1}$ since the following diagram
\[\xymatrix{
W_n \Omega^l \ar[d]^R \ar[rd]^p \\ 
W_{n-1} \Omega^l \ar[d]^{R^{n-i}} \ar[r]^{\tilde{p}} & W_n \Omega^l \ar[d]^{R^{n-i}} \\ 
W_{i-1} \Omega^l \ar[r]^{\tilde{p}} & W_i \Omega^l,
}\]
is commutative and the maps $\tilde{p}$ in the diagram are injective.
\end{proof}

\begin{lemma}
\label{p-Fil1-quot}
For $n \geq 1$, the natural projection 
\[
W_{n+1} \Omega^{\ast} / pW_{n+1} \Omega^{\ast} \to 
W_n \Omega^{\ast} / pW_n \Omega^{\ast}
\]
is a quasi-isomorphism. Moreover, the natural projection 
\[
W_n \Omega^\ast / pW_n \Omega^\ast \to
W_n \Omega^\ast / \mathrm{Fil}^1 W_n \Omega^\ast
\]
is a quasi-isomorphism.
\end{lemma}

\begin{proof}
We have the morphism of exact sequences 
\[\xymatrix{
0 \ar[r] & \mathrm{gr}^n W\Omega^l \ar[r] \ar[d]^0 & W_{n+1} \Omega^l \ar[r] \ar[d]^p & W_n \Omega^l \ar[r] \ar[d]^p & 0 \\
0 \ar[r] & \mathrm{gr}^n W\Omega^l \ar[r] & W_{n+1} \Omega^l \ar[r] & W_n \Omega^l \ar[r] & 0.
}\]
By Lemma \ref{kerpi}, we have the long exact sequence 
\[
0 \to \mathrm{gr}^{n-1} W\Omega^l \to \mathrm{gr}^n W\Omega^l \to 
W_{n+1} \Omega^l / p \to W_{n+1} \Omega^l / p \to 0,
\]
where the second morphism is induced from the morphism $\tilde{p}$ in the proof of Lemma \ref{kerpi}. Thus, it suffices to show the complex 
\[
\mathrm{gr}^n W\Omega^\ast / \tilde{p} \mathrm{gr}^{n-1} W\Omega^\ast = \mathrm{gr}^n W\Omega^\ast / p \mathrm{gr}^n W\Omega^\ast = \mathrm{Fil}^n W_{n+1} \Omega^\ast / p\mathrm{Fil}^n W_{n+1} \Omega^\ast
\]
is acyclic. If $x \in \mathrm{Fil}^n W_{n+1} \Omega^l / p\mathrm{Fil}^n W_{n+1} \Omega^l$ such that $dx \in p W_{n+1} \Omega^{l+1}$, we show 
\[
x \in p \mathrm{Fil}^n W_{n+1} \Omega^l / (p \mathrm{Fil}^n W_{n+1} \Omega^l \cap d\mathrm{Fil}^n W_{n+1} \Omega^{l-1}).
\]
We can take a lift $\tilde{x}$ of $x$ on $\mathrm{Fil}^n W_{n+1} \Omega^l$ and write $\tilde{x} = V^n y$ where $y \in \Omega^l$. Then we have $dy = F^n d V^n y = 0$ and so there exists $z \in W_2 \Omega^l$ such that $F(z) = y$. Thus we have 
\[
\tilde{x} = V^n y = V^n F (z) = pV^{n-1} (z) \in \tilde{p} (\mathrm{Fil}^{n-1} W_n \Omega^l) = p(\mathrm{Fil}^n W_{n+1} \Omega^l)
\]
by $VF = p$ for the case of Witt rings of an $\mathbb{F}_p$-algebra. Therefore, the claim is verified. 
\end{proof}

\begin{lemma}
\label{kerd}
For $n \geq 1$, $i \geq 0$ and $l \geq 0$, we have
\[
\mathrm{ker}(F^i d \colon W_{n+i} \Omega^l \to W_n \Omega^{l+1}) = F^n W_{2n+i} \Omega^l.
\]
In particular, we have
\[
ZW_n \Omega^l = F^n W_{2n} \Omega^l.
\]
\end{lemma}

\begin{proof}
We shall show the inclusion $\supset$. Take an element $x \in W_{2n+i} \Omega^l$. Then we have
\[
dF^n (x) = F^n d V^n F^n (x) = F^n d (V^n (1) x) = F^n (V^n (1) dx) = p^n F^n dx.
\]
Hence, we have $F^i dF^n (x) = F^i p^n F^n dx = p^n F^{n+i} dx$. Since $F^{n+i} dx \in W_n \Omega^{l+1}$ and $W_n \Omega^{l+1}$ is killed by $p^n$, we have $F^i d F^n (x) = 0$ and the inclusion $\supset$ is proved. 

We shall show the opposite inclusion $\subset$. 
We take an element $x \in W_{n+i} \Omega^l$ satisfying $F^i dx = 0$.
We claim
$x \in F^r W_{n+i+r} \Omega^l$ 
for $0 \leq r \leq n$. We prove the claim by induction on $r$. The case $r = 0$ is obvious. We assume the claim is proved for $r < n$. By assumption, we may write $x = F^r y$ for some $y \in W_{n+i+r} \Omega^l$. Then we have $F^i dx = F^i d(F^r y) = p^r F^{i+r} dy = 0$. By Lemma \ref{kerpi}, we have $F^{i+r} dy \in \mathrm{Fil}^{n-r} W_n \Omega^{l+1} \subset \mathrm{Fil}^1 W_n \Omega^{l+1}$. Let $\overline{y}$ be the image of $y$ by $R^{n+i+r-1} \colon W_{n+i+r} \Omega^l \to \Omega^l$. Since $C^{-(i+r)} d\overline{y} = 0$, we have $d\overline{y} = 0$. By Lemma \ref{p-Fil1-quot}, the $l$-th cohomology class given by $y \ \mathrm{mod} \ pW_{n+i+r} \Omega^l$ of the complex $W_{n+i+r} \Omega^{\ast} / pW_{n+i+r} \Omega^{\ast}$ is generated by $F$ by the Cartier isomorphism. So we can write $y = Fa + pb + dc$ for some $a \in W_{n+i+r+1} \Omega^l$, $b \in W_{n+i+r} \Omega^l$ and $c \in W_{n+i+r} \Omega^{l-1}$. Since 
\[
y = Fa + pb + dc = F(a + Vb + dVc)
\in FW_{n+i+r+1}\Omega^l,
\]
we have $x = F^r y \in F^{r+1} W_{n+i+r+1} \Omega^l$, and the claim is proved for $r+1$. Hence the claim is proved for every $r$ with $0 \leq r \leq n$. From the case $r = n$ of the claim, we have $x \in F^n W_{2n+i} \Omega^l$.
\end{proof}

\begin{lemma}
\label{kerFi}
For $i \geq 0$, $n \geq 1$ and $l \geq 0$, we have 
\[
\mathrm{ker} (F^i \colon W_{n+i} \Omega^l \to W_n \Omega^l) = V^i W_n \Omega^l.
\]
\end{lemma}

\begin{proof}
In the case of $i = 0$ is trivial, so we assume $i \geq 1$. We have 
\[
\mathrm{ker} \ F^i \subset \mathrm{ker} \ p^i =  
\mathrm{Fil}^n W_{n+i} \Omega^l = 
V^n W_i \Omega^l + dV^n W_i \Omega^{l-1}
\]
by Corollary \ref{FilWLambda} and Lemma \ref{kerpi}. Then it suffices to show that we prove 
\[
(\mathrm{ker} \ F^i) \cap (V^n W_i \Omega^l + \mathrm{Fil}^{m} W_{n+i} \Omega^l) \subset 
V^n W_i \Omega^l + \mathrm{Fil}^{m+1} W_{n+i} \Omega^l
\]
for $n \leq m \leq n+i-1$. We take $x \in W_i \Omega^l$, $y \in W_{n+i-m} \Omega^{l-1}$ and $z = V^n x + dV^m y$ such that $F^i z = F^i dV^m y = 0$. If $i \leq m$, we have $dV^{m-i} y = 0$. It follows that $V^{m-i} y \in F^n W_{2n} \Omega^{l-1}$ by Lemma \ref{kerd}. So we have 
\[
dV^m y \in dV^i F^n W_{2n} \Omega^{l-1} \subset p^n W_{n+i} \Omega^l \subset V^n W_i \Omega^l
\]
since $dF^n = p^n F^n$. If $i > m$, we have $F^{i-m} d = 0$. Let $\overline{y}$ be the image of $y$ in $\Omega^{l-1}$. It follows $d\overline{y} = 0$ from $C^{-(i-m)} d\overline{y} = 0$ and the injectivity of $C^{-1}$. Then there is $a \in W_{n+i-m+1} \Omega^{l-1}$ and $b \in \mathrm{Fil}^1 W_{n+i-m} \Omega^{l-1}$ such that $y = Fa + b$ by Lemma \ref{p-Fil1-quot} and usual log Cartier isomorphism. Since $dV^m b \in \mathrm{Fil}^{m+1} W_{n+i} \Omega^{l-1}$ by the commutativity of $R$ and $V$, we have 
\[
dV^m y = VFdV^{m-1} y + dV^m b \in V^n W_i \Omega^l + \mathrm{Fil}^{m+1} W_{n+i} \Omega^l. 
\]
\end{proof}

We prove $p$-adic Cartier isomorphism of the base $k_{\mathcal{O}}$. 

\begin{proposition}
\label{k-logdRW}
For $n \geq 1$ and $l \geq 0$, There exists a unique isomorphism 
\[
C^{-n} \colon W_n \Omega^l \to H^l (W_n \Omega^\ast)
\]
such that the following diagram is commutative, 
\[\xymatrix{
W_{2n} \Omega^l \ar[r] \ar[d]^{F^n} & W_n \Omega^l \ar[d]^{C^{-n}} \\
ZW_n \Omega^l \ar[r] & H^l (W_n \Omega^\ast). 
}\]
Here the horizontal morphisms are the natural projections. 
\end{proposition}

\begin{proof}
The uniqueness follows from  
\[
F^n (\mathrm{Fil}^n W_{2n} \Omega^l) = F^n (V^n W_n \Omega^l + dV^n W_n \Omega^{l-1}) \subset BW_n \Omega^l
\]
by Corollary \ref{FilWLambda} and surjectivity follows from Lemma \ref{kerd}. For injectivity, set $x \in W_{2n} \Omega^l$ and $y \in W_n \Omega^{l-1}$ such that $F^n x = dy$. Then we have $F^n (x - dV^n y) = 0$. Hence $x - dV^n y \in V^n W_n \Omega^l$ by Lemma \ref{kerFi}. So we have $x \in \mathrm{Fil}^n W_{2n} \Omega^l$ and the injectivity follows. 
\end{proof}

\subsection{$p$-adic Cartier isomorphisms}
\label{LdRWC-padicCisom}

We use the notations in the section \ref{sec;AinfCoh}.

We shall define the algebras 
\begin{align*}
\mathcal{A} &\coloneqq A_{\inf} [X_0, \ldots, X_r, X_{r+1}^{\pm 1}, \ldots, X_d^{\pm 1}] /(X_0 \cdots X_r - [({p^{\flat}})^q]), \\ 
\mathcal{A}_{\infty} &\coloneqq \varinjlim_m A_{\inf} [X_0^{1/p^m}, \ldots, X_r^{1/p^m}, X_{r+1}^{\pm 1/p^m}, \ldots, X_d^{\pm 1/p^m}]/(\prod_{i=0}^r X_i^{1/p^m} - [({p^{\flat}})^{q/p^m}]).
\end{align*}
Recall that $\Delta'$ denote the free abelian subgroup of $\Delta$ which is generated by $\delta_i, \ 1 \leq i \leq d$. We define actions of $\Delta'$ on $\mathcal{A}$ (resp. $\mathcal{A}_{\infty}$) as those of $A(R)^{\Box}$ and $A_{\inf} (R_{\infty}^{\Box})$. We set $D \coloneqq R\Gamma (\Delta', \mathcal{A})$ and $D_{\infty} \coloneqq R\Gamma (\Delta', \mathcal{A}_{\infty})$. The $(p, \mu)$-adic completion of $\mathcal{A}, \mathcal{A}_{\infty}$ is respectively $A(R)^{\Box}, A_{\inf} (R_{\infty}^{\Box})$ by definition. For a $\mathbb{Z}_p$-algebra $S$ with a prelog structure $\alpha_S : M_S \to S$ and a fixed $\pi \in M_S$, we recall the algebra 
\[
R_S^{\Box} = R_{S, \pi}^{\Box} \coloneqq S [t_0, \ldots, t_r, t_{r+1}^{\pm 1}, \ldots, t_d^{\pm 1}] /(t_0 \cdots t_r - \alpha_S(\pi)).
\]
with the prelog structure $N_S \to R_S^{\Box}$. We mainly consider the case $S = \mathcal{O}'$ and $\pi = p^q$. 
We define the object 
\[
\mathcal{W}_n^i (D_{\infty}) \coloneqq 
H^i (L\eta_{\mu} D_{\infty} \otimes_{A_{\inf}, \tilde{\theta}_n}^L A_{\inf} / \tilde{\xi}_n)
\]
for $n \geq 1$ and $i \in \mathbb{Z}$. Here $L\eta$ is the d\'ecalage functor as defined in \cite[Section 6]{BMS18}. There is a decomposition 
\[
\eta_{\mu} D_{\infty} \cong 
\eta_{[\epsilon] - 1} K_{\mathcal{A}_{\infty}} (\delta_1 - 1,\ldots,\delta_d - 1) \cong 
\bigoplus_{\substack{k \colon \{0,\ldots,d\} \to \mathbb{Z} [1/p] \\ 
k([0, r]) \subset \mathbb{Z}_{\geq 0} [1/p] \\ 
[0, r] \not\subset \Supp k}}K_{A_{\inf}} (\delta_1 - 1,\ldots,\delta_d - 1)
\]
by \cite[Lemma 7.3]{BMS18} and \cite[Lemma 7.5]{BMS18}. For $a \in \mathbb{Z}\left[ \tfrac{1}{p} \right] \backslash \mathbb{Z}$, we have $[\epsilon]^a - 1$ divides $[\epsilon] - 1$. This implies that 
\[
\eta_{\mu} D \cong 
\eta_{\mu} K_{\mathcal{A}} (\delta_1 - 1,\ldots,\delta_d - 1) \to 
\eta_{\mu} K_{\mathcal{A}_{\infty}} (\delta_1 - 1,\ldots,\delta_d - 1) \cong 
\eta_{\mu} D_{\infty}
\]
is a quasi-isomorphism by \cite[Lemma 7.9]{BMS18}. Moreover, we have the following lemma.
\begin{lemma}
\label{WnD-Koszul}
There is a natural quasi-isomorphism
\[
[\epsilon] \textrm{-} \Omega_{\mathcal{A}_{\infty} / A_{\inf}}^{\ast} \to L\eta_{\mu} D_{\infty}
\]
and the natural map
\[
(L\eta_{\mu} D_{\infty}) \otimes_{A_{\inf}, \tilde{\theta}_n}^L A_{\inf} / \tilde{\xi}_n \to 
L\eta_{[\zeta_{p^n}] - 1} (D \otimes_{A_{\inf}, \tilde{\theta}_n}^L A_{\inf} / \tilde{\xi}_n)
\]
is a quasi-isomorphism.
\end{lemma}

\begin{proof}
We can prove the assertion as in Lemma \ref{ALambda-WnLambda} and \cite[Proposition 11.8]{BMS18}.
\end{proof}

\begin{lemma}[cf. {\cite[Lemma 11.9]{BMS18}}]
\label{WnD-decomp}
For $n \geq 0$, there is an isomorphism of $W_n (\mathcal{O})$-modules 
\[
\mathcal{W}_n^i (D_{\infty}) \cong 
\bigoplus_{\substack{k \colon \{0,\ldots,d\} \to p^{-n} \mathbb{Z} \\ k([0, r]) \subset \mathbb{Z}_{\geq 0} [1/p] \\ 
[0, r] \not\subset \Supp k}} (W_{n - \mu(k)} (\mathcal{O}))^{\binom{d}{i}},
\]
where $k$ runs over weights without poles which have the image $p^{-n} \mathbb{Z}$, and $\mu(k)$ is as in Definition \ref{dRW-Box+i}. Moreover, $\mathcal{W}_n^i (D_{\infty})$ is $p$-torsion-free.
\end{lemma}

\begin{proof}
By the above argument and \cite[Lemma 7.9]{BMS18}, we have
\begin{align*}
&L\eta_{\mu} D_{\infty} \\ 
\cong &\bigoplus_{\substack{k \colon \{0,\ldots,d\} \to \mathbb{Z} \\ k([0, r]) \subset \mathbb{Z}_{\geq 0} \\ [0, r] \not\subset \Supp k} } \bigg( \bigotimes_{j=1}^{r} (A_{\inf} (\mathcal{O}) \cdot X^a \xrightarrow{[k_j - k_0]_q}
A_{\inf} (\mathcal{O}) \cdot X^a) \\
& \qquad \qquad \qquad \qquad \otimes \bigotimes_{j = r+1}^{d} (A_{\inf} (\mathcal{O}) \cdot X^a \xrightarrow{[k_j]_q} A_{\inf} (\mathcal{O}) \cdot X^a) \bigg) \\ 
= &\bigoplus_{\substack{k \colon \{0,\ldots,d\} \to \mathbb{Z} \\ k([0, r]) \subset \mathbb{Z}_{\geq 0} \\ [0, r] \not\subset \Supp k} }K_{A_{\inf} (\mathcal{O})} ([k_1 - k_0]_q,\ldots,[k_r - k_0]_q, [k_{r+1}]_q,\ldots,[k_d]_q),
\end{align*}
where $[n]_q \coloneqq ([\epsilon]^n - 1)/([\epsilon] - 1)$ for $n \in \mathbb{Z}$. Through the base change along $\tilde{\theta}_n \colon A_{\inf} (\mathcal{O}) \to W_n (\mathcal{O})$, we have 
\begin{align*}
&L\eta_{\mu} D_{\infty} / \tilde{\xi}_n \\
\cong &\bigoplus_{\substack{k \colon \{0,\ldots,d\} \to \mathbb{Z} \\ k([0, r]) \subset \mathbb{Z}_{\geq 0} \\ [0, r] \not\subset \Supp k}} K_{W_n (\mathcal{O})} \left( \dfrac{[\zeta_{p^n}^{k_1 - k_0}] - 1}{[\zeta_{p^n}] - 1},\ldots,\dfrac{[\zeta_{p^n}^{k_r - k_0}] - 1}{[\zeta_{p^n}] - 1}, \dfrac{[\zeta_{p^n}^{k_{r+1}}] - 1}{[\zeta_{p^n}] - 1},\ldots,\dfrac{[\zeta_{p^n}^{k_d}] - 1}{[\zeta_{p^n}] - 1} \right).
\end{align*}
Each element in the bracket in the right hand side is divisible by $\tfrac{[\zeta_{p^n}^{\mu(p^{-n} k)}] - 1}{[\zeta_{p^n}] - 1}$, and at least one element is equal to this element up to a unit of $W_n (\mathcal{O})$. 
By \cite[Lemma 7.10 (ii)]{BMS18}, the $i$-th cohomology of the Koszul complex is isomorphic to
\[
\mathrm{Ann}_{W_n (\mathcal{O})} \left( \dfrac{[\zeta_{p^n}^{\mu(p^{-n} k)}] - 1}{[\zeta_{p^n}] - 1}\right)^{\binom{d-1}{i}} \oplus 
\left( W_n (\mathcal{O}) / \dfrac{[\zeta_{p^n}^{\mu(p^{-n} k)}] - 1}{[\zeta_{p^n}] - 1} W_n (\mathcal{O}) \right)^{\binom{d-1}{i-1}}.
\]
Moreover, the module is isomorphic to $W_{n - \mu(p^{-n} k)} (\mathcal{O})^{\binom{d}{i}}$ by \cite[Corollary 3.18]{BMS18}. The assertion follows by renaming $p^{-n} k$ by $k$.
\end{proof}

We set 
\[
\mathcal{W}_n^i (D_{\infty})_{\mathrm{pre}} \coloneqq H^i (D_{\infty} \otimes_{A_{\inf} (\mathcal{O}), \tilde{\theta}_n}^L W_n (\mathcal{O})).
\]
Following \cite[11.1.1]{BMS18}, we can construct the following natural maps of graded $W_n (\mathcal{O})$-modules
\begin{align*}
R & \colon \mathcal{W}_{n+1}^\ast (D_{\infty})_{\mathrm{pre}} \to \mathcal{W}_n^\ast (D_{\infty})_{\mathrm{pre}}, \\ 
F & \colon \mathcal{W}_{n+1}^\ast (D_{\infty})_{\mathrm{pre}} \to \mathcal{W}_n^\ast (D_{\infty})_{\mathrm{pre}}, \\ 
V & \colon \mathcal{W}_n^\ast (D_{\infty})_{\mathrm{pre}} \to \mathcal{W}_{n+1}^\ast (D_{\infty})_{\mathrm{pre}}, \\ 
d & \colon \mathcal{W}_n^i (D_{\infty})_{\mathrm{pre}} \to \mathcal{W}_n^{i+1} (D_{\infty})_{\mathrm{pre}}
\end{align*}
induced by $\tilde{\theta}_n (\xi)^n \varphi_D^{-1}$, the canonical projection, the multiplication by $\varphi^{n+1} (\xi)$, and the Bockstein differential respectively. By \cite[Proposition 11.2]{BMS18}, $\mathcal{W}_n^i (D_{\infty})_{\mathrm{pre}}$ satisfies some axioms of $F$-$V$-procomplexes. The module $\mathcal{W}_n^i (D_{\infty})$ turns out to be a submodule of $\mathcal{W}_{n}^i (D_{\infty})_{\mathrm{pre}}$ by \cite[Proposition 11.5 (ii)]{BMS18} and the $p$-torsion-freeness of $\mathcal{W}_n^i (D_{\infty})$ proved in Lemma \ref{WnD-decomp}. We shall equip $\mathcal{W}_n^{\ast} (D_{\infty})$ with the structure of a log $F$-$V$-procomplex as a submodule of $\mathcal{W}_n^{\ast} (D_{\infty})_{\mathrm{pre}}$. Thanks to the constructions of 
\cite[Proposition 11.5]{BMS18}, it suffices to show the following lemma in order to equip $\mathcal{W}_n^i (D_{\infty})$ the structure of a log $F$-$V$-procomplex.

\begin{lemma}
\label{lambdanast}
\begin{enumerate}
\item There is a unique collection of $W_n (\mathcal{O})$-algebra morphisms $\lambda_n \colon W_n (R_{\mathcal{O}}^{\Box}) \to \mathcal{W}_n^0 (D_{\infty})$ which satisfy $\lambda_n ([t_j]) = X_j$ for $j \in [0, d]$, $Fd\lambda_{n+1} ([b]) = \lambda_n ([b])^{p-1} d\lambda_n ([b])$ for $b \in W_n (R_{\mathcal{O}}^{\Box})$ and which commute with the maps $F$, $V$ and $R$. 
\item There is a collection of natural morphisms $\partial_n \colon Q_{R_{\mathcal{O}}^{\Box}} \to \mathcal{W}_n^1 (D_{\infty})$ of monoids which satisfy $F\partial_{n+1} = \partial_n$ and which commutes with $R$. 
\end{enumerate}
In particular, 
there are natural maps of log $F$-$V$-procomplexes 
\[
\lambda_n^\ast \colon W_n \Omega_{(R_{\mathcal{O}}^{\Box}, Q_{R_{\mathcal{O}}^{\Box}})/(\mathcal{O}, Q_{\mathcal{O}})}^\ast \to \mathcal{W}_n^\ast (D_{\infty}).
\]
\end{lemma}

\begin{proof}
For the first claim, we have 
\[
\mathcal{W}_n^0 (D_{\infty}) = H^0 ((L\eta_{\mu} D_{\infty})/\tilde{\xi}_n) = H^0 (L\eta_{[\zeta_{p^r}] - 1} (D_{\infty} / \tilde{\xi}_n)) = H^0 (D_{\infty}/\tilde{\xi}_n) = (\mathcal{A}_{\infty} / \tilde{\xi}_n)^\Delta,
\]
where the second equation by Lemma \ref{WnD-Koszul}, which can be similarly proved in the case of derived pullbacks are naturally quasi-isomorphic to the usual pullback, and the third equation by \cite[Lemma 6.4]{BMS18} since $H^0 (D_{\infty}/\tilde{\xi}_n) = (\mathcal{A}_{\infty} / \tilde{\xi}_n)^\Delta$ is $p$-torsion free. Then we have the natural identity $(\mathcal{A}_{\infty} / \tilde{\xi}_n)^\Delta = W_n (R_{\mathcal{O}}^{\Box})$ by Lemma \ref{convsum-W} of the Laurent case $R_S^{\Box}$ and the same argument as \cite[Lemma 11.11]{BMS18}. It follows that commutativity satisfies by the constructions of $F, V, R$.  

For the second claim, we recall that the differentials on $\mathcal{W}_n^{\ast} (D_{\infty})$ are defined by the Bockstein differential of the following distinguished triangle : 
\[
L\eta_{\mu} D_{\infty} / \tilde{\xi}_n \xrightarrow{\xi_n} 
L\eta_{\mu} D_{\infty} / \tilde{\xi}_n^2 \to 
L\eta_{\mu} D_{\infty} / \tilde{\xi}_n \to. 
\]
We shall calculate the image of a differential with Koszul complexes
\begin{align*}
&L\eta_{\mu} D_{\infty} / \tilde{\xi}_n \\
\cong &\bigoplus_{\substack{k \colon \{0,\ldots,d\} \to p^{-n} \mathbb{Z} \\ k([0, r]) \subset \mathbb{Z}_{\geq 0} [1/p] \\ [0, r] \not\subset \Supp k}} K_{W_n (\mathcal{O})} \left( \dfrac{[\zeta_{p^n}^{k_1 - k_0}] - 1}{[\zeta_{p^n}] - 1},\ldots,\dfrac{[\zeta_{p^n}^{k_r - k_0}] - 1}{[\zeta_{p^n}] - 1}, \dfrac{[\zeta_{p^n}^{k_{r+1}}] - 1}{[\zeta_{p^n}] - 1},\ldots,\dfrac{[\zeta_{p^n}^{k_d}] - 1}{[\zeta_{p^n}] - 1} \right).
\end{align*}
and let $X^k$ be the monomial of the direct summand of $H^0 (L\eta_{\mu} D_{\infty})$, which is corresponded to an integral Laurent weight without poles $k$. The image of $X^k$ by the differential $H^0 (L\eta_{\mu} D_{\infty}) \to H^1 (L\eta_{\mu} D_{\infty})$ is also in the direct summand corresponded to $k$. The decomposition in Lemma \ref{WnD-decomp} assures that the multiplication by the monomial $X^k$ on $\mathcal{W}_n^i (D_{\infty})$ is injective. So we define a morphism of monoids $\partial_n \colon N_{\mathcal{O}} \to \mathcal{W}_n^1 (D_{\infty})$ as $x \mapsto d[x]/[x]$, and it extends to the morphism $Q_{R_{\mathcal{O}}^{\Box}} \to \mathcal{W}_n^1 (D_{\infty})$ by the equation $N_{\mathcal{O}} \oplus \mathcal{O}^{\ast} = Q_{R_{\mathcal{O}}^{\Box}}$
The map is well-defined and satisfies $F\partial_{n+1} = \partial_n$. We can see the equation 
\[
\partial_1 (X_i) = \dfrac{1}{\zeta_p - 1} ((0,\ldots,0,d\log \zeta_{p^n},0,\ldots,0))_n,
\]
where the $i$-th element is nonzero, via the identification $\xi/\xi^2 = T_p (\Omega_{\mathcal{O}/\mathbb{Z}_p}^1)$ by the construction of $\partial_1$. The last claim follows by the argument of \cite[Proposition 11.5]{BMS18}.
\end{proof}
We prove the $k_{\mathcal{O}}$-analog of the main theorem. For this purpose, we prove the following lemma along the argument of \cite[Proposition 10.14]{BMS18}. 
\begin{lemma}
\label{WLambda-perfdbasechange}
For a homomorphism $h \colon (A, M_A) \to (A', M_{A'})$ of perfectoid rings with prelog structures, where $h$ induces an isomorphism $M_A \cong M_{A'}$, the natural map
\[
W_n \Omega_{(R_A^{\Box}, N_A) / (A, M_A)}^i \otimes_{W_n (A)}^L W_n (A') \to 
W_n \Omega_{(R_{A'}^{\Box}, N_{A'}) / (A', M_{A'})}^i
\]
is a quasi-isomorphism for every $i \in \mathbb{N}$. 
\end{lemma}

\begin{proof}
We see that $W_n \Omega_{(R_A^{\Box}, N_A) / (A, M_A)}^i$ (resp.\ $W_n \Omega_{(R_{A'}^{\Box}, N_{A'}) / (A', M_{A'})}^i$) is naturally isomorphic to the direct sum of $V^{\mu(k)} W_{n - \mu(k)} (A)$ (resp.\ $V^{\mu(k)} W_{n - \mu(k)} (A')$) as a $W_n (A)$ (resp.\ $W_n (A')$)-module by Proposition \ref{dRW-RBox-Basis}. By \cite[Remark 3.19]{BMS18}, we have natural isomorphisms $V^{\mu(k)} W_{n - \mu(k)} (A) \cong W_{n - \mu(k)} (A)$ (resp.\ $V^{\mu(k)} W_{n - \mu(k)} (A') \cong W_{n - \mu(k)} (A')$) of $W_n (A)$ (resp.\ $W_n (A')$)-modules. So the natural map
\[
W_n \Omega_{(R_A^{\Box}, N_A) / (A, M_A)}^i \otimes_{W_n (A)}^L W_n (A') \to 
W_n \Omega_{(R_{A'}^{\Box}, N_{A'}) / (A', M_{A'})}^i
\]
is a quasi-isomorphism by \cite[Lemma 3.13]{BMS18}. 
\end{proof}

\begin{proposition}[cf. {\cite[Lemma 11.16]{BMS18}}]
\label{k-logdRWC}
There is an isomorphism of differential graded algebras 
\[
\alpha \colon \mathcal{W}_n^\ast (D_{\infty})_{k_{\mathcal{O}}} \coloneqq 
\mathcal{W}_n^\ast (D_{\infty}) \otimes_{W_n (\mathcal{O})} W_n (k) \cong 
W_n \Omega_{(R_{k_{\mathcal{O}}}^{\Box}, N_{\mathcal{O}})/(k_{\mathcal{O}}, Q_{\mathcal{O}})}^\ast,
\]
where $k_{\mathcal{O}}^{\Box}, R_{k_{\mathcal{O}}}^{\Box}$ are respectively equipped with the log structures associated to $\mathbb{N} \to k_{\mathcal{O}}, \ 1 \mapsto 0$ and $\mathbb{N}^{r+1} \sqcup_{\mathbb{N}} \mathbb{N} \to R_{k_{\mathcal{O}}}, \ (e_i, 0) \mapsto t_i, \ (0, 1) \mapsto 0$ for a fixed base $\{e_i\}$ of $\mathbb{N}^{r+1}$. 
\end{proposition}

\begin{proof}
By Lemma \ref{WnD-decomp} and \cite[Lemma 3.13]{BMS18}, the cohomology group $\mathcal{W}_n^i (D_{\infty})$ and $W_n (k)$ are Tor-independent over $W_n (\mathcal{O})$. So by \cite[Tag 061Z]{SP}, we have 
\[
\mathcal{W}_n^i (D_{\infty})_{k_{\mathcal{O}}} = H^i (L\eta_{\mu} D_{\infty} \otimes_{A_{\inf} (\mathcal{O})}^L W_n (k_{\mathcal{O}})).
\]
We may identify $L\eta_{\mu} D_{\infty} \otimes_{A_{\inf} (\mathcal{O})}^L W_n (k_{\mathcal{O}})$ with the log de Rham complex $\Omega_{(R_{k_{\mathcal{O}}}^{\Box}, N_{\mathcal{O}})/(k_{\mathcal{O}}, Q_{\mathcal{O}})}^{\ast}$ by the presentation of $q$-de Rham complex as in the proof of Lemma \ref{lambdanast} and $A_{\inf} (\mathcal{O}) \ni [\epsilon] \mapsto 1 \in W_n (k_{\mathcal{O}})$. We also identify $\Omega_{(R_{W_n (k_{\mathcal{O}})}^{\Box}, N_{\mathcal{O}}) / (W_n (k_{\mathcal{O}}), Q_{\mathcal{O}})}^{\ast}$ with the log de Rham-Witt complex $W_n \Omega_{(R_{k_{\mathcal{O}}}^{\Box}, N_{\mathcal{O}})/(k_{\mathcal{O}}, Q_{\mathcal{O}})}^{\ast}$ by \cite[Theorem 7.9]{Mat15}, and $F^n$ induces the Cartier isomorphism 
\[
C^{-n} \colon W_n \Omega_{(R_{k_{\mathcal{O}}}^{\Box}, N_{\mathcal{O}})/(k_{\mathcal{O}}, Q_{\mathcal{O}})}^{\ast} \to H^i (W_n \Omega_{(R_{k_{\mathcal{O}}}^{\Box}, N_{\mathcal{O}})/(k_{\mathcal{O}}, Q_{\mathcal{O}})}^\ast)
\]
by Lemma \ref{k-logdRW}. let $C^n$ denote the inverse of $C^{-n}$. For the natural isomorphisms 
\begin{multline*}
\mathcal{W}_n^i (D_{\infty})_{k_{\mathcal{O}}} \to 
H^i (L\eta_{\mu} D_{\infty} \otimes_{A_{\inf} (\mathcal{O})}^L W_n (k_{\mathcal{O}})) \to 
H^i (\Omega_{(R_{W_n (k_{\mathcal{O}})}^{\Box}, N_{\mathcal{O}})/(W_n (k_{\mathcal{O}}), Q_{\mathcal{O}})}^{\ast}) \\
\to 
W_n \Omega_{(R_{k_{\mathcal{O}}}^{\Box}, N_{\mathcal{O}})/(k_{\mathcal{O}}, Q_{\mathcal{O}})}^i
\end{multline*}
the structures of multiplication are compatible by the fact that $C^n$ is a morphism of graded rings and \cite[Lemma 7.5]{BMS18}. 

It remains to show that $C^n$ preserves both of the differential structures. The differential on $H^i (\Omega_{(R_{W_n (k_{\mathcal{O}})}^{\Box}, Q_{\mathcal{O}})/(W_n (k_{\mathcal{O}}), M_{\mathcal{O}})}^{\ast})$ is induced from the Bockstein differential in the triangle 
\[
L\eta_{\mu} D_{\infty} \otimes_{A_{\inf} (\mathcal{O})}^L W_n (k_{\mathcal{O}}) \xrightarrow{p^n} 
L\eta_{\mu} D_{\infty} \otimes_{A_{\inf} (\mathcal{O})}^L W_{2n} (k_{\mathcal{O}}) \to 
L\eta_{\mu} D_{\infty} \otimes_{A_{\inf} (\mathcal{O})}^L W_n (k_{\mathcal{O}}).
\]
Then we have 
\[
dF^n ([t_j]) = p^{-n} d([t_j]^{p^n}) = [t_j]^{p^n - 1} d([t_j]) = F^n d([t_j])
\]
for $j \in [0, d]$ and obviously $dF^n ([a]) = 0$ for $a \in k$. Thus $dF^n (x) = F^n d(x)$ for any $x \in W_n (k_{\mathcal{O}})$ by the Leibniz rule. 
\end{proof}

\begin{proposition}
\label{lambdani-isom}
For $n \geq 1$ and $i \geq 0$, the map 
\[
\lambda_n^i \colon W_n \Omega_{(R_{\mathcal{O}}^{\Box}, N_{\mathcal{O}}) / (\mathcal{O}, Q_{\mathcal{O}})}^i \to \mathcal{W}_n^i (D_{\infty})
\]
in Lemma \ref{lambdanast} is an isomorphism. 
\end{proposition}

\begin{proof}
We may easily see that $W_n {\Omega}_{(R_{\mathcal{O}}^{\Box}, N_{\mathcal{O}}) / (\mathcal{O}, Q_{\mathcal{O}})}^i$ is also a log $F$-$V$-complex for $(R, Q_R)/(\mathcal{O}_C, Q_{\mathcal{O}_C})$. We have a natural isomorphism $W_n {\Omega}_{(R_{\mathcal{O}}^{\Box}, N_{\mathcal{O}}) / (\mathcal{O}, Q_{\mathcal{O}})}^i \cong W_n \Omega_{(R_{\mathcal{O}}^{\Box}, Q_{R_{\mathcal{O}}}) / (\mathcal{O}, Q_{\mathcal{O}})}^i$ of differential graded algebras by universality. With the similar argument as the proof of \cite[Theorem 11.13, Lemma 11.14]{BMS18}, $\mathcal{W}_n^\ast (D_{\infty})$ and $W_n {\Omega'}_{(R_{\mathcal{O}}^{\Box}, N_{\mathcal{O}}) / (\mathcal{O}, Q_{\mathcal{O}})}^{\Box}$ are equipped with natural actions of $\Delta'[1/p]$ as the restrictions of those of $\Delta$, and we have a natural direct sum decompositions 
\[
\mathcal{W}_n^i (D_{\infty}) = \bigoplus_{\substack{k \colon \{0,\ldots,d\} \to \mathbb{Z}[1/p] \\ k([0, r]) \subset \mathbb{Z}_{\geq 0} [1/p] \\ [0, r] \not\subset \Supp k}} M_k, \qquad W_n {\Omega'}_{(R_{\mathcal{O}}^{\Box}, N_{\mathcal{O}}) / (\mathcal{O}, Q_{\mathcal{O}})}^i  = \bigoplus_{\substack{k \colon \{0,\ldots,d\} \to \mathbb{Z}[1/p] \\ k([0, r]) \subset \mathbb{Z}_{\geq 0} [1/p] \\ [0, r] \not\subset \Supp k}} N_k
\] 
such that 
\begin{enumerate}
\item The decomposition is compatible with the action of $\Delta'[1/p]$. 
\item Each $M_k$ is isomorphic to a finite direct sum of copies of $W_{n - \mu(k)} (\mathcal{O})$. 
\item The decompositions are compatible with $\lambda_n^i$. 
\end{enumerate}
Here $k$ runs over weights without poles which have image $p^{-n} \mathbb{Z}$ by Corollary \ref{dRW-RBox-Basis} and Lemma \ref{WnD-decomp}. We can see that $\lambda_n^i$ is an isomorphism if and only if 
\begin{align*}
\overline{\lambda}_n^i \colon W_n \Omega_{(R_{k_{\mathcal{O}}}, N_{\mathcal{O}}) / (k_{\mathcal{O}}, Q_{\mathcal{O}})}^\ast \cong 
&W_n \Omega_{(R_{\mathcal{O}}, N_{\mathcal{O}}) / (\mathcal{O}, Q_{\mathcal{O}})}^\ast \otimes_{W_n (\mathcal{O})} W_n (k_{\mathcal{O}}) \\ 
\xrightarrow{\lambda_n^i \otimes_{W_n (\mathcal{O})} W_n (k_{\mathcal{O}})} 
&\mathcal{W}_n^i (D_{\infty}) \otimes_{W_n (\mathcal{O})} W_n (k_{\mathcal{O}}) \eqqcolon \mathcal{W}_n^i (D_{\infty})_{k_{\mathcal{O}}},
\end{align*}
where the first isomorphism is from Lemma \ref{WLambda-perfdbasechange}, is an isomorphism by Nakayama's lemma. Then the map $\alpha \overline{\lambda}_n^i$ where $\alpha$ is as in Proposition \ref{k-logdRWC} is the identity in degree $0$. It follows $\overline{\lambda}_n^i$ is an isomorphism since $W_n \Omega_{(R_{k_{\mathcal{O}}}, N_{\mathcal{O}}) / (k_{\mathcal{O}}, Q_{\mathcal{O}})}^\ast$ is generated in degree $0$ as a log differential graded $R_{k_{\mathcal{O}}}^{\Box} / k_{\mathcal{O}}$-algebra. 
\end{proof}

We prove the main theorem. We assume $K = C$ in the following argument. Recall that $R$ is equipped with the log structure $Q_R$. 

\begin{theorem}
\label{O-logdRWC}
Assume a $p$-adically complete algebra $R$ over $\mathcal{O}_C$ admits an $p$-adically \'etale coordinate map over $R^{\Box}$. For any $n \geq 1$ and $i \in \mathbb{Z}$, there is a natural isomorphism of $W_n (R)$-modules 
\[
H^i (\widetilde{W_n \Omega}_R) \cong 
\varprojlim_m W_n \Omega_{(R/p^m, Q_R)/(\mathcal{O}_C/p^m, Q_{\mathcal{O}_C})}^i \{ -i\} \eqqcolon 
W_n \Omega_{(R, Q_R)/(\mathcal{O}_C, Q_{\mathcal{O}_C})}^{i, \cont} \{ -i\}.
\]
Moreover, for $n = 1$, the map is compatible with the isomorphism in Theorem \ref{CKisom}. 
\end{theorem}

\begin{proof}
We shall locally construct the map. So we can assume $N_{\mathcal{O}_C} \to \Gamma(\Spf(R), Q_{\Spf(R)})$ is a small chart by \cite[Lemma A.9]{Kos20}. We prove the theorem along the proof of \cite[Theorem 11.1]{BMS18}. 

We claim that the cohomology module
\[
H^i (\widetilde{W_n \Omega}_R) \cong H^i (L\eta_{[\zeta_{p^n}] - 1} (R\Gamma_{\proet} (\mathfrak{X}_C^{\ad}, \mathbb{A}_{\inf, \mathfrak{X}_K^{\ad}}) / \tilde{\xi}_n))
\]
is $p$-torsion free. We show the claim along the argument of \cite[Lemma 9.7 (i)]{BMS18} and \cite[Proposition 3.19]{CK19}. 
We firstly consider the case of $H^i (L\eta_{[\zeta_{p^n}] - 1} R\Gamma_{\cont} (\Delta, A_{\inf} (R_{\infty}^{\Box})))$. We have the decomposition
\[
A_{\inf} (R_{\infty}^{\Box}) \cong 
{\widehat{\bigoplus}}_{\substack{k \colon \{0,\ldots,d\} \to \mathbb{Z}[1/p] \\ k([0, r]) \subset \mathbb{Z}_{\geq 0} [1/p] \\ [0, r] \not\subset \Supp k}} A_{\inf} \cdot X_0^{k_0} \cdots X_d^{k_d}.
\]
Reducing the map by $\tilde{\theta}_n : A_{\inf} (\mathcal{O}_C) \to W_n (\mathcal{O}_C)$, we have a similar decomposition of $W_n (R_{\infty}^{\Box})$. Then  $L\eta_{[\zeta_{p^n}] - 1} R\Gamma_{\cont} (\Delta, W_n (\mathcal{O}_C) \cdot X_0^{k_0} \cdots X_d^{k_d})$ is acyclic if $p^n k_i \not\in \mathbb{Z}$ for some $i \in [0, d]$. Otherwise it is calculated as the following Koszul complex
\[
K_{W_n (\mathcal{O})} \left( \dfrac{[\zeta_{p^n}^{k_1 - k_0}] - 1}{[\zeta_{p^n}] - 1},\ldots,\dfrac{[\zeta_{p^n}^{k_r - k_0}] - 1}{[\zeta_{p^n}] - 1}, \dfrac{[\zeta_{p^n}^{k_{r+1}}] - 1}{[\zeta_{p^n}] - 1},\ldots,\dfrac{[\zeta_{p^n}^{k_d}] - 1}{[\zeta_{p^n}] - 1} \right).
\]
The complex is $p$-torsion free by the proof of Lemma \ref{WnD-decomp}. Consequently, by \cite[Lemma 3.6]{CK19}, the module $H^i (L\eta_{[\zeta_{p^n}] - 1} R\Gamma_{\cont} (\Delta, A_{\inf} (R_{\infty}^{\Box})))$ turns out to be $p$-torsion free.

For the case of $R$, 
by \cite[Tag 061Z]{SP} and the claim for $R^{\Box}$, we have the following natural isomorphisms
\begin{align}
\label{O-logdRWC-HiLeta}
&H^i (L\eta_{\mu} R\Gamma_{\cont} (\Delta, A_{\inf} (R_{\infty}^{\Box}) / \tilde{\xi}_n)) \otimes_{A_{\inf} (\mathcal{O}_C)} A_{\inf} (\mathcal{O}_C) / p^m \notag \\ 
\cong &H^i (L\eta_{\mu} R\Gamma_{\cont} (\Delta, A_{\inf} (R_{\infty}^{\Box}) / \tilde{\xi}_n) \otimes_{A_{\inf} (\mathcal{O}_C)}^L A_{\inf} (\mathcal{O}_C) / p^m) \notag \\ 
\cong &H^i (L\eta_{\mu} R\Gamma_{\cont} (\Delta, A_{\inf} (R_{\infty}^{\Box}) / \tilde{\xi}_n) \otimes_{\mathbb{Z}}^L \mathbb{Z} / p^m)
\end{align}
for each $i \in \mathbb{Z}$ and $m \in \mathbb{N}$.
Each cohomology module of $L\eta_{\mu} R\Gamma_{\cont} (\Delta, A_{\inf} (R_{\infty}^{\Box}) / \tilde{\xi}_n))$ is $A_{\inf} (R_{\infty}^{\Box}) / \tilde{\xi}_n$-module. 
The map $A(R^{\Box}) / \tilde{\xi}_n \to A(R) / \tilde{\xi}_n$ is $p$-adically \'etale, so we can take an \'etale $A'$ over $A(R^{\Box}) / \tilde{\xi}_n$.
We have the following natural isomorphisms
\begin{align}
\label{O-logdRWC-HiLeta-basechange}
&H^i (L\eta_{\mu} R\Gamma_{\cont} (\Delta, A_{\inf} (R_{\infty}^{\Box}) / \tilde{\xi}_n) \otimes_{\mathbb{Z}}^L \mathbb{Z} / p^m) \otimes_{A(R^{\Box})} A(R) \notag \\ 
\cong &H^i (L\eta_{\mu} R\Gamma_{\cont} (\Delta, A_{\inf} (R_{\infty}^{\Box}) / \tilde{\xi}_n) \otimes_{\mathbb{Z}}^L \mathbb{Z} / p^m) \otimes_{A(R^{\Box}) / \tilde{\xi}_n} A(R) / \tilde{\xi}_n \notag \\ 
\cong &H^i ((L\eta_{\mu} R\Gamma_{\cont} (\Delta, A_{\inf} (R_{\infty}^{\Box}) / \tilde{\xi}_n) \otimes_{A(R^{\Box}) / \tilde{\xi}_n} A(R) / \tilde{\xi}_n) \otimes_{\mathbb{Z}}^L \mathbb{Z} / p^m) \notag \\ 
\cong &H^i ((L\eta_{\mu} R\Gamma_{\cont} (\Delta, A_{\inf} (R_{\infty}^{\Box}) / \tilde{\xi}_n) \otimes_{A(R^{\Box}) / \tilde{\xi}_n} A') \otimes_{\mathbb{Z}}^L \mathbb{Z} / p^m) \notag \\ 
\cong &H^i (q\textrm{-}\Omega_{A' / W_n (\mathcal{O}_C)} \otimes_{\mathbb{Z}}^L \mathbb{Z} / p^m) \notag \\
\cong &H^i (q\textrm{-}\Omega_{A(R) / W_n (\mathcal{O}_C)} \otimes_{\mathbb{Z}}^L \mathbb{Z} / p^m) \notag \\
\cong &H^i (L\eta_{\mu} R\Gamma_{\cont} (\Delta, A_{\inf} (R_{\infty}) / \tilde{\xi}_n) \otimes_{\mathbb{Z}}^L \mathbb{Z} / p^m).
\end{align}
The second isomorphism is given by the fact that $A_{\inf} (R^{\Box}) / \tilde{\xi}_n \to A_{\inf} (R) / \tilde{\xi}_n$ is $p$-adically \'etale.
The fifth isomorphism follows from the isomorphism
\[
A' \otimes_{\mathbb{Z}}^L \mathbb{Z} / p^m \cong A(R) \otimes_{\mathbb{Z}}^L \mathbb{Z} / p^m.
\]
The fourth and last isomorphisms follow from descriptions of the complexes as $q$-de Rham complexes
as in Lemma \ref{ALambda-WnLambda}.
By (\ref{O-logdRWC-HiLeta}) and (\ref{O-logdRWC-HiLeta-basechange}), we have the following exact sequence
\begin{multline}
\label{O-logdRWC-es-Ninfty}
0 \to H^i (L\eta_{\mu} R\Gamma_{\cont} (\Delta, A_{\inf} (R_{\infty}) / \tilde{\xi}_n) \otimes_{\mathbb{Z}}^L \mathbb{Z} / p^m)[p] \\ 
\to H^i (L\eta_{\mu} R\Gamma_{\cont} (\Delta, A_{\inf} (R_{\infty}) / \tilde{\xi}_n) \otimes_{\mathbb{Z}}^L \mathbb{Z} / p^m) \\
\to H^i (L\eta_{\mu} R\Gamma_{\cont} (\Delta, A_{\inf} (R_{\infty}) / \tilde{\xi}_n) \otimes_{\mathbb{Z}}^L \mathbb{Z} / p^{m-1})
\to 0.
\end{multline}
Since $L\eta_{\mu} R\Gamma_{\cont} (\Delta, A_{\inf} (R_{\infty}) / \tilde{\xi}_n)$ is derived $p$-adically complete, we also have the natural isomorphism
\begin{equation}
\label{O-logdRWC-Hiisom}
H^i (L\eta_{\mu} R\Gamma_{\cont} (\Delta, A_{\inf} (R_{\infty}) / \tilde{\xi}_n)) \cong 
\varprojlim_m (H^i (L\eta_{\mu} R\Gamma_{\cont} (\Delta, A_{\inf} (R_{\infty}) / (\tilde{\xi}_n, p^m)) \otimes_{\mathbb{Z}}^L \mathbb{Z}/p^m))
\end{equation}
by (\ref{O-logdRWC-es-Ninfty}) and \cite[Tag 0D6K]{SP}. 
Then 
\[
H^i (L\eta_{\mu} R\Gamma_{\cont} (\Delta, A_{\inf} (R_{\infty}) / \tilde{\xi}_n)) \cong H^i (L\eta_{[\zeta_{p^n}] - 1} R\Gamma_{\cont} (\Delta, W_n (R_{\infty})))
\]
turns out to be $p$-torsion free by taking limit of (\ref{O-logdRWC-es-Ninfty}) for $m$. 

Next, we shall give a morphism of monoids $\partial : Q_R \to H^1 (\widetilde{W_n \Omega}_R)$, which gives a structure of log differential graded algebra to $H^{\ast} (\widetilde{W_n \Omega}_R)$. Taking a $p$-adically \'etale morphism $R^{\Box} \to R$, it suffices to construct a morphism $Q_{R^{\Box}} \to H^1 (\widetilde{W_n \Omega}_R)$ by Lemma \ref{chart-ssalg} and Lemma \ref{chart-ssfalg}. The morphism can be constructed as in Lemma \ref{lambdanast} (2). 
We also have 
\[
H^0 (\widetilde{W_n \Omega}_R) \cong H_{\proet}^0 (\mathfrak{X}_C^{\ad}, W_n (\widehat{\mathcal{O}}_X^+)) = W_n (R)
\]
for $X = \Spa(R[1/p], R)$. Then we may naturally equip $H^\ast (\widetilde{W_n \Omega}_R)$ with the structure of a log $F$-$V$-procomplex as in Lemma \ref{lambdanast}. 

Then we shall show that the target of the universal morphism $W_n \Omega_{(R, Q_R)/(\mathcal{O}_C, Q_{\mathcal{O}_C})}^{i} \to H^i (\widetilde{W_n \Omega}_R)$ is $p$-adically complete, so it extends to the morphism $W_n \Omega_{(R, Q_R)/(\mathcal{O}_C, Q_{\mathcal{O}_C})}^{i, \cont} \to H^i (\widetilde{W_n \Omega}_R)$ and we shall also show the last map is an isomorphism. We remark that the map $W_n \Omega_{(R_{\mathcal{O}}^{\Box}, Q_{R_{\mathcal{O}}^{\Box}})/(\mathcal{O}, Q_{\mathcal{O}})}^{i} \to H^i (\widetilde{W_n \Omega}_{R^{\Box}})$ is compatible with the map $W_n \Omega_{(R, Q_R)/(\mathcal{O}_C, Q_{\mathcal{O}_C})}^{i, \cont} \to H^i (\widetilde{W_n \Omega}_R)$. We choose an \'etale morphism $R_{\mathcal{O}}^{\Box} \to R_0$ of smooth $\mathcal{O} = \overline{W(k)}$-algebras whose $p$-adic completion is $R^{\Box} \to R$. 
Then we have a morphism 
\begin{equation}
\label{O-logdRWC-naturalmap}
(W_n \Omega_{(R_{\mathcal{O}}^{\Box}, Q_{R_{\mathcal{O}}}) / (\mathcal{O}, Q_{\mathcal{O}})}^i \otimes_{W_n (R_{\mathcal{O}}^{\Box})} W_n (R_0))^{\wedge} \to 
(\mathcal{W}_n^i (D_{\infty}) \otimes_{W_n (R_{\mathcal{O}}^{\Box})} W_n (R_0))^{\wedge},
\end{equation}
where $(-)^\wedge$ is the derived $p$-adic completions by Proposition \ref{lambdani-isom}.
We note that the map $W_n (R_{\mathcal{O}}) \to W_n (R_0^{\Box})$ is flat by \cite[Theorem 10.4]{BMS18}. 

For the module of the right hand side of the map (\ref{O-logdRWC-naturalmap}), to show that the module is isomorphic to
\[
(H^i (\widetilde{W_n \Omega}_{R^{\Box}}) \otimes_{W_n (R_{\mathcal{O}}^{\Box})} W_n (R_0))^{\wedge},
\]
it suffices to show $(\mathcal{W}_n^i (D_{\infty}))^{\wedge} \cong H^i (\widetilde{W_n \Omega}_{R^{\Box}})$. In fact, we have $(L\eta_{\mu} D_{\infty})^{\wedge}$ is naturally quasi-isomorphic to $\widetilde{W_n \Omega}_{R^{\Box}}$ by commutativity of derived $p$-adic completions and $L\eta$ (\cite[Lemma 6.20]{BMS18}), $R\Gamma_{\cont}$ (\cite[Tag 0A07]{SP}). 
The isomorphism $H^i ((L\eta_{\mu} D_{\infty} / \tilde{\xi}_n)^{\wedge}) \cong (H^i (L\eta_{\mu} D_{\infty} / \tilde{\xi}_n))^{\wedge}$ is shown by the fact that each cohomology of $L\eta_{\mu} D_{\infty} / \tilde{\xi}_n$ is $p$-torsion free by Lemma \ref{WnD-decomp} and the vanishing of 
\[
R^1 \varprojlim H^i ((L\eta_{\mu} D_{\infty} / \tilde{\xi}_n) \otimes_{\mathbb{Z}}^L \mathbb{Z}/p^m ) \cong 
R^1 \varprojlim (H^i (L\eta_{\mu} D_{\infty} / \tilde{\xi}_n) \otimes_{\mathbb{Z}} \mathbb{Z}/p^m) 
\]
by \cite[Tag 061Z]{SP}.
The module of the right hand side of the map (\ref{O-logdRWC-naturalmap}) is also isomorphic to $H^i (\widetilde{W_n \Omega}_{R^{\Box}}) \widehat{\otimes}_{W_n (R^{\Box})} W_n (R)$ by the fact that $W_n (\widehat{A})$ is naturally isomorphic to the $p$-adic completion of $W_n (A)$ for a ring $A$, whose $p$-adic completion is $\widehat{A}$, which is seen by \cite[Corollary 10.2, Lemma 10.3]{BMS18}.
Hence, the module of right hand side is isomorphic to $H^i (\widetilde{W_n \Omega}_R)$ by Lemma \ref{WnLambda-et-basechange} and the flatness of $W_n (R_{\mathcal{O}}^{\Box}) \to W_n (R_0)$. 

Since the modules we consider are $p$-torsion free, we can use usual $p$-adic completions instead of derived $p$-adic completions of the modules. We can see that the left hand side of the map (\ref{O-logdRWC-naturalmap}) is isomorphic to the $p$-adic completion of $W_n \Omega_{(R_0, Q_{R_0})/(\mathcal{O}, Q_{\mathcal{O}})}^i$
by Proposition \ref{et-bc-ldRW}. We claim that the $p$-adic completion of $W_n \Omega_{(R_0, Q_{R_0})/(\mathcal{O}, Q_{\mathcal{O}})}^i$ is naturally isomorphic to 
$W_n \Omega_{(R_0, Q_{R_0})/(\mathcal{O}, Q_{\mathcal{O}})}^{i, \cont}$ and also
$W_n \Omega_{(R, Q_R)/(\mathcal{O}, Q_{\mathcal{O}})}^{i, \cont}$ by the log de Rham-Witt case of \cite[Corollary 10.10]{BMS18} for semistable formal schemes.
We explain a proof of the claim. For $m \in \mathbb{N}$, let $W_n \Omega_{((R, Q_R), p^m)/(\mathcal{O}_C, Q_{\mathcal{O}_C})}^{\ast}$ (resp.\ $W'_n \Omega_{((R, Q_R), p^m)/(\mathcal{O}_C, Q_{\mathcal{O}_C})}^{\ast}$) be the kernel 
\[
\ker(W_n \Omega_{(R, Q_R)/(\mathcal{O}_C, Q_{\mathcal{O}_C})}^{\ast} \to W_n \Omega_{(R/p^m, Q_R)/(\mathcal{O}_C, Q_{\mathcal{O}_C})}^{\ast})
\]
(resp.\ the log differential graded ideal of $W_n \Omega_{(R, Q_R)/(\mathcal{O}_C, Q_{\mathcal{O}_C})}^{\ast}$, which is generated by $W_n (p^m R)$ as a differential graded module). It follows that the quotient 
\[
W_n \Omega_{(R, Q_R)/(\mathcal{O}_C, Q_{\mathcal{O}_C})}^{\ast} / W'_n \Omega_{((R, Q_R), p^m)/(\mathcal{O}_C, Q_{\mathcal{O}_C})}^{\ast}
\]
inherits the structure of log $F$-$V$-procomplex over $(R/p^m, Q_R)/(\mathcal{O}_C, Q_{\mathcal{O}_C})$ from the structure of $W_n \Omega_{(R, Q_R)/(\mathcal{O}_C, Q_{\mathcal{O}_C})}^{\ast}$. There are the natural surjection 
\[
\pi \colon W_n \Omega_{(R, Q_R)/(\mathcal{O}_C, Q_{\mathcal{O}_C})}^{\ast} / W'_n \Omega_{((R, Q_R), p^m)/(\mathcal{O}_C, Q_{\mathcal{O}_C})}^{\ast} \to W_n \Omega_{(R/p^m, Q_R)/(\mathcal{O}_C, Q_{\mathcal{O}_C})}^{\ast}
\]
and the universal map 
\[
u \colon W_n \Omega_{(R/p^m, Q_R)/(\mathcal{O}_C, Q_{\mathcal{O}_C})}^{\ast} \to  W_n \Omega_{(R, Q_R)/(\mathcal{O}_C, Q_{\mathcal{O}_C})}^{\ast} / W'_n \Omega_{((R, Q_R), p^m)/(\mathcal{O}_C, Q_{\mathcal{O}_C})}^{\ast}
\]
of log $F$-$V$-procomplexes. The composition $\pi u$ is the identity map and the composition $u \pi$ is also the identity map since the natural surjection
\[
p \colon W_n \Omega_{(R, Q_R)/(\mathcal{O}_C, Q_{\mathcal{O}_C})}^{\ast} \to W_n \Omega_{(R, Q_R)/(\mathcal{O}_C, Q_{\mathcal{O}_C})}^{\ast} / W'_n \Omega_{((R, Q_R), p^m)/(\mathcal{O}_C, Q_{\mathcal{O}_C})}^{\ast}
\]
is the same as the composition $u \pi p$ by universality. So we have $W_n \Omega_{((R, Q_R), p^m)/(\mathcal{O}_C, Q_{\mathcal{O}_C})}^{\ast} = W'_n \Omega_{((R, Q_R), p^m)/(\mathcal{O}_C, Q_{\mathcal{O}_C})}^{\ast}$. So it suffices to show the following chains of ideals 
\[
([p^s]W_n \Omega_{(R, Q_R)/(\mathcal{O}_C, Q_{\mathcal{O}_C})}^i)_{s \geq 1} \qquad (W'_n \Omega_{((R, Q_R), p^s)/(\mathcal{O}_C, Q_{\mathcal{O}_C})}^i)_{s \geq 1}
\]
of $W_n \Omega_{(R, Q_R)/(\mathcal{O}_C, Q_{\mathcal{O}_C})}^i$ are intertwined for $i \in \mathbb{N}$. The inclusion 
\[
[p^s]W_n \Omega_{(R, Q_R)/(\mathcal{O}_C, Q_{\mathcal{O}_C})}^i \subset W'_n \Omega_{((R, Q_R), p^s)/(\mathcal{O}_C, Q_{\mathcal{O}_C})}^i
\]
is obvious. On the other hand, if we take an element $\alpha \in W'_n \Omega_{((R, Q_R), p^m)/(\mathcal{O}_C, Q_{\mathcal{O}_C})}^i$, it is written by a finite sum of elements of the form $\omega = a_0 da_1 \cdots da_i$ such that $a_j \in W_n (p^m R)$ for some $j \in [1, i]$. We have the inclusion $[p^t]W_n(R) \subset W_n (p^m R)$ for sufficiently large $t$ by \cite[Lemma 10.1]{BMS18}. Then the Leibniz rule assures $\alpha \in [p^t]W_n \Omega_{(R, Q_R)/(\mathcal{O}_C, Q_{\mathcal{O}_C})}^i$ and the inclusion 
\[
[p^t]W_n \Omega_{(R, Q_R)/(\mathcal{O}_C, Q_{\mathcal{O}_C})}^i \supset W'_n \Omega_{((R, Q_R), p^s)/(\mathcal{O}_C, Q_{\mathcal{O}_C})}^i.
\]
Consequently, the claim is proved by the equation 
\[
W_n \Omega_{(R/p^m, Q_R)/(\mathcal{O}_C, Q_{\mathcal{O}_C})}^{\ast} = W_n \Omega_{(R/p^m, Q_R)/(\mathcal{O}_C / p^m, Q_{\mathcal{O}_C})}^{\ast}.
\]

The BKF twist part $\{-i\}$ is reconstructed by the Bockstein differential for  
\[
0 \to \tilde{\xi}_n A_{\inf} / \tilde{\xi}_n^2 A_{\inf} \to A_{\inf} / \tilde{\xi}_n^2 A_{\inf} \to A_{\inf} / \tilde{\xi}_n A_{\inf} \to 0
\]
and the natural isomorphism 
\[
\bigotimes^i (\tilde{\xi}_n A_{\inf} / \tilde{\xi}_n^2 A_{\inf}) \cong \tilde{\xi}_n^i A_{\inf} / \tilde{\xi}_n^{i+1} A_{\inf} \cong \mathcal{O}_C \{i\}.
\]
For compatibility with Theorem \ref{CKisom} in the case $n = 1$, it suffices to check that the basis elements $d\log t_1,\ldots,d\log t_d$ in $\Omega_{(R, Q_R)/(\mathcal{O}_C, Q_{\mathcal{O}_C})}^{1, \mathrm{cont}}$ in the case of $R = R^{\Box}$ are compatible with Theorem \ref{CKisom} because those generate $\Omega_{(R, Q_R)/(\mathcal{O}_C, Q_{\mathcal{O}_C})}^{i, \cont} \{ -i\}$ as a projective limit of log differential graded $R/\mathcal{O}_C$-algebras. It is seen by the argument of the proof of Lemma \ref{lambdanast} (2).
\end{proof}


\end{document}